\documentclass[11pt,letterpaper,twoside,reqno,nosumlimits]{amsart}

\allowdisplaybreaks

\usepackage[usenames,dvipsnames]{xcolor}
\usepackage{fancyhdr}
\usepackage{amsmath,amsfonts,amsbsy,amsgen,amscd,amssymb,amsthm}
\usepackage{subfig}
\usepackage{url}

\usepackage{mathtools}

\mathtoolsset{showonlyrefs}

\usepackage[font=small,margin=0.25in,labelfont={sc},labelsep={colon}]{caption}

\usepackage{tikz}
\usepackage{microtype}
\usepackage{enumitem}

\definecolor{dark-gray}{gray}{0.3}
\definecolor{dkgray}{rgb}{.4,.4,.4}
\definecolor{dkblue}{rgb}{0,0,.5}
\definecolor{medblue}{rgb}{0,0,.75}
\definecolor{rust}{rgb}{0.5,0.1,0.1}

\usepackage[colorlinks=true]{hyperref}

\hypersetup{urlcolor=rust}
\hypersetup{citecolor=dkblue}
\hypersetup{linkcolor=dkblue}

\usepackage{setspace}

\usepackage{graphicx}
\usepackage{booktabs,longtable,tabu} 
\usepackage{multirow} 
\usepackage{float}

\usepackage[full]{textcomp}

\usepackage[scaled=.98,sups,osf]{XCharter}\usepackage[scaled=1.04,varqu,varl]{inconsolata}\usepackage[type1]{cabin}\usepackage[charter,vvarbb,scaled=1.07]{newtxmath}
\usepackage[cal=boondoxo]{mathalfa}
\usepackage[T1]{fontenc}

\usepackage{bm} 

\graphicspath{{figures/}}

\newtheorem{theorem}{Theorem}[section]
\newtheorem{lemma}[theorem]{Lemma}

\newtheorem{proposition}[theorem]{Proposition}
\newtheorem{fact}[theorem]{Fact}

\newtheorem{corollary}[theorem]{Corollary}

\theoremstyle{definition}

\newtheorem{example}[theorem]{Example}
\newtheorem{remark}[theorem]{Remark}

\numberwithin{equation}{section} 
\numberwithin{figure}{section}
\numberwithin{table}{section}

\floatstyle{plaintop}
\newfloat{recipe}{thp}{lor}
\floatname{recipe}{Recipe}
\numberwithin{recipe}{section}

\providecommand{\mathbold}[1]{\bm{#1}}

\renewcommand{\phi}{\varphi}

\newcommand{\eps}{\varepsilon}

\newcommand{\econst}{\mathrm{e}}
\newcommand{\iunit}{\mathrm{i}}

\newcommand{\Id}{\mathbf{I}}

\providecommand{\mathbbm}{\mathbb} 
\newcommand{\R}{\mathbbm{R}}

\newcommand{\CC}{\mathbbm{C}}
\newcommand{\N}{\mathbbm{N}}

\newcommand{\abs}[1]{\left\vert {#1} \right\vert}

\newcommand{\sgn}[1]{\operatorname{sgn}{#1}}
\newcommand{\real}{\operatorname{Re}}
\newcommand{\imag}{\operatorname{Im}}

\newcommand{\diff}[1]{\mathrm{d}{#1}}
\newcommand{\idiff}[1]{\, \diff{#1}}

\newcommand{\Prob}[1]{\mathbbm{P}\left\{{#1}\right\}}

\newcommand{\Expect}{\operatorname{\mathbbm{E}}}

\DeclareMathOperator{\Var}{Var}

\DeclareMathOperator{\mVar}{\mathbold{\mathrm{Var}}}

\newcommand{\vct}[1]{\mathbold{#1}}
\newcommand{\mtx}[1]{\mathbold{#1}}

\newcommand{\trsp}{\mathsf{T}}

\newcommand{\trace}{\operatorname{tr}}
\newcommand{\ntr}{\operatorname{\bar{\trace}}}

\newcommand{\psdle}{\preccurlyeq}
\newcommand{\psdge}{\succcurlyeq}

\newcommand{\ip}[2]{\left\langle {#1},\ {#2} \right\rangle}

\newcommand{\norm}[1]{\left\Vert {#1} \right\Vert}

\newcommand{\triplenorm}[1]{{\left\vert\kern-0.25ex\left\vert\kern-0.25ex\left\vert #1
    \right\vert\kern-0.25ex\right\vert\kern-0.25ex\right\vert}}

\newcommand{\mL}{\mathcal{L}}

\newcommand{\E}{\Expect}

\usepackage[numbers]{natbib}
 
\begin{document}

\title[Matrix concentration via semigroup methods]{Nonlinear matrix concentration via semigroup methods}
\author[D. Huang and J. A. Tropp]{De Huang$^*$ and Joel A. Tropp$^\dagger$}
\thanks{$^*$California Institute of Technology, USA. E-mail: dhuang@caltech.edu}
\thanks{$^\dagger$California Institute of Technology, USA. E-mail: jtropp@cms.caltech.edu}

\begin{abstract}
Matrix concentration inequalities provide information about the
probability that a random matrix is close to its expectation
with respect to the $\ell_2$ operator norm.
This paper uses semigroup methods to derive sharp
nonlinear matrix inequalities.
In particular, it is shown that the classic Bakry--{\'E}mery
curvature criterion implies subgaussian concentration
for ``matrix Lipschitz'' functions.
This argument circumvents the need to develop a matrix version
of the log-Sobolev inequality, a technical obstacle that has
blocked previous attempts to derive matrix concentration inequalities
in this setting.
The approach unifies and extends much of the previous work
on matrix concentration.  When applied to a product measure,
the theory reproduces the matrix Efron--Stein inequalities
due to Paulin et al.  It also handles matrix-valued functions
on a Riemannian manifold with uniformly positive Ricci curvature.
\end{abstract}


\maketitle

\section{Motivation}

Matrix concentration inequalities describe the probability that a random matrix is close to its expectation,
with deviations measured in the $\ell_2$ operator norm.
The basic models---sums of independent random matrices
and matrix-valued martingales---have been studied extensively,
and they admit a wide spectrum of applications~\cite{tropp2015introduction}.
Nevertheless, we lack a complete understanding of more general
random matrix models.  The purpose of this paper is to develop
a systematic approach for deriving ``nonlinear'' matrix
concentration inequalities.

In the scalar setting, functional inequalities
offer a powerful framework for studying nonlinear concentration.
For example, consider a real-valued Lipschitz function $f(Z)$ of
a real random variable $Z$ with distribution $\mu$.
If the measure $\mu$ satisfies a Poincar{\'e} inequality,
then the variance of $f(Z)$ is controlled by the squared
Lipschitz constant of $f$.  If the measure satisfies
a log-Sobolev inequality, then $f(Z)$
enjoys subgaussian concentration
on the scale of the Lipschitz constant.

Now, suppose that we can construct a semigroup,
acting on real-valued functions, with stationary distribution $\mu$.
Functional inequalities for the measure $\mu$
are intimately related to the convergence of the semigroup.
In particular, the measure admits a Poincar{\'e} inequality
if and only if the semigroup rapidly tends to
equilibrium (in the sense that the variance is exponentially ergodic).
Meanwhile, log-Sobolev inequalities are associated with
finer types of ergodicity.

In recent years, researchers have attempted to use functional
inequalities and semigroup tools to prove matrix concentration
results.  So far, these arguments have met some success,
but they are not strong enough to reproduce the results that are
already available for the simplest random matrix models.
The main obstacle has been the lack of a suitable extension of
the log-Sobolev inequality to the matrix setting.
See Section~\ref{sec:concentration_history} for an account of prior work.

The purpose of this paper is to advance the theory of
semigroups acting on matrix-valued functions and to
apply these methods to obtain matrix concentration
inequalities for nonlinear random matrix models.
To do so, we argue that the classic Bakry--{\'E}mery
curvature criterion for a semigroup acting on real-valued
functions ensures that an associated matrix semigroup
also satisfies a curvature condition.  This property
further implies local ergodicity of the matrix semigroup,
which we can use to prove strong bounds on the trace
moments of nonlinear random matrix models.

The power of this approach is that the Bakry--{\'E}mery
condition has already been verified for a large number
of semigroups.  We can exploit these results to identify
many new settings where matrix concentration is in force.
This program entirely evades the question about the proper way
to extend log-Sobolev inequalities to matrices.

Our approach reproduces many existing results from
the theory of matrix concentration, such as the matrix
Efron--Stein inequalities~\cite{paulin2016efron}.
Among other new results, we can achieve subgaussian concentration
for a matrix-valued ``Lipschitz'' function on a positively curved
Riemannian manifold.  Here is a simplified formulation of this fact.

\begin{theorem}[Euclidean submanifold: Subgaussian concentration]
\label{thm:riemann-simple}
Let $M$ be a compact $n$-dimensional Riemannian submanifold
of a Euclidean space, and let $\mu$ be the uniform measure on $M$.
Suppose that the eigenvalues of the Ricci curvature tensor of $M$ are uniformly
bounded below by $\rho$.  Let $\mtx{f} : M \to \mathbb{H}_d$ be a differentiable function.
For all $t \geq 0$, $$
\mathbb{P}_{\mu} \big\{ \norm{ \smash{\mtx{f} - \E_{\mu} \mtx{f}} } \geq t \big\}
	\leq 2d \, \exp\left( \frac{-\rho t^2}{2 v_{\mtx{f}}} \right)
	\quad\text{where}\quad
	v_{\mtx{f}} := \sup\nolimits_{x \in M} \norm{ \sum_{i=1}^n (\partial_i \mtx{f}(x))^2 }.
$$
Furthermore, for $q = 2$ and $q\geq 3$,
$$
\left[ \Expect_{\mu} \trace ( \mtx{f} - \Expect_{\mu} \mtx{f} )^q \right]^{1/q}
	\leq \rho^{-1/2} \sqrt{q - 1} \left[
	\Expect_{\mu} \trace \left( \sum_{i=1}^n (\partial_i \mtx{f})^2 \right)^{q/2} \right]^{1/q}.
$$
The real-linear space $\mathbb{H}_d$ contains all $d \times d$ Hermitian matrices,
and $\norm{\cdot}$ is the $\ell_2$ operator norm.  The operators $\partial_i$
compute partial derivatives in local (normal) coordinates.
\end{theorem}

Theorem~\ref{thm:riemann-simple} follows from abstract
concentration inequalities (Theorem~\ref{thm:polynomial_moment} and Theorem~\ref{thm:exponential_concentration})
and the classic fact that the Brownian motion on a positively curved
Riemannian manifold satisfies the Bakry--{\'E}mery criterion~\cite[Sec.~1.16]{bakry2013analysis}.
See Section~\ref{sec:concentration_results_Riemannian} for details.

Particular settings where the theorem is valid
include the unit Euclidean sphere and the special orthogonal group. The variance proxy $v_{\mtx{f}}$ is analogous with the squared Lipschitz
constant that appears in scalar concentration results.  We emphasize that
$\partial_i \mtx{f}$ is an Hermitian matrix, and the variance proxy involves
a sum of the matrix squares.  Thus, the ``Lipschitz constant'' is tailored to
the matrix setting.

As a concrete example, consider the $n$-dimensional sphere $\mathbb{S}^n \subset \R^{n+1}$,
with uniform measure $\sigma_n$ and  curvature $\rho = n - 1$.
Let $\mtx{A}_1, \dots, \mtx{A}_{n+1} \in \mathbb{H}_d$
be fixed matrices.  Construct the random matrix
$$
\mtx{f}(x) = \sum_{i=1}^{n+1} x_i \mtx{A}_i
\quad\text{where $x \sim \sigma_n$.}
$$
By symmetry, $\E_{\sigma_n} \mtx{f} = \mtx{0}$.
Moreover, the variance proxy
$v_{\mtx{f}} \leq \norm{ \sum_{i=1}^{n+1} \mtx{A}_i^2 }$.
Thus, Theorem~\ref{thm:riemann-simple} delivers the bound
$$
\mathbb{P}_{\sigma_n} \big\{ \norm{ \mtx{f} } \geq t \big\}
	\leq 2d \exp\left( \frac{-(n-1) t^2}{2 v_{\mtx{f}}} \right).
$$
See Section~\ref{sec:riemann-exp} for more instances of Theorem~\ref{thm:riemann-simple} in action.

\begin{remark}[Noncommutative moment inequalities]
After this paper was complete, we learned that Junge \& Zeng~\cite{junge2015noncommutative}
have developed a similar method, based on a noncommutative Bakry--Emery criterion,
to obtain moment inequalities in the setting of a von Neumann algebra equipped
with a noncommutative diffusion semigroup.  Their results are not fully comparable
with ours, so we will elaborate on the relationship as we go along.
\end{remark}

\section{Matrix Markov semigroups: Foundations}
\label{sec:matrix_Markov_semigroups}

To start, we develop some basic facts about an important class of Markov semigroups that acts on matrix-valued functions.  Given a Markov process, we define the associated matrix Markov semigroup and its infinitesimal generator.  Then we construct the matrix carr\'e du champ operator and the Dirichlet form.  Afterward, we outline the connection between convergence properties of the semigroup and Poincar{\'e} inequalities.  Parts of our treatment are adapted from~\cite{cheng2017exponential,ABY20:Matrix-Poincare}, but some elements appear to be new.

\subsection{Notation}

Let $\mathbb{M}_d$ be the algebra of all $d \times d$ complex matrices.
The real-linear subspace $\mathbb{H}_d$ contains all Hermitian matrices, and $\mathbb{H}_d^+$ is the cone of all positive-semidefinite matrices.  Matrices are written in boldface. In particular, $\Id_d$ is the $d$-dimensional identity matrix, while $\mtx{f}$, $\mtx{g}$ and $\mtx{h}$ refer to matrix-valued functions.  We use the symbol $\preccurlyeq$ for the semidefinite partial order on Hermitian matrices:
For matrices $\mtx{A}, \mtx{B} \in \mathbb{H}_d$, the inequality $\mtx{A}\preccurlyeq \mtx{B}$ means that $\mtx{B}-\mtx{A}\in \mathbb{H}_d^+$.

For a matrix $\mtx{A}\in\mathbb{M}_d$, we write $\|\mtx{A}\|$ for the $\ell_2$ operator norm,  $\|\mtx{A}\|_\mathrm{HS}$ for the Hilbert--Schmidt norm, and $\trace \mtx{A}$ for the trace.  The normalized trace is defined as 
$\ntr \mtx{A}  := d^{-1} \trace \mtx{A}$. Nonlinear functions bind before the trace.
Given a scalar function $\varphi:\mathbb{R}\rightarrow \mathbb{R}$, we construct
the \emph{standard matrix function} $\phi : \mathbb{H}_d \to \mathbb{H}_d$ using the eigenvalue decomposition:
\[\varphi(\mtx{A}) := \sum_{i=1}^d \varphi(\lambda_i) \, \vct{u}_i\vct{u}_i^* \quad \text{where}\quad \mtx{A} = \sum_{i=1}^d \lambda_i \,\vct{u}_i\vct{u}_i^*.  \]
We constantly rely on basic tools from matrix theory; see~\cite{carlen2010trace}.

Let $\Omega$ be a Polish space equipped with a probability measure $\mu$.  Define $\E_\mu$ and $\mathrm{Var}_{\mu}$ to be the expectation and variance of a real-valued function with respect to the measure $\mu$.  When applied to a random matrix, $\E_\mu$ computes the entrywise expectation.
Nonlinear functions bind before the expectation.

\subsection{Markov semigroups acting on matrices}

This paper focuses on a special class of Markov semigroups acting on matrices.
In this model, a classical Markov process drives the evolution of a matrix-valued function.
Remark~\ref{rem:nc-semigroup} mentions some generalizations.

Suppose that $(Z_t)_{t\geq0} \subset \Omega$ is a time-homogeneous Markov process
on the state space $\Omega$ with stationary measure $\mu$.
For each matrix dimension $d \in \N$, we can construct a Markov semigroup $(P_t)_{t\geq0}$
that acts on a (bounded) measurable matrix-valued function
$\mtx{f} : \Omega\rightarrow\mathbb{H}_d$ according to  
\begin{equation} \label{eqn:semigroup}
(P_t\mtx{f})(z) := \E[\mtx{f}(Z_t)\,|\,Z_0 = z]\quad \text{for all $t\geq 0$ and all $z\in \Omega$}.
\end{equation}
The semigroup property $P_{t+s} = P_{t}P_{s} = P_{s}P_{t}$ holds for all $s, t \geq 0$
because $(Z_t)_{t\geq 0}$ is a homogeneous Markov process.

Note that the operator $P_0$ is the identity map: $P_0 \mtx{f} = \mtx{f}$.
For a fixed $\mtx{A} \in \mathbb{H}_d$, regarded as a constant function on $\Omega$,
the semigroup also acts as the identity: $P_t \mtx{A} = \mtx{A}$ for all $t\geq0$.  Furthermore,
$\Expect_{\mu}[ P_t \mtx{f} ] = \Expect_{\mu}[ \mtx{f} ]$ because $Z_0 \sim \mu$ implies
that $Z_t \sim \mu$ for all $t \geq 0$.  We use these facts without comment.

Although~\eqref{eqn:semigroup} defines a family of semigroups indexed by the matrix dimension $d$,
we will abuse terminology and speak of this collection as if it were as single semigroup.
A major theme of this paper is that facts about the action of the semigroup~\eqref{eqn:semigroup}
on real-valued functions ($d = 1$) imply parallel facts about the action on matrix-valued functions
($d \in \N$).

\begin{remark}[Noncommutative semigroups] \label{rem:nc-semigroup}
There is a very general class of noncommutative semigroups acting
on a von Neumann algebra
where the action is determined by a family of completely
positive unital maps~\cite{junge2015noncommutative}.  This framework includes~\eqref{eqn:semigroup}
as a special case; it covers quantum semigroups~\cite{cheng2017exponential}
acting on $\mathbb{H}_d$ with a fixed matrix dimension $d$;
it also includes more exotic examples.  We will not study
these models, but we will discuss the relationship
between our results and prior work.
\end{remark}

\subsection{Ergodicity and reversibility}

We say that the semigroup $(P_t)_{t\geq0}$ defined in~\eqref{eqn:semigroup}
is \emph{ergodic} if
\begin{equation*}\label{eqn:ergodicity_scalar}
P_tf \rightarrow \E_\mu f\quad  \text{in $L_2(\mu)$}\quad  \text{as}\quad t\rightarrow+\infty \quad \text{for all $f:\Omega\rightarrow\mathbb{R}$}.
\end{equation*}
Furthermore, $(P_t)_{t\geq0}$ is \emph{reversible} if each operator $P_t$
is a \emph{symmetric} operator
on $L_2(\mu)$.  That is, 
\begin{equation}\label{eqn:reversibility_scalar}
\E_\mu [(P_tf) \, g] = \E_\mu [f \, (P_tg)] \quad \text{for all $t\geq0$ and all $f,g:\Omega\rightarrow\mathbb{R}$}.
\end{equation}
Note that these definitions involve only real-valued functions $(d = 1)$. 

In parallel, we say that the Markov process $(Z_t)_{t\geq0}$ is reversible (resp.~ergodic) if the associated Markov semigroup $(P_t)_{t\geq0}$ is reversible (resp.~ergodic).  The reversibility of the process $(Z_t)_{t\geq0}$ implies that, when $Z_0 \sim \mu$, the pair $(Z_t, Z_0)$ is \emph{exchangeable} for all $t \geq 0$.  That is, $(Z_t,Z_0)$ and $(Z_0,Z_t)$ follow the same distribution for all $t\geq0$.

Our matrix concentration results require ergodicity and reversibility
of the semigroup action on matrix-valued functions.
These properties are actually a consequence of the
analogous properties for real-valued functions.
Evidently, the ergodicity of $(P_t)_{t\geq0}$ is equivalent with
the statement 
\begin{equation}\label{eqn:ergodicity}
P_t\mtx{f} \rightarrow \E_\mu \mtx{f}\quad  \text{in $L_2(\mu)$}\quad \text{as}\quad t\rightarrow+\infty \quad \text{for all $\mtx{f}:\Omega\rightarrow\mathbb{H}_d$ and each $d \in \N$.}
\end{equation}
Note that the $L_2(\mu)$ convergence in the matrix setting means $\lim_{t\rightarrow \infty}\E_\mu(P_t\mtx{f}-\E_\mu \mtx{f})^2= \mtx{0}$ , which is readily implied by the $L_2(\mu)$ convergence of all entries of $P_t\mtx{f}-\E_\mu \mtx{f}$.
As for reversibility, we have the following result.

\begin{proposition}[Reversibility] \label{prop:reversibility}
Let $(P_t)_{t \geq 0}$ be the family of semigroups defined in~\eqref{eqn:semigroup}.
The following are equivalent.

\begin{enumerate}
\item	The semigroup acting on real-valued functions is symmetric, as in~\eqref{eqn:reversibility_scalar}.

\item	The semigroup acting on matrix-valued functions is symmetric.  That is, for each $d \in \N$,
\begin{equation}\label{eqn:reversibility_1}
\E_\mu [(P_t\mtx{f}) \, \mtx{g}] = \E_\mu [\mtx{f} \, (P_t\mtx{g})] \quad \text{for all $t\geq0$ and all $\mtx{f},\mtx{g}:\Omega\rightarrow\mathbb{H}_d$}.
\end{equation}
\end{enumerate}
\end{proposition}

\noindent
Let us emphasize that~\eqref{eqn:reversibility_1} now involves matrix products.
The proof of Proposition~\ref{prop:reversibility} appears below in
Section~\ref{sec:reversibility-pf}.

\subsection{Convexity}

Given a convex function $\Phi:\mathbb{H}_d\rightarrow\mathbb{R}$ that is bounded below, the semigroup satisfies
a Jensen inequality of the form
\begin{equation*}\label{eqn:semigroup_Jensen_1}
\Phi(P_t\mtx{f}(z)) = \Phi(\E[\mtx{f}(Z_t)\,|\,Z_0 = z]) \leq \E[\Phi(\mtx{f}(Z_t))\,|\,Z_0 = z]\quad \text{for all $z\in \Omega$}.
\end{equation*}
This is an easy consequence of the definition~\eqref{eqn:semigroup}.  In particular, 
\begin{equation}\label{eqn:semigroup_Jensen_2}
\E_\mu \Phi(P_t\mtx{f})  \leq \E_{Z\sim\mu} \E[\Phi(\mtx{f}(Z_t))\,|\,Z_0 = Z] = \E_{Z_0\sim\mu}[\Phi(\mtx{f}(Z_t))] = \E_\mu \Phi(\mtx{f}) .
\end{equation}
A typical choice of $\Phi$ is the trace function $\trace \phi$, where $\phi : \mathbb{H}_d \to \mathbb{H}_d$ is a standard matrix function. 

\subsection{Infinitesimal generator}

The \emph{infinitesimal generator} $\mL$ of the semigroup~\eqref{eqn:semigroup}
acts on a (nice) measurable function
$\mtx{f}:\Omega\rightarrow\mathbb{H}_d$ via the formula
\begin{equation}\label{eqn:Markov_generator}
(\mL\mtx{f})(z) := \lim_{t\downarrow 0}\frac{(P_t\mtx{f})(z) -\mtx{f}(z)}{t}
\quad\text{for all $z \in \Omega$.}
\end{equation}
Because $(P_t)_{t \geq 0}$ is a semigroup, it follows immediately that 
\begin{equation}\label{eqn:derivative_relation}
\frac{\diff{} }{\diff t}P_t = \mL P_t = P_t\mL\quad \text{for all}\ t\geq0.
\end{equation}
The null space of $\mL$ contains all constant functions:
$\mL \mtx{A} = \mtx{0}$ for each fixed $\mtx{A} \in \mathbb{H}_d$.
Moreover,
\begin{equation}\label{eqn:mean_zero}
\E_\mu[\mL \mtx{f} ] 	= \mtx{0}\quad \text{for all $\mtx{f}:\Omega\rightarrow\mathbb{H}_d$}.
\end{equation}
That is, the infinitesimal generator converts an arbitrary function into a zero-mean function.

We say that the infinitesimal generator $\mL$ is {symmetric} on $L_2(\mu)$ when
its action on real-valued functions is symmetric:
\[
\E_\mu [(\mL f)\,g] = \E_\mu [f \, (\mL g)] \quad \text{for all $f,g:\Omega\rightarrow\mathbb{R}$}.
\]
The generator $\mL$ is symmetric if and only if the semigroup $(P_t)_{t \geq 0}$ is symmetric (i.e., reversible).
In this case, the action of $\mL$ on matrix-valued functions is also symmetric: 
\begin{equation}\label{eqn:reversibility_2}
\E_\mu [(\mL \mtx{f})\,\mtx{g}] = \E_\mu [\mtx{f} \, (\mL\mtx{g})] \quad \text{for all $\mtx{f},\mtx{g}:\Omega\rightarrow\mathbb{H}_d$}.
\end{equation}
This point follows from Proposition~\ref{prop:reversibility}.

As we have alluded,
the limit in \eqref{eqn:Markov_generator} need not exist for all functions.
The set of functions $\mtx{f}:\Omega\rightarrow\mathbb{H}_d$ for which $\mL\mtx{f}$
is defined $\mu$-almost everywhere is called the \emph{domain} 
of the generator.
It is highly technical, but usually unimportant, to characterize the domain of the generator
and related operators.

For our purposes, we may restrict attention to an unspecified
algebra of \emph{suitable} functions (say, smooth and compactly supported)
where all operations involving limits, derivatives, and integrals
are justified.  By approximation, we can extend the main results
to the entire class of functions where the statements make sense.
We refer the reader to the monograph~\cite{bakry2013analysis}
for an extensive discussion about how to make these arguments
airtight.

\subsection{Carr\'e du champ operator and Dirichlet form}

For each $d \in \N$, given the infinitesimal generator $\mL$,
the matrix \textit{carr\'e du champ operator} is the bilinear form
\begin{equation}\label{eqn:definition_Gamma}
\Gamma(\mtx{f},\mtx{g}) := \frac{1}{2}\left[ \mL(\mtx{f}\mtx{g}) - \mtx{f}\mL(\mtx{g}) - \mL(\mtx{f})\mtx{g} \right] \in \mathbb{M}_d \quad \text{for all suitable $\mtx{f},\mtx{g} : \Omega \to \mathbb{H}_d$}.
\end{equation}
The matrix \emph{Dirichlet form} is the bilinear form obtained by integrating the carr{\'e} du champ:
\begin{equation} \label{eqn:Dirichlet_form}
\mathcal{E}(\mtx{f},\mtx{g}) := \E_\mu \Gamma(\mtx{f},\mtx{g}) \in \mathbb{M}_d \quad \text{for all suitable $\mtx{f},\mtx{g} : \Omega \to \mathbb{H}_d$}.
\end{equation}
We abbreviate the associated quadratic forms as $\Gamma(\mtx{f}):=\Gamma(\mtx{f},\mtx{f})$ and $\mathcal{E}(\mtx{f}):=\mathcal{E}(\mtx{f},\mtx{f})$.  Proposition~\ref{prop:Gamma_property}
states that both these quadratic forms 
are positive operators in the sense that they take values in the cone of
positive-semidefinite Hermitian matrices.
In many instances, the carr\'e du champ $\Gamma(\mtx{f})$
has a natural interpretation as the squared magnitude of the derivative of $\mtx{f}$,
while the Dirichlet form $\mathcal{E}(\mtx{f})$
reflects the total energy of the function $\mtx{f}$. 

Using~\eqref{eqn:mean_zero}, we can rewrite the Dirichlet form as
\begin{align}\label{eqn:Dirichlet_expression_1}
\mathcal{E}(\mtx{f},\mtx{g}) = \E_\mu\Gamma(\mtx{f},\mtx{g}) = -\frac{1}{2} \E_\mu \left[\mtx{f}\mL(\mtx{g}) + \mL(\mtx{f})
\mtx{g}\right] 
\end{align}
When the semigroup $(P_t)_{t \geq 0}$ is reversible,
then~\eqref{eqn:reversibility_2} and~\eqref{eqn:Dirichlet_expression_1} indicate that
\begin{align}\label{eqn:Dirichlet_expression_2}
\mathcal{E}(\mtx{f},\mtx{g}) = -\E_\mu [\mtx{f}\mL(\mtx{g})] = -\E_\mu [\mL(\mtx{f})\mtx{g}].
\end{align}
These alternative expressions are very useful for calculations.

\subsection{The matrix Poincar\'e inequality}

For each function $\mtx{f}:\Omega\rightarrow\mathbb{H}_d$,
the \textit{matrix variance} with respect to the distribution $\mu$ is defined as
\begin{equation} \label{eqn:matrix_variance}
\mVar_\mu[\mtx{f}] := \E_\mu\big[(\mtx{f}-\E_\mu\mtx{f})^2\big] = \E_\mu[\mtx{f}^2] - (\E_\mu\mtx{f})^2\in \mathbb{H}_d^+.
\end{equation}
We say that the Markov process satisfies a \textit{matrix Poincar\'e inequality} with constant $\alpha>0$ if 
\begin{equation}\label{eqn:matrix_Poincare}
\mVar_\mu(\mtx{f})\preccurlyeq \alpha \cdot \mathcal{E}(\mtx{f})\quad \text{for all suitable $\mtx{f} : \Omega \to \mathbb{H}_d$}.
\end{equation}
This definition seems to be due to Chen et al.~\cite{cheng2017exponential};
see also Aoun et al.~\cite{ABY20:Matrix-Poincare}.

When the matrix dimension $d = 1$, the inequality~\eqref{eqn:matrix_Poincare} reduces to the usual scalar Poincar{\'e} inequality for the semigroup.  For the semigroup~\eqref{eqn:semigroup},
the scalar Poincar{\'e} inequality ($d = 1$) already implies the
matrix Poincar{\'e} inequality (for all $d \in \N$).  Therefore, to check the validity of~\eqref{eqn:matrix_Poincare},
it suffices to consider real-valued functions.

\begin{proposition}[Poincar{\'e} inequalities: Equivalence] \label{prop:poincare_equiv}
For each $d \in \N$, let $(P_t)_{t\geq 0}$ be the semigroup defined in~\eqref{eqn:semigroup}.
The following are equivalent:

\begin{enumerate}
\item \label{Poincare_inequality_scalar}
\textbf{Scalar Poincar\'e inequality.}
$\Var_\mu[f]\leq \alpha \cdot \mathcal{E}(f)$ for all suitable $f:\Omega\to \mathbb{R}$.
\item \label{Poincare_inequality_matrix}
\textbf{Matrix Poincar\'e inequality.}
$\mVar_\mu[\mtx{f}]\preccurlyeq \alpha \cdot \mathcal{E}(\mtx{f})$ for all suitable $\mtx{f}: \Omega \to \mathbb{H}_d$ and all $d \in \N$. 
\end{enumerate}
\end{proposition}

\noindent
The proof of Proposition~\ref{prop:poincare_equiv} appears in Section~\ref{sec:scalar_matrix}.
We are grateful to Ramon van Handel for this observation.

\subsection{Poincar{\'e} inequalities and ergodicity}

As in the scalar case, the matrix Poincar{\'e} inequality~\eqref{eqn:matrix_Poincare} is a powerful tool for understanding
the action of a semigroup on matrix-valued functions.  Assuming ergodicity, the Poincar{\'e}
inequality is equivalent with the exponential convergence of the Markov semigroup $(P_t)_{t \geq 0}$
to the expectation operator $\E_{\mu}$.  The constant $\alpha$ determines the rate
of convergence.  The following result makes this principle precise.

\begin{proposition}[Poincar\'e inequality: Consequences] \label{prop:matrix_poincare} 
Consider a Markov semigroup $(P_t)_{t \geq 0}$
with stationary measure $\mu$ acting on suitable functions $\mtx{f} : \Omega \to \mathbb{H}_d$
for a fixed $d \in \N$, as defined in~\eqref{eqn:semigroup}. 
The following are equivalent:
\begin{enumerate}
\item \label{Poincare_inequality}
\textbf{Poincar\'e inequality.}
$\mVar_\mu[\mtx{f}]\preccurlyeq \alpha \cdot \mathcal{E}(\mtx{f})$ for all suitable $\mtx{f}: \Omega \to \mathbb{H}_d$.
\item  \label{variance_convergence}
\textbf{Exponential ergodicity of variance.}
$\mVar_\mu[P_t\mtx{f}]\preccurlyeq \econst^{-2t/\alpha} \cdot \mVar_\mu[\mtx{f}]$ for all $t\geq 0$ and for all suitable $\mtx{f}:\Omega \to \mathbb{H}_d$.
\end{enumerate}
Moreover, if the semigroup $(P_t)_{t\geq 0}$ is reversible and ergodic, then the statements above are also equivalent to the following:
\begin{enumerate}[resume]
\item \label{energy_convergence}
\textbf{Exponential ergodicity of energy.}
$\mathcal{E}(P_t\mtx{f})\preccurlyeq \econst^{-2t/\alpha} \cdot \mathcal{E}(\mtx{f})$ for all $t \geq 0$ and for all suitable $\mtx{f}:\Omega \to \mathbb{H}_d$.
\end{enumerate} 
\end{proposition}

\noindent
Section~\ref{sec:equivalence_Poincare}
contains the proof of Proposition~\ref{prop:matrix_poincare},
which is essentially the same as in the scalar case~\cite[Theorem 2.18]{van550probability}.

\begin{remark}[Quantum semigroups]
Proposition~\ref{prop:matrix_poincare} only concerns the action of
a semigroup on matrices of fixed dimension $d$.  As such, the result
can be adapted to quantum Markov semigroups.
A partial version of the result for this general setting
already appears in \cite[Remark IV.2]{cheng2017exponential}.
\end{remark}

\subsection{Iterated carr{\'e} du champ operator}

To better understand how quickly a Markov semigroup converges to equilibrium, it is valuable to consider the \textit{iterated carr\'e du champ operator}.  In the matrix setting, this operator is defined as 
\begin{equation}\label{eqn:definition_Gamma2}
\Gamma_2(\mtx{f},\mtx{g}) := \frac{1}{2}\left[ \mL\Gamma(\mtx{f},\mtx{g}) - \Gamma(\mtx{f},\mL(\mtx{g})) - \Gamma(\mL(\mtx{f}),\mtx{g}) \right] \in \mathbb{M}_d \quad  \text{for all suitable $\mtx{f},\mtx{g} : \Omega \to \mathbb{H}_d$}.
\end{equation}
As with the carr\'e du champ, we abbreviate the quadratic form $\Gamma_2(\mtx{f}) := \Gamma_2(\mtx{f},\mtx{f})$.  We remark that this quadratic form is not necessarily a positive operator.
Rather, $\Gamma_2(\mtx{f})$ reflects the ``magnitude'' of the squared Hessian
of $\mtx{f}$ plus a correction factor that reflects the ``curvature'' of the matrix semigroup.

When the underlying Markov semigroup $(P_t)_{t \geq 0}$ is reversible, it holds that 
\[\E_\mu \Gamma_2(\mtx{f},\mtx{g}) = \E_\mu\left[\mL(\mtx{f}) \, \mL(\mtx{g})\right]\quad \text{for all suitable $\mtx{f},\mtx{g} : \Omega \to \mathbb{H}_d$}.\]
Thus, for a reversible semigroup, the average value $\Expect_{\mu} \Gamma_2(\mtx{f})$ is a positive-semidefinite matrix.

\subsection{Bakry--{\'E}mery criterion}
\label{sec:local_matrix_Poincare_inequality}

When the iterated carr{\'e} du champ is comparable with the carr{\'e} du champ, we can obtain more information about the convergence of the Markov semigroup.
We say the semigroup satisfies the \textit{matrix Bakry--\'Emery criterion} with constant $c>0$ if 
\begin{equation}\label{Bakry-Emery}
\Gamma(\mtx{f}) \psdle c \cdot \Gamma_2(\mtx{f}) \quad \text{for all suitable $\mtx{f} : \Omega \to \mathbb{H}_d$}. 
\end{equation}
Since $\Gamma(\mtx{f})$ and $\Gamma_2(\mtx{f})$ are functions, one interprets this condition
as a pointwise inequality that holds $\mu$-almost everywhere in $\Omega$.  It reflects
uniform positive curvature of the semigroup.

When the matrix dimension $d = 1$, the condition~\eqref{Bakry-Emery}
reduces to the classic Bakry--{\'E}mery criterion~\cite[Sec.~1.16]{bakry2013analysis}.
For a semigroup of the form~\eqref{eqn:semigroup}, the scalar result actually
implies the matrix result for all $d \in \N$.

\begin{proposition}[Bakry--{\'E}mery: Equivalence] \label{prop:BE_equiv}
Let $(P_t)_{t\geq 0}$ be the family of semigroups defined in~\eqref{eqn:semigroup}.
The following statements are equivalent:

\begin{enumerate}
\item \label{Bakry-Emery_criterion_scalar} \textbf{Scalar Bakry--\'Emery criterion.}
$\Gamma(f)\leq c \cdot \Gamma_2(f)$ for all suitable $f:\Omega\to \mathbb{R}$.

\item \label{Bakry-Emery_criterion_matrix} \textbf{Matrix Bakry--\'Emery criterion.}
$\Gamma(\mtx{f})\preccurlyeq c \cdot \Gamma_2(\mtx{f})$ for all suitable $\mtx{f}:\Omega \to \mathbb{H}_d$
and all $d \in \N$.
\end{enumerate}
\end{proposition}

\noindent
See Section~\ref{sec:scalar_matrix} for the proof
of Proposition~\ref{prop:BE_equiv}.

Proposition~\ref{prop:BE_equiv} is a very powerful tool, and it is a key
part of our method.
Indeed, it is already known~\cite{bakry2013analysis} that many kinds of Markov processes 
satisfy the scalar Bakry--{\'E}mery criterion~\eqref{Bakry-Emery_criterion_scalar}.
When contemplating novel settings, we only need to check the scalar criterion,
rather than worrying about matrix-valued functions.
In all these cases, we obtain the matrix extension for free.

\begin{remark}[Curvature]\label{rmk:curvature}
The scalar Bakry--\'Emery criterion, Proposition~\ref{prop:BE_equiv}\eqref{Bakry-Emery_criterion_scalar}, is also known as the curvature condition $CD(\rho,\infty)$  with $\rho=c^{-1}$. In the scenario where the infinitesimal generator $\mL$ is the Laplace--Beltrami operator $\Delta_{\mathfrak{g}}$ on a Riemannian manifold $(M,\mathfrak{g})$ with co-metric $\mathfrak{g}$, the Bakry--\'Emery criterion holds if and only if the Ricci curvature tensor is everywhere positive definite, with eigenvalues bounded from below by $\rho>0$.  See~\cite[Section 1.16]{bakry2013analysis} for a discussion.  We will return to this example
in Section~\ref{sec:Riemannin_intro}.
\end{remark}

\subsection{Bakry--{\'E}mery and ergodicity} 

The scalar Bakry--\'Emery criterion, Proposition~\ref{prop:BE_equiv}\eqref{Bakry-Emery_criterion_scalar},
is equivalent to a local Poincar\'e inequality,
which is strictly stronger than the scalar Poincar\'e inequality, Proposition~\ref{prop:poincare_equiv}\eqref{Poincare_inequality_scalar}.
It is also equivalent to a powerful local ergodicity property~\cite[Theorem 2.35]{van550probability}.
The next result states that the matrix Bakry--{\'E}mery criterion~\eqref{Bakry-Emery}
implies counterparts of these facts.

\begin{proposition}[Bakry--{\'Emery}: Consequences] 
\label{prop:local_Poincare}
Let $(P_t)_{t \geq 0}$ be a Markov semigroup acting on
suitable functions $\mtx{f} : \Omega \to \mathbb{H}_d$ for fixed $d \in \N$,
as defined in~\eqref{eqn:semigroup}. 
The following are equivalent:
\begin{enumerate}

\item \label{Bakry-Emery_criterion} \textbf{Bakry--\'Emery criterion.}
$\Gamma(\mtx{f})\preccurlyeq c \cdot \Gamma_2(\mtx{f})$ for all suitable $\mtx{f}:\Omega \to \mathbb{H}_d$.

\item  \label{local_ergodicity}
\textbf{Local ergodicity.}
$\Gamma(P_t\mtx{f})\preccurlyeq \econst^{-2t/c} \cdot P_t\Gamma(\mtx{f})$ for all $t \geq 0$ and for all suitable $\mtx{f}:\Omega \to \mathbb{H}_d$. 
\item \label{local_Poincare}
\textbf{Local Poincar\'e inequality.} $P_t(\mtx{f}^2) - (P_t\mtx{f})^2 \preccurlyeq c \,(1-\econst^{-2t/c}) \cdot P_t\Gamma(\mtx{f})$ for all $t \geq 0$ and for all suitable $\mtx{f}:\Omega \to \mathbb{H}_d$. \end{enumerate}
\end{proposition}

\noindent
The proof Proposition~\ref{prop:local_Poincare} appears in Section~\ref{sec:equivalence_local_Poincare}.
It follows along the same lines as the scalar result~\cite[Theorem 2.36]{van550probability}. 

Proposition~\ref{prop:local_Poincare} plays a central role in this paper.
With the aid of Proposition~\ref{prop:BE_equiv}, we can verify the
Bakry--{\'E}mery criterion~\eqref{Bakry-Emery_criterion} for many
particular Markov semigroups.  Meanwhile, the local ergodicity
property~\eqref{local_ergodicity} supports short derivations of
trace moment inequalities for random matrices.

The results in Proposition~\ref{prop:local_Poincare} refine the statements
in Proposition~\ref{prop:matrix_poincare}.
Indeed, the carr\'e du champ operator $\Gamma(\mtx{f})$ measures the local fluctuation of a function $\mtx{f}$,
so the local ergodicity condition~\eqref{local_ergodicity} means that the fluctuation of
$P_t\mtx{f}$ at every point $z\in \Omega$ is decreasing exponentially fast.
By applying $\E_\mu$ to both sides of the local ergodicity inequality,
we obtain the ergodicity of energy, Proposition~\ref{prop:matrix_poincare}\eqref{energy_convergence}.

If $(P_t)_{t\geq 0}$ is ergodic, applying the expectation $\E_\mu$ to the local Poincar\'e inequality~\eqref{local_Poincare} and then taking $t\rightarrow +\infty$
 yields the matrix Poincar\'e inequality, Proposition~\ref{prop:matrix_poincare}\eqref{Poincare_inequality}
 with constant $\alpha = c$.
In fact, a standard method for establishing a Poincar\'e inequality
 is to check the Bakry--\'Emery criterion. 

\begin{remark}[Noncommutative semigroups]
Junge \& Zeng have investigated the implications of
the Bakry--{\'E}mery criterion~\eqref{Bakry-Emery} for noncommutative diffusion processes on a von Neumann algebra. 
For this setting, a partial version of Proposition~\ref{prop:local_Poincare} appears in~\cite[Lemma 4.6]{junge2015noncommutative}.
\end{remark}

\subsection{Basic examples}\label{sec:examples} This section contains some examples of Markov semigroups that satisfy the Bakry--{\'E}mery criterion~\eqref{Bakry-Emery}.  In Section~\ref{sec:main_results}, we will use these semigroups to derive matrix concentration results for several random matrix models.

\subsubsection{Product measures}\label{sec:product_measure_all_intro} Consider a product space $\Omega = \Omega_1\otimes \Omega_2\otimes \cdots\otimes \Omega_n $ equipped with a product measure $\mu = \mu_1\otimes \mu_2\otimes \cdots\otimes\mu_n$.  In Section~\ref{sec:product_measure_all}, we present the standard construction of
the associated Markov semigroup, adapted to the matrix setting.
This semigroup is ergodic and reversible, and its carr\'e du champ operator takes the form
of a discrete squared derivative: 
\begin{equation}\label{eqn:variance_proxy}
\Gamma(\mtx{f})(z) = \mtx{V}(\mtx{f})(z) := \frac{1}{2}\sum_{i=1}^n\E_Z \left[ (\mtx{f}(z) - \mtx{f}((z;Z)_i))^2 \right] \quad \text{for all $z\in \Omega$}.
\end{equation}
In this expression, $Z = (Z^1,\dots,Z^n)\sim\mu$ and $(z;Z)_i = (z^1,\dots,z^{i-1},Z^i,z^{i+1},\dots,z^n)$ for each $i=1,\dots,n$.  Superscripts denote the coordinate index.

Aoun et al.~\cite{ABY20:Matrix-Poincare} have shown that this Markov semigroup satisfies
the matrix Poincar{\'e} inequality~\eqref{eqn:matrix_Poincare} with constant $\alpha = 1$.
In Section~\ref{sec:product_measure_all}, we will show that the semigroup also satisfies the Bakry--\'Emery criterion~\eqref{Bakry-Emery} with constant $c = 2$.

\subsubsection{Log-concave measures}\label{sec:log-concave_intro}
Log-concave distributions~\cite{Pre73:Logarithmic-Concave,ambrosio2009existence,saumard2014log} are a fundamental class of probability measures on $\Omega = \mathbb{R}^n$ that are closely related to diffusion processes.  A log-concave measure takes the form $\diff \mu \propto \econst^{-W(z)}\idiff z$ where the potential $W:\mathbb{R}^n\rightarrow \mathbb{R}$ is a convex function, so it captures a form of negative dependence.
The associated diffusion process naturally induces a semigroup whose
carr{\'e} du champ operator takes
the form of the squared ``magnitude'' of the gradient:
\[\Gamma(\mtx{f})(z) = \sum_{i=1}^n(\partial_i\mtx{f}(z))^2\quad \text{for all $z\in \mathbb{R}^n$}.\] 
As usual, $\partial_i := \partial/\partial z_i$ for $i = 1, \dots, n$.

Many interesting results follow from the condition that the potential $W$
is uniformly strongly convex on $\mathbb{R}^n$.
In other words, for a constant $\eta > 0$,
we assume that the Hessian matrix satisfies
\begin{equation} \label{eqn:hess-sc-intro}
(\operatorname{Hess} W)(z) := \big[ (\partial_{ij} W)(z) \big]_{i,j=1}^n \succcurlyeq \eta \cdot \Id_n
\quad\text{for all $z \in \mathbb{R}^n$.}
\end{equation}
The partial derivative $\partial_{ij} := \partial^2/(\partial z_i \partial z_j)$ for $i,j=1,\dots,n$.
It is a standard result~\cite[Sec. 4.8]{bakry2013analysis} that the strong convexity condition~\eqref{eqn:hess-sc-intro} implies the scalar Bakry--\'Emery criterion with constant $c = \eta^{-1}$.  Therefore, according to Proposition~\ref{prop:BE_equiv},
the matrix Bakry--{\'E}mery criterion~\eqref{Bakry-Emery} is valid for every $d \in \N$.

One of the core examples of a log-concave measure is the standard Gaussian measure on $\mathbb{R}^n$, which is given by the potential $W(z) = z^\trsp z/2$.  The associated diffusion process induces the Ornstein--Uhlenbeck semigroup, which satisfies the Bakry--\'Emery criterion~\eqref{Bakry-Emery} with constant $c = 1$.

A more detailed discussion on log-concave measures is presented in Section~\ref{sec:log-concave}.

\subsection{Measures on Riemannian manifolds} \label{sec:Riemannin_intro}

The theory of diffusion processes on Euclidean spaces can be generalized to the setting of Riemannian manifolds.  Although this exercise may seem abstract, it allows us to treat some interesting and important examples in a unified way.  We refer to~\cite{bakry2013analysis} for more background on this subject, and we instate their conventions.

Consider an $n$-dimensional compact Riemannian manifold $(M,\mathfrak{g})$. Let $\mathfrak{g}(x) = (g^{ij}(x) : 1 \leq i,j \leq n )$ be the matrix representation of the co-metric tensor $\mathfrak{g}$ in local coordinates, which is a symmetric and positive-definite matrix defined for every $x \in M$.
The manifold is equipped with a canonical Riemannian probability measure $\mu_\mathfrak{g}$ that has local density $\diff \mu_\mathfrak{g} \propto \det(\mathfrak{g}(x))^{-1/2} \idiff{x}$ with respect to the Lebesgue measure in local coordinates.  This measure $\mu_\mathfrak{g}$ is the stationary measure of the diffusion process on $M$ whose infinitesimal generator $\mL$ is the Laplace--Beltrami operator $\Delta_\mathfrak{g}$.  This diffusion process is called the \emph{Riemannian Brownian motion}.\footnote{Many authors use the convention that Riemmanian Brownian motion has infinitesimal generator $\tfrac{1}{2} \Delta_{\mathfrak{g}}$.}
The associated matrix carr{\'e} du champ operator coincides with the squared ``magnitude'' of the differential:
\begin{equation}\label{eqn:gamma_Riemannian_0}
\Gamma(\mtx{f})(x) = \sum_{i,j=1}^ng^{ij}(x) \,\partial_i\mtx{f}(x)\, \partial_j\mtx{f}(x)\quad \text{for suitable $\mtx{f} : M \to \mathbb{H}_d$.}
\end{equation} 
Here, $\partial_i$ for $i=1,\dots,n$ are the components of the differential, computed in local coordinates.  We emphasize that the matrix carr{\'e} du champ operator is intrinsic; expressions for the carr{\'e} du champ resulting from different choices of local coordinates are equivalent under change of variables. See Section~\ref{sec:extension_Riemannian_manifold} for a more detailed discussion.

As mentioned in Remark~\ref{rmk:curvature}, the scalar Bakry--\'{E}mery criterion holds with $c=\rho^{-1}$ if and only if the Ricci curvature tensor of $(M, \mathfrak{g})$ is everywhere positive, with eigenvalues bounded from below by $\rho>0$. In other words, for Brownian motion on a manifold, the Bakry--{\'E}mery criterion is equivalent to the uniform positive curvature of the manifold.  Proposition~\ref{prop:BE_equiv} ensures
that the matrix Bakry--{\'E}mery criterion~\eqref{Bakry-Emery} holds with $c = \rho^{-1}$
under precisely the same circumstances.

Many examples of positively curved Riemannian manifolds are discussed in \cite{ledoux2001concentration,gromov2007metric,cheeger2008comparison,bakry2013analysis}.
We highlight two particularly interesting cases.

\begin{example}[Unit sphere]
Consider the $n$-dimensional unit sphere $\mathbb{S}^{n} \subset \R^{n+1}$ for $n \geq 2$.
The sphere is equipped with the Riemannian manifold structure induced by
$\R^{n+1}$.  The canonical Riemannian measure on the
sphere is simply the uniform probability measure.
The sphere has a constant Ricci curvature tensor, whose eigenvalues all equal $n - 1$.
Therefore, the Brownian motion on $\mathbb{S}^n$ satisfies
the Bakry--{\'E}mery criterion~\eqref{Bakry-Emery}
with $c = (n-1)^{-1}$.
See~\cite[Sec.~2.2]{bakry2013analysis}.
\end{example}

\begin{example}[Special orthogonal group]
The special orthogonal group $\mathrm{SO}(n)$ can be regarded as a Riemannian submanifold
of $\R^{n \times n}$.  The Riemannian metric is the Haar probability measure
on $\mathrm{SO}(n)$.  It is known that the eigenvalues of the Ricci curvature tensor
are uniformly bounded below by $(n-1)/4$.  Therefore, the Brownian motion on $\mathrm{SO}(n)$ satisfies the Bakry--{\'E}mery criterion~\eqref{Bakry-Emery} with $c = 4/(n-1)$.  
See~\cite[pp.~26ff]{ledoux2001concentration}.
\end{example}

The lower bound on Ricci curvature is stable under (Riemannian) products of manifolds, so similar results are valid for products of spheres or products of the orthogonal group; cf.~\cite[p.~27]{ledoux2001concentration}.

\subsection{History}

In the scalar setting, much of the classic research on Markov processes concerns the behavior of diffusion processes on Riemannian manifolds.  Functional inequalities connect the convergence of these Markov processes to the geometry of the manifold.
The rate of convergence to equilibrium of a Markov process plays a core role in developing concentration properties
for the measure.  The treatise \cite{bakry2013analysis} contains a comprehensive discussion.  Other references include \cite{ledoux2001concentration,boucheron2013concentration,van550probability}.

Matrix-valued Markov processes were originally introduced to model the evolution of quantum systems \cite{davies1969quantum,lindblad1976generators,accardi1982quantum}. In recent years,
the long-term behavior of quantum Markov processes has received significant attention in the field of quantum information.
A general approach to exponential convergence of a quantum system
is to establish quantum log-Sobolev inequalities for density operators \cite{majewski1998dissipative,olkiewicz1999hypercontractivity,kastoryano2013quantum}. 

In this paper, we consider a mixed classical-quantum setting,
where a classical Markov process drives a matrix-valued function.
The papers~\cite{cheng2017exponential,cheng2019matrix,ABY20:Matrix-Poincare}
contain some foundational results for this model.
Our work provides a more detailed understanding
of the connections between the ergodicity of the
semigroup and matrix functional inequalities.
The companion paper~\cite{HT20:Trace-Poincare}
contains further results on trace Poincar{\'e} inequalities,
which are equivalent to the Poincar{\'e} inequality~\eqref{eqn:matrix_Poincare}.

A general framework for noncommutative diffusion processes
on von Neumann algebras can be found in \cite{junge2006h,junge2015noncommutative}.
In particular, the paper~\cite{junge2015noncommutative} shows that a
noncommutative Bakry--{\'E}mery criterion implies local ergodicity
of a noncommutative diffusion process. 

In spite of its generality, the presentation in~\cite{junge2015noncommutative}
does not fully contain our treatment.  On the one hand,
the noncommutative semigroup model includes the
mixed classical-quantum model~\eqref{eqn:semigroup} as a special case.
On the other hand, we do not need the underlying Markov process to
be a diffusion (with continuous sample paths), while Junge \& Zeng
pose a diffusion assumption.

\section{Nonlinear Matrix Concentration: Main Results}
\label{sec:main_results}

The matrix Poincar{\'e} inequality~\eqref{eqn:matrix_Poincare} has been associated
with subexponential concentration inequalities for random matrices~\cite{ABY20:Matrix-Poincare,HT20:Trace-Poincare}.
The central purpose of this paper is to establish that the (scalar) Bakry--{\'E}mery criterion
leads to matrix concentration inequalities via a straightforward semigroup method.
This section outlines our main results; the proofs appear in Section~\ref{sec:trace_to_moment}.

\begin{remark}[Noncommutative setting]
After this paper was written, we learned that Junge \& Zeng~\cite{junge2015noncommutative}
have used the (noncommutative) Bakry--{\'E}mery criterion to obtain subgaussian moment
bounds for elements of von Neumann algebra using a martingale approach.  Their setting
is more general (if we ignore the diffusion assumptions),
but we will see that their results are weaker in several respects.
\end{remark}

\subsection{Markov processes and random matrices}

Let $Z$ be a random variable, taking values in the state space $\Omega$, with the distribution $\mu$.
For a matrix-valued function $\mtx{f} : \Omega \to \mathbb{H}_d$, we can
define the random matrix $\mtx{f}(Z)$, whose distribution is the push-forward
of $\mu$ by the function $\mtx{f}$.  Our goal is to understand how well
the random matrix $\mtx{f}(Z)$ concentrates around its expectation
$\Expect \mtx{f}(Z) = \Expect_{\mu} \mtx{f}$.

To do so, suppose that we can construct a reversible, ergodic Markov process
$(Z_t)_{t \geq 0} \subset \Omega$ whose stationary distribution is $\mu$.
We have the intuition that the faster that the process $(Z_t)_{t \geq 0}$ converges
to equilibrium, the more sharply the random matrix $\mtx{f}(Z)$ concentrates
around its expectation.

To quantify the rate of convergence of the matrix Markov process,
we use the Bakry--{\'E}mery criterion~\eqref{Bakry-Emery}
to obtain local ergodicity of the semigroup.  This property allows
us to prove strong bounds on the trace moments of the random matrix.
Using standard arguments (Appendix~\ref{apdx:matrix_moments}), these moment bounds imply
nonlinear matrix concentration inequalities.

\subsection{Polynomial concentration}

We begin with a general estimate on the polynomial trace moments
of a random matrix under a Bakry--\'Emery criterion. 

\begin{theorem}[Polynomial moments]\label{thm:polynomial_moment} Let $\Omega$ be a Polish space equipped with a probability measure $\mu$.  Consider a reversible, ergodic Markov semigroup~\eqref{eqn:semigroup} with stationary measure $\mu$ that acts on (suitable) functions $\mtx{f} : \Omega \to \mathbb{H}_d$.
Assume that the Bakry--\'Emery criterion \eqref{Bakry-Emery} holds for a constant $c>0$.
Then, for $q=1$ and $q\geq1.5$, 
\begin{equation}\label{eqn:polynomial_moment_1}
\left[ \E_\mu \trace|\mtx{f}-\E_\mu\mtx{f}|^{2q}\right]^{1/(2q)}\leq \sqrt{c\,(2q-1)}\left[ \E_\mu\trace\Gamma(\mtx{f})^q\right]^{1/(2q)}.
\end{equation}
If the variance proxy $v_{\mtx{f}} := \norm{ \|\Gamma(\mtx{f})\| }_{L_{\infty}(\mu)} <+\infty$,
then
\begin{equation}\label{eqn:polynomial_moment_2}
\left[ \E_\mu \trace|\mtx{f}-\E_\mu\mtx{f}|^{2q}\right]^{1/(2q)}\leq d^{1/(2q)}\sqrt{c\,(2q-1) \,\smash{v_{\mtx{f}}}} .
\end{equation}
\end{theorem}

\noindent
We establish this theorem in Section~\ref{sec:trace_to_moment}.

For noncommutative diffusion semigroups,
Junge \& Zeng~\cite{junge2015noncommutative} have developed polynomial moment bounds
similar to Theorem~\ref{thm:polynomial_moment}, but they only obtain moment growth
of $O(q)$ in the inequality~\eqref{eqn:polynomial_moment_1}.  We can trace this
discrepancy to the fact that they use a martingale argument based on the
noncommutative Burkholder--Davis--Gundy inequality.  At present, our proof only
applies to the mixed classical-quantum semigroup~\eqref{eqn:semigroup}, but it
seems plausible that our approach can be generalized.

For now, let us present some concrete results that follow when we apply Theorem~\ref{thm:polynomial_moment}
to the semigroups discussed in Section~\ref{sec:examples}.  In each of these cases,
we can derive bounds for the expectation and tails of $\norm{ \smash{\mtx{f} - \E_{\mu} \mtx{f}} }$
using the matrix Chebyshev inequality (Proposition~\ref{prop:matrix_Chebyshev}).
In particular, when $v_{\mtx{f}} < + \infty$, we obtain subgaussian concentration.

\subsubsection{Polynomial Efron--Stein inequality for product measures} 

The first consequence of Theorem~\ref{thm:polynomial_moment} is a polynomial moment inequality for product measures.
This result exactly reproduces the matrix polynomial Efron--Stein inequalities established by Paulin et al.~\cite[Theorem 4.2]{paulin2016efron}.

\begin{corollary}[Product measure: Polynomial moments]\label{cor:product_measure_Efron--Stein} Let $\mu = \mu_1\otimes \mu_2\otimes \cdots\otimes\mu_n$ be a product measure on a product space $\Omega = \Omega_1\otimes \Omega_2\otimes \cdots\otimes \Omega_n $.  Let $\mtx{f}:\Omega \rightarrow \mathbb{H}_d$ be a suitable function.
Then, for $q= 1$ and $q\geq1.5$,
\begin{equation}\label{eqn:product_measure_Efron--Stein}
\left[ \E_\mu \trace|\mtx{f}-\E_\mu\mtx{f}|^{2q}\right]^{1/(2q)}\leq \sqrt{2(2q-1)}\left[\E_\mu\trace\mtx{V}(\mtx{f})^q\right]^{1/(2q)}.
\end{equation}
The matrix variance proxy $\mtx{V}(\mtx{f})$ is defined in \eqref{eqn:variance_proxy}. 
\end{corollary} 

\noindent
The details appear in Section~\ref{sec:concentration_results_product}.

\subsubsection{Log-concave measures}

The second result is a new polynomial moment inequality for matrix-valued
functions of a log-concave measure.
To avoid domain issues, we restrict our attention to the Sobolev space
\begin{equation}\label{def:H2_function}
\mathrm{H}_{2,\mu}(\mathbb{R}^n;\mathbb{H}_d) := \left\{\mtx{f} : \mathbb{R}^n\rightarrow\mathbb{H}_d:\E_\mu \|\mtx{f}\|_\mathrm{HS}^2+\sum_{i=1}^n\E_\mu \|\partial_i\mtx{f}\|_\mathrm{HS}^2 + \sum_{i,j=1}^n\E_\mu \|\partial_{ij}\mtx{f}\|_\mathrm{HS}^2 <\infty\right\}.
\end{equation}
For these functions, we have the following matrix concentration inequality.

\begin{corollary}[Log-concave measure: Polynomial moments]\label{cor:log-concave_polynomial_inequality}
Let $\diff \mu \propto \econst^{-W(z)}\idiff z$ be a log-concave measure on $\mathbb{R}^n$
whose potential $W:\mathbb{R}^n\rightarrow\mathbb{R}$ satisfies a uniform strong convexity
condition: $\operatorname{Hess} W \psdge \eta \cdot \Id_n$
with constant $\eta > 0$.
Let $\mtx{f}\in \mathrm{H}_{2,\mu}(\mathbb{R}^n;\mathbb{H}_d)$.
Then, for $q=1$ and $q\geq 1.5$, 
\begin{equation*}\label{eqn:log-concave_polynomial_inequality}
\left[\E_\mu \trace|\mtx{f}-\E_\mu\mtx{f}|^{2q}\right]^{1/(2q)}\leq \sqrt{\frac{2q-1}{\eta}}\left[\E_\mu\trace\left(\sum_{i=1}^n(\partial_i\mtx{f})^2\right)^q\right]^{1/(2q)}.
\end{equation*}
\end{corollary}

\noindent
The details appear in Section~\ref{sec:concentration_results_log-concave}.

\subsection{Exponential concentration}

As a consequence of the Bakry--{\'E}mery criterion~\eqref{Bakry-Emery},
we can also derive exponential matrix concentration inequalities.
In principle, polynomial moment inequalities are stronger,
but the exponential inequalities often lead to better constants and more
detailed information about tail decay.

\begin{theorem}[Exponential concentration]\label{thm:exponential_concentration}
Let $\Omega$ be a Polish space equipped with a probability measure $\mu$.
Consider a reversible, ergodic Markov semigroup \eqref{eqn:semigroup} with stationary measure $\mu$
that acts on (suitable) functions $\mtx{f} : \Omega \to \mathbb{H}_d$.
Assume that the Bakry--\'Emery criterion \eqref{Bakry-Emery}
holds for a constant $c>0$.
Then 
\begin{align}\label{eqn:tail_bound_1}
\mathbb{P}_{\mu}\left\{\lambda_{\max}(\mtx{f}-\E_\mu\mtx{f})\geq t \right\} \leq&\ d\cdot \inf_{\beta>0} \exp \left(\frac{-t^2}{2cr_{\mtx{f}}(\beta) + 2t\sqrt{c/\beta} }\right) \quad\text{for all $t \geq 0$.} 
\end{align}
The function $r_{\mtx{f}}$ computes an exponential mean of the carr{\'e} du champ:
\[
r_{\mtx{f}}(\beta):=\frac{1}{\beta}\log \E_\mu\ntr \econst^{ \beta\Gamma(\mtx{f}) }
\quad\text{for $\beta > 0$.}
\]
In addition, suppose that the variance proxy $v_{\mtx{f}} := \norm{ \|\Gamma(\mtx{f}) \| }_{L_{\infty}(\mu)} <+\infty$.
Then 
\begin{equation*}\label{eqn:tail_bound_2}
\mathbb{P}_{\mu}\left\{\lambda_{\max}(\mtx{f}-\E_\mu\mtx{f})\geq t \right\} \leq d\cdot \exp \left(\frac{-t^2}{2cv_{\mtx{f}}}\right)
\quad\text{for all $t \geq 0$.}
\end{equation*}
Furthermore,
\begin{equation*}\label{eqn:expectation_bound}
\E_\mu\lambda_{\max}(\mtx{f}-\E_\mu\mtx{f}) \leq \sqrt{2cv_{\mtx{f}}\log d}.
\end{equation*}
Parallel inequalities hold for the minimum eigenvalue $\lambda_{\min}$.
\end{theorem}

\noindent
We establish Theorem~\ref{thm:exponential_concentration} in Section~\ref{sec:exponential_concentration_proof}
as a consequence of an exponential moment inequality, Theorem~\ref{thm:exponential_moment}, for random matrices.
By combining Theorem~\ref{thm:exponential_concentration}  with the examples in Section~\ref{sec:examples},
we obtain concentration results for concrete random matrix models.

A partial version of Theorem~\ref{thm:exponential_concentration} with slightly worse constants
appears in \cite[Corollary 4.13]{junge2015noncommutative}.
When comparing these results, note that probability measure in \cite{junge2015noncommutative}
is normalized to absorb the dimensional factor $d$.

\subsubsection{Exponential Efron--Stein inequality for product measures}

We can reproduce the matrix exponential Efron--Stein inequalities of Paulin et al.~\cite[Theorem 4.3]{paulin2016efron}
by applying Theorem~\ref{thm:exponential_moment} to a product measure (Section~\ref{sec:product_measure_all_intro}).
For instance, we obtain the following subgaussian inequality.

\begin{corollary}[Product measure: Subgaussian concentration]\label{cor:product_measure_tailbound} Let $\mu = \mu_1\otimes \mu_2\otimes \cdots\otimes\mu_n$ be a product measure on a product space $\Omega = \Omega_1\otimes \Omega_2\otimes \cdots\otimes \Omega_n $. Let $\mtx{f}:\Omega \rightarrow \mathbb{H}_d$ be a suitable function.
Define the variance proxy $v_{\mtx{f}} := \norm{ \|\mtx{V}(\mtx{f}) \| }_{L_{\infty}(\mu)}$,
where $\mtx{V}(\mtx{f})$ is given by \eqref{eqn:variance_proxy}.  Then \begin{align*}\label{eqn:product_measure_tailbound}
\Prob{\lambda_{\max}(\mtx{f}-\E_\mu\mtx{f})\geq t } \leq d\cdot\exp\left(-\frac{t^2}{4v_{\mtx{f}}}\right)
\quad\text{for all $t \geq 0$.} 
\end{align*}
Furthermore, 
\begin{equation*}\label{eqn:product_measure_expectation}
\E_\mu \lambda_{\max}(\mtx{f}-\E_\mu\mtx{f}) \leq 2\sqrt{v_{\mtx{f}}\log d}.
\end{equation*}
Parallel results hold for the minimum eigenvalue $\lambda_{\min}$.
\end{corollary}

\noindent
We defer the proof to Section~\ref{sec:concentration_results_product}. 

\subsubsection{Log-concave measures}

We can also obtain exponential concentration for a matrix-valued function
of a log-concave measure by combining Theorem~\ref{thm:exponential_concentration}
with the results in Section~\ref{sec:log-concave_intro}.

\begin{corollary}[Log-concave measure: Subgaussian concentration]\label{cor:log-concave_concentration}
Let $\diff \mu \propto \econst^{-W(z)}\idiff z$ be a log-concave probability measure on $\mathbb{R}^n$
whose potential $W : \mathbb{R}^n \to \mathbb{R}$ satisfies a uniform strong convexity condition:
$\operatorname{Hess} W \succcurlyeq \eta \cdot \Id_n$ where $\eta > 0$.
Let $\mtx{f} \in \mathrm{H}_{2,\mu}(\R^n; \mathbb{H}_d)$, and define the variance proxy
\[
v_{\mtx{f}} := \sup\nolimits_{z \in \R^n} \norm{ \sum_{i=1}^n(\partial_i\mtx{f}(z))^2 }.
\]
Then
\[\mathbb{P}_{\mu}\left\{\lambda_{\max}(\mtx{f}-\E_\mu\mtx{f})\geq t \right\} \leq d\cdot\exp\left(\frac{-\eta t^2}{2v_{\mtx{f}}}\right)
\quad\text{for all $t \geq 0$.}
\]
Furthermore, 
\[\E_\mu \lambda_{\max}(\mtx{f}-\E_\mu\mtx{f}) \leq \sqrt{2 \eta^{-1} v_{\mtx{f}}\log d }.\]
Parallel results hold for the minimum eigenvalue $\lambda_{\min}$.
\end{corollary}

\noindent
See Section~\ref{sec:concentration_results_log-concave} for the proof.

\begin{example}[Matrix Gaussian series] \label{ex:matrix-gauss}
Consider the standard normal measure $\gamma_n$ on $\R^n$.
Its potential, $W(z) = z^\trsp z / 2$, is uniformly strongly convex
with parameter $\eta = 1$.  Therefore, Corollary~\ref{cor:log-concave_concentration}
gives subgaussian concentration for matrix-valued functions of a Gaussian
random vector.  
To make a comparison with familiar results, we construct the matrix Gaussian series
\begin{equation*} \label{eqn:gauss-series_1}
\mtx{f}(z) = \sum_{i=1}^n Z_i \mtx{A}_i
\quad\text{where $z = (Z_1, \dots, Z_n) \sim \gamma_n$ and $\mtx{A}_i \in \mathbb{H}_d$ are fixed.}
\end{equation*}
In this case, the carr{\'e} du champ is simply
$$
\Gamma(\mtx{f})(z) = \sum_{i=1}^n \mtx{A}_i^2.
$$
Thus, the expectation bound states that
\begin{equation*} \label{eqn:gauss-series_2}
\Expect_{\gamma_n} \lambda_{\max}(\mtx{f}(z)) \leq \sqrt{2 v_{\mtx{f}} \log d}
\quad\text{where}\quad
v_{\mtx{f}} = \norm{ \sum_{i=1}^n \mtx{A}_i^2 }.
\end{equation*}
Up to and including the constants, this matches the sharp bound that follows from
``linear'' matrix concentration techniques~\cite[Chapter 4]{tropp2015introduction}.
\end{example}

Van Handel (private communication) has outlined out an alternative proof of
Corollary~\ref{cor:log-concave_concentration} with slightly worse constants.
His approach uses Pisier's method~\cite[Thm.~2.2]{pisier1986probabilistic}
and the noncommutative Khintchine inequality~\cite{buchholz2001operator} to
obtain the statement for the standard normal measure.  Then Caffarelli's
contraction theorem~\cite{Caf00:Monotonicity-Properties} implies that the
same bound holds for every log-concave measure whose potential is
uniformly strongly convex with $\eta \geq 1$.  This approach is short
and conceptual, but it is more limited in scope.

\subsection{Riemannian measures}
\label{sec:riemann-exp}

As discussed in Section~\ref{sec:Riemannin_intro},
the Brownian motion on a Riemannian manifold with uniformly positive curvature
satisfies the Bakry--{\'E}mery criterion~\eqref{Bakry-Emery}.
Therefore, we can apply both Theorem~\ref{thm:polynomial_moment} and 
Theorem~\ref{thm:exponential_concentration} 
in this setting.  Let us give a few concrete examples of the
kind of results that can be derived with these methods.

\subsubsection{The sphere}

Consider the uniform distribution $\sigma_n$ on the $n$-dimensional unit sphere
$\mathbb{S}^n \subset \R^{n+1}$ for $n \geq 2$.  The Brownian motion on the sphere satisfies
the Bakry--{\'E}mery criterion~\eqref{Bakry-Emery} with $c = (n-1)^{-1}$.
Therefore, Theorem~\ref{thm:polynomial_moment} implies that, for any suitable function $\mtx{f} : \mathbb{S}^n \to \mathbb{H}_d$,
\[
\left[ \E_{\sigma_n} \trace|\mtx{f}-\E_{\sigma_n} \mtx{f}|^{2q}\right]^{1/(2q)}
\leq \sqrt{\frac{2q-1}{n-1}}\left[ \E_{\sigma_n} \trace\Gamma(\mtx{f})^q\right]^{1/(2q)},
\]
where the carr{\'e} du champ $\Gamma(\mtx{f})$ is defined by~\eqref{eqn:gamma_Riemannian_0}.
We can also obtain subgaussian tail bounds in terms of the variance proxy
$
v_{\mtx{f}} := \norm{ \norm{ \Gamma(\mtx{f}) } }_{L_{\infty}(\sigma_n)}.
$
Indeed, Theorem~\ref{thm:exponential_concentration} yields the bound
\[
\mathbb{P}_{\sigma_n}\left\{\lambda_{\max}(\mtx{f}-\E_{\sigma_n}\mtx{f})\geq t \right\}
	\leq d\cdot \exp \left(\frac{-(n-1)t^2}{2v_{\mtx{f}}} \right)
\quad\text{for all $t \geq 0$.}
\]
To use these concentration inequalities, we need to compute the carr{\'e} du champ $\Gamma(\mtx{f})$ and bound the variance proxy $v_{\mtx{f}}$ for particular functions $\mtx{f}$.  

We give two illustrations, postponing the detailed calculations to Section~\ref{sec:Riemannian_gamma}.
In each case, let $x = (x_1, \dots, x_{n+1}) \in \mathbb{S}^n$
be a random vector drawn from the uniform probability measure $\sigma_n$.
Suppose that $(\mtx{A}_1, \dots, \mtx{A}_{n+1}) \subset \mathbb{H}_d$ is a list of deterministic Hermitian matrices.

\begin{example}[Sphere I]\label{example:sphere_I}
Consider the random matrix $\mtx{f}(x) = \sum_{i=1}^{n+1}x_i\mtx{A}_i$. We can compute the carr{\'e} du champ as
\begin{equation}\label{eqn:gamma_sphere_I}
\Gamma(\mtx{f})(x) = \sum_{i=1}^{n+1}\mtx{A}_i^2 - \left(\sum_{i=1}^{n+1} x_i\mtx{A}_i\right)^2
	\psdge \mtx{0}.
\end{equation}
It is obvious that $\Gamma(\mtx{f})(x) \psdle \sum_{i=1}^{n+1}\mtx{A}_i^2$ for all $x\in \mathbb{S}^n$, so the variance proxy  $v_{\mtx{f}}\leq \norm{ \sum_{i=1}^{n+1}\mtx{A}_i^2 }$.

Compare this calculation with Example~\ref{ex:matrix-gauss},
where the coefficients follow the standard normal distribution.
For the sphere, the carr{\'e} du champ operator is
smaller because a finite-dimensional sphere has
slightly more curvature than the Gauss space.
\end{example}

\begin{example}[Sphere II]\label{example:sphere_II}
Consider the random matrix $\mtx{f}(x) = \sum_{i=1}^{n+1}x_i^2\mtx{A}_i$. The carr{\'e} du champ admits the expression
\begin{equation}\label{eqn:gamma_sphere_II}
\Gamma(\mtx{f})(x) = 2\sum_{i,j=1}^{n+1}x_i^2x_j^2(\mtx{A}_i-\mtx{A}_j)^2. 
\end{equation}
A simple bound shows that the variance proxy $v_{\mtx{f}} \leq 2 \max_{i, j} \norm{ \smash{\mtx{A}_i - \mtx{A}_j} }$.
It is possible to make further improvements in some cases.
\end{example}

\subsubsection{The special orthogonal group}

The Riemannian manifold framework also encompasses matrix-valued functions of random orthogonal matrices.
For instance, suppose that $\mtx{O}_1, \dots, \mtx{O}_n \in \mathrm{SO}(d)$ are drawn independently
and uniformly from the Haar measure $\mu$ on the special orthogonal group $\mathrm{SO}(d)$.
As discussed in Section~\ref{sec:Riemannin_intro}, the Brownian motion on the product
space satisfies the Bakry--{\'E}mery criterion with constant $c = 4/(d-1)$.
In particular, if $\mtx{f} : \mathrm{SO}(d)^{\otimes n} \to \mathbb{H}_d$,
$$
\mathbb{P}_{\mu^{\otimes n}}\left\{ \lambda_{\max}(\mtx{f} - \E_{\mu^{\otimes n}} \mtx{f}) \geq t \right\}
	\leq d \cdot \exp\left( \frac{-(d-1) t^2}{8 v_{\mtx{f}}} \right)
	\quad\text{for all $t \geq 0$.}
$$
Here is a particular example where we can bound the variance proxy.

\begin{example}[Special orthogonal group]\label{example:SO_d}
Let $(\mtx{A}_1, \dots, \mtx{A}_n) \subset \mathbb{H}_d(\mathbb{R})$ be a fixed list of real, symmetric matrices.
Consider the random matrix $\mtx{f}(\mtx{O}_1, \dots, \mtx{O}_n) = \sum_{i=1}^n \mtx{O}_i \mtx{A}_i \mtx{O}_i^\trsp$.
The carr{\'e} du champ is  
\begin{equation}\label{eqn:gamma_SO_d}
\Gamma(\mtx{f})(\mtx{O}_1, \dots, \mtx{O}_n) = \frac{1}{2}\sum_{i=1}^n\mtx{O}_i\left[ \left(\trace[\mtx{A}_i^2]-d^{-1}\trace[\mtx{A}_i]^2\right)\cdot\Id_d + d\left(\mtx{A}_i-d^{-1}\trace[\mtx{A}_i]\cdot \Id_d \right)^2\right] \mtx{O}_i^\trsp.
\end{equation}
Each matrix $\mtx{O}_i$ is orthogonal, so the variance proxy satisfies 
\[
v_{\mtx{f}} \leq \frac{1}{2}\sum_{i=1}^n \left[ \trace[\mtx{A}_i^2]-d^{-1}\trace[\mtx{A}_i]^2 + d\cdot \norm{\mtx{A}_i - d^{-1}\trace[\mtx{A}_i]\cdot \Id_d  }^2 \right].
\]
The details of the calculation appear in Section~\ref{sec:Riemannian_gamma}.
\end{example}

\subsection{Extension to general rectangular matrices}

By a standard formal argument, we can extend the results in this section to a function $\mtx{h}:\Omega\rightarrow \mathbb{M}^{d_1\times d_2}$ that takes rectangular matrix values.  To do so, we simply apply the theorems to the self-adjoint dilation 
\[\mtx{f}(z) = \left[\begin{array}{cc} 
\mtx{0} & \mtx{h}(z)\\
\mtx{h}(z)^*& \mtx{0} \end{array}\right] \in \mathbb{H}_{d_1+d_2}.\] 
See~\cite{tropp2015introduction} for many examples of this methodology.

\subsection{History}\label{sec:concentration_history}

Matrix concentration inequalities are noncommutative extensions of their scalar counterparts.
They have been studied extensively, and they have had a profound impact on a
wide range of areas in computational mathematics and statistics.
The models for which the most complete results are available
include a sum of independent random matrices~\cite{lust1986inegalites,rudelson1999random,oliveira2010sums,tropp2012user,huang2019generalized}
and a matrix-valued martingale sequence~\cite{pisier1997non,oliveira2009concentration,tropp2011freedman,junge2015noncommutative,howard2018exponential}.
We refer to the monograph \cite{tropp2015introduction} for an introduction and an extensive bibliography.  Very recently, some concentration results for products of random matrices have
also been established~\cite{henriksen2020concentration,huang2020matrix}.

In recent years, many authors have sought concentration results for more general random matrix models.
One natural idea is to develop matrix versions of scalar concentration techniques based on functional inequalities
or based on Markov processes.

In the scalar setting, the subadditivity of the entropy plays a basic
role in obtaining modified log-Sobolev inequalities for product spaces,
a core ingredient in proving subgaussian concentration results.
Chen and Tropp \cite{chen2014subadditivity}
established the subadditivity of matrix trace entropy quantities.
Unfortunately, the approach in \cite{chen2014subadditivity}
requires awkward additional assumptions to derive matrix
concentration from modified log-Sobolev inequalities.
Cheng et al.~\cite{cheng2016characterizations,cheng2017exponential,cheng2019matrix}
have extended this line of research.

Mackey et al.~\cite{mackey2014,paulin2016efron} observed that the method
of exchangeable pairs~\cite{stein1972,stein1986approximate,chatterjee2005concentration}
leads to more satisfactory matrix concentration inequalities,
including matrix generalizations of the Efron--Stein--Steele inequality.
The argument in~\cite{paulin2016efron} can be viewed as a discrete version
of the semigroup approach that we use in this paper;
see Appendix~\ref{apdx:Stein_method} for more discussion.

Very recently, Aoun et al.~\cite{ABY20:Matrix-Poincare} showed how to derive
exponential matrix concentration inequalities from the matrix Poincar{\'e} inequality~\eqref{eqn:matrix_Poincare}.
Their approach is based on the classic iterative argument, due to
Aida \& Stroock~\cite{aida1994moment}, that operates in the scalar setting.
For matrices, it takes serious effort to implement this technique.
In our companion paper~\cite{HT20:Trace-Poincare}, we have shown that
a trace Poincar{\'e} inequality leads to stronger exponential concentration
results via an easier argument.

Another appealing contribution of the paper~\cite{ABY20:Matrix-Poincare} is to
establish the validity of a matrix Poincar\'e inequality for
particular matrix-valued Markov processes.  Unfortunately,
Poincar{\'e} inequalities are apparently not strong
enough to capture subgaussian concentration. 
In the scalar case, log-Sobolev inequalities lead to subgaussian concentration inequalities.
At present, it is not clear how to extend the theory of log-Sobolev inequalities to matrices,
and this obstacle has delayed progress on studying matrix concentration via functional inequalities.

In the scalar setting, one common technique for establishing a log-Sobolev inequality is
to prove that the Bakry--{\'E}mery criterion holds~\cite[Problem 3.19]{van550probability}.
Inspired by this observation, we have chosen to investigate the implications
of the Bakry--{\'E}mery criterion~\eqref{Bakry-Emery} for Markov semigroups acting on matrix-valued functions. 
Our work demonstrates that this type of curvature condition allows us to establish matrix moment
bounds directly, without the intermediation of a log-Sobolev inequality.
As a consequence, we can obtain subgaussian and subgamma concentration
for nonlinear random matrix models. 

After establishing the results in this paper, we discovered that Junge \& Zeng~\cite{junge2015noncommutative}
have also derived subgaussian matrix concentration inequalities from the (noncommutative) Bakry--{\'E}mery criterion.
Their approach is based on a noncommutative version of the Burkholder--Davis--Gundy inequality and a martingale argument that applies to a wider class of noncommutative diffusion semigroups acting on von Neumann algebras. As a consequence, their results apply to a larger family of examples, but the moment growth bounds are somewhat worse.

In contrast, our paper develops a direct argument for the
mixed classical-quantum semigroup~\eqref{eqn:semigroup}
that does not require any sophisticated tools from operator
theory or noncommutative probability.
Instead, we establish a new trace inequality (Lemma~\ref{lem:key_Gamma})
that mimics the chain rule for a scalar diffusion semigroup.

\section{Matrix Markov semigroups: Properties and proofs}
\label{sec:matrix_Markov_semigroups_more}

This section presents some other fundamental facts about matrix Markov semigroups.
We also provide proofs of the propositions from Section~\ref{sec:matrix_Markov_semigroups}.

\subsection{Properties of the carr\'e du champ operator}

Our first proposition gives the matrix extension of some classic
facts about the carr\'e du champ operator $\Gamma$.
Parts of this result are adapted from~\cite[Prop.~2.2]{ABY20:Matrix-Poincare}.

\begin{proposition}[Matrix carr{\'e} du champ] \label{prop:Gamma_property}
Let $(Z_t)_{t \geq 0}$ be a Markov process.
The associated matrix bilinear form $\Gamma$ has the following properties:
\begin{enumerate}
\item \label{limit_formula}
For all suitable $\mtx{f},\mtx{g}: \Omega \rightarrow \mathbb{H}_d$ and all $z \in \Omega$,
\begin{equation}\label{eqn:limit_formula_Gamma}
\Gamma(\mtx{f},\mtx{g})(z) = \lim_{t\downarrow 0} \frac{1}{2t} \E\big[\big(\mtx{f}(Z_t)-\mtx{f}(Z_0)\big)\big(\mtx{g}(Z_t)-\mtx{g}(Z_0)\big)\,\big|\,Z_0=z\big].
\end{equation}

\item\label{gamma_psd}
In particular, the quadratic form $\mtx{f} \mapsto \Gamma(\mtx{f})$ is positive: $\Gamma(\mtx{f})\succcurlyeq \mtx{0}$.

\item\label{gamma_young}
For all suitable $\mtx{f},\mtx{g}: \Omega \rightarrow \mathbb{H}_d$ and all $s > 0$, 
\[\Gamma(\mtx{f},\mtx{g}) + \Gamma(\mtx{g},\mtx{f})\preccurlyeq s \,\Gamma(\mtx{f}) + s^{-1}\,\Gamma(\mtx{g}).\]

\item \label{convexity}
The quadratic form induced by $\Gamma$ is operator convex:
\[\Gamma\big(\tau\mtx{f}+(1-\tau)\mtx{g}\big)\preccurlyeq \tau \,\Gamma(\mtx{f}) + (1-\tau)\,\Gamma(\mtx{g})\quad \text{for each}\ \tau\in[0,1].\]
\end{enumerate}
Similar results hold for the matrix Dirichlet form, owing
to the definition~\eqref{eqn:Dirichlet_form}.
\end{proposition}

\begin{proof}
\textit{Proof of \eqref{limit_formula}.}
The limit form of the carr\'e du champ can be verified with a short calculation:
\begin{align*}
\Gamma(\mtx{f},\mtx{g})(z) =&\ \lim_{t\downarrow 0}\frac{1}{2t} \big[ \E[\mtx{f}(Z_t)\mtx{g}(Z_t)\,|\,Z_0=z]-\mtx{f}(z)\mtx{g}(z) \big] \\
&\quad - \lim_{t\downarrow 0}\frac{1}{2t} \big[\mtx{f}(z)\big(\E[\mtx{g}(Z_t)\,|\,Z_0=z]-\mtx{g}(z)\big) \big] - \lim_{t\downarrow 0}\frac{1}{2t}\big[\big(\E[\mtx{f}(Z_t)\,|\,Z_0=z]-\mtx{f}(z)\big)\mtx{g}(z)\big] \\
=&\ \lim_{t\downarrow 0}\frac{1}{2t} \E\big[\mtx{f}(Z_t)\mtx{g}(Z_t) - \mtx{f}(z)\mtx{g}(Z_t) -\mtx{f}(Z_t)\mtx{g}(z) + \mtx{f}(z)\mtx{g}(z)\,|\,Z_0=z\big]\\
=&\ \lim_{t\downarrow 0}\frac{1}{2t} \E
\big[(\mtx{f}(Z_t)-\mtx{f}(Z_0))(\mtx{g}(Z_t)-\mtx{g}(Z_0))\,|\,Z_0=z\big].
\end{align*}
The first relation depends on the definition~\eqref{eqn:definition_Gamma} of $\Gamma$
and the definition~\eqref{eqn:Markov_generator} of $\mL$.

\textit{Proof of \eqref{gamma_psd}.}  The fact that $\mtx{f} \mapsto \Gamma(\mtx{f})$
is positive follows from~\eqref{limit_formula} because the square
of a matrix is positive-semidefinite and the expectation preserves positivity.

\textit{Proof of \eqref{gamma_young}.}
The Young inequality for the carr{\'e} du champ follows from the fact that $\Gamma$
is positive:
\[\mtx{0}\preccurlyeq \Gamma(s^{1/2}\mtx{f} - s^{-1/2}\mtx{g}) = s \, \Gamma(\mtx{f}) + s^{-1} \,\Gamma(\mtx{g}) - \Gamma(\mtx{f},\mtx{g}) - \Gamma(\mtx{g},\mtx{f}). \]
The second relation holds because $\Gamma$ is a bilinear form.

\textit{Proof of \eqref{convexity}.}
To establish operator convexity, we use bilinearity again:
\begin{align*}
\Gamma(\tau\mtx{f}+(1-\tau)\mtx{g}) &= \tau^2\,\Gamma(\mtx{f}) + (1-\tau)^2\,\Gamma(\mtx{g}) + \tau(1-\tau)\left(\Gamma(\mtx{f},\mtx{g}) + \Gamma(\mtx{g},\mtx{f})\right) \\
&\preccurlyeq \tau^2\,\Gamma(\mtx{f}) + (1-\tau)^2\,\Gamma(\mtx{g}) + \tau(1-\tau)\left(\Gamma(\mtx{f}) + \Gamma(\mtx{g})\right)
=\tau \,\Gamma(\mtx{f}) + (1-\tau)\,\Gamma(\mtx{g}).
\end{align*}
The first semidefinite inequality follows from~\eqref{gamma_young} with $s = 1$.
\end{proof}

The next lemma is an extension of Proposition~\ref{prop:Gamma_property}\eqref{limit_formula}.
We use this result to establish the all-important chain rule inequality in Section~\ref{sec:trace_to_moment}.

\begin{lemma}[Triple product] \label{lem:three_limit} Let $(Z_t)_{t\geq0}$ be a reversible Markov process with a stationary measure $\mu$ and infinitesimal generator $\mL$. For all suitable $\mtx{f},\mtx{g},\mtx{h}:\Omega \rightarrow \mathbb{H}_d$
and all $z \in \Omega$,
\begin{align*}
& \lim_{t\downarrow 0} \frac{1}{t} \trace \E\big[\big(\mtx{f}(Z_t)-\mtx{f}(Z_0)\big)\big(\mtx{g}(Z_t)-\mtx{g}(Z_0)\big)\big(\mtx{h}(Z_t)-\mtx{h}(Z_0)\big)\big]\,\big|\,Z_0=z\big] \\
&\qquad= \trace\big[ \mL(\mtx{f}\mtx{g}\mtx{h}) - \mL(\mtx{f}\mtx{g})\mtx{h} - \mL(\mtx{h}\mtx{f})\mtx{g} - \mL(\mtx{g}\mtx{h})\mtx{f} + \mL(\mtx{f})\mtx{g}\mtx{h}+ \mL(\mtx{g})\mtx{h}\mtx{f} + \mL(\mtx{h})\mtx{f}\mtx{g}\big](z).
\end{align*}
In particular,
\[\E_{Z\sim\mu} \lim_{t\downarrow 0} \frac{1}{t} \trace \E\big[\big(\mtx{f}(Z_t)-\mtx{f}(Z_0)\big)\big(\mtx{g}(Z_t)-\mtx{g}(Z_0)\big)\big(\mtx{h}(Z_t)-\mtx{h}(Z_0)\big)\big]\,\big|\,Z_0=Z\big] = 0.\]
\end{lemma}

\begin{proof}
For simplicity, we abbreviate 
\[\mtx{f}_t = \mtx{f}(Z_t),\quad \mtx{g}_t = \mtx{g}(Z_t),\quad \mtx{h}_t = \mtx{h}(Z_t)\quad\text{and}\quad \mtx{f}_0 = \mtx{f}(Z_0),\quad \mtx{g}_0 = \mtx{g}(Z_0),\quad \mtx{h}_0 = \mtx{h}(Z_0).\]
Direct calculation gives 
\begin{align*}
& \lim_{t\downarrow 0} \frac{1}{t} \trace \E\left[\big(\mtx{f}(Z_t)-\mtx{f}(Z_0)\big)\big(\mtx{g}(Z_t)-\mtx{g}(Z_0)\big)\big(\mtx{h}(Z_t)-\mtx{h}(Z_0)\big) \,\big|\,Z_0=z\right] \\
&\quad= \lim_{t\downarrow 0} \frac{1}{t} \trace \E\left[ \mtx{f}_t\mtx{g}_t\mtx{h}_t - \mtx{f}_t\mtx{g}_t\mtx{h}_0 -\mtx{f}_t\mtx{g}_0\mtx{h}_t + \mtx{f}_t\mtx{g}_0\mtx{h}_0 -\mtx{f}_0\mtx{g}_t\mtx{h}_t + \mtx{f}_0\mtx{g}_t\mtx{h}_0 + \mtx{f}_0\mtx{g}_0\mtx{h}_t - \mtx{f}_0\mtx{g}_0\mtx{h}_0 \,\big|\, Z_0 = z\right]\\
&\quad = \lim_{t\downarrow 0} \frac{1}{t} \trace \E\big[ \big(\mtx{f}_t\mtx{g}_t\mtx{h}_t - \mtx{f}_0\mtx{g}_0\mtx{h}_0\big) - \big((\mtx{f}_t\mtx{g}_t - \mtx{f}_0\mtx{g}_0)\mtx{h}_0\big) - \big((\mtx{h}_t\mtx{f}_t - \mtx{h}_0\mtx{f}_0)\mtx{g}_0\big)  + \big((\mtx{f}_t - \mtx{f}_0)\mtx{g}_0\mtx{h}_0\big) \\
&\qquad\qquad\qquad\qquad  -\ \big((\mtx{g}_t\mtx{h}_t - \mtx{g}_0\mtx{h}_0)\mtx{f}_0\big) + \big((\mtx{g}_t - \mtx{g}_0)\mtx{h}_0\mtx{f}_0\big) + \big((\mtx{h}_t - \mtx{h}_0)\mtx{f}_0\mtx{g}_0\big) \,\big|\,Z_0=z \big]\\
&\quad= \trace\big[ \mL(\mtx{f}\mtx{g}\mtx{h})(z) - \mL(\mtx{f}\mtx{g})(z)\mtx{h}(z) - \mL(\mtx{h}\mtx{f})(z)\mtx{g}(z) - \mL(\mtx{g}\mtx{h})(z)\mtx{f}(z) \\
& \qquad\qquad\quad +\ \mL(\mtx{f})(z)\mtx{g}(z)\mtx{h}(z) + \mL(\mtx{g})(z)\mtx{h}(z)\mtx{f}(z) + \mL(\mtx{h})(z)\mtx{f}(z)\mtx{g}(z)\big]. 
\end{align*}
We have applied the cyclic property of the trace.
Using the reversibility~\eqref{eqn:reversibility_2} of the Markov process
and the zero-mean property~\eqref{eqn:mean_zero} of the infinitesimal generator,
we have
\begin{align*}
& \E_\mu \trace\left[ \mL(\mtx{f}\mtx{g}\mtx{h}) - \mL(\mtx{f}\mtx{g})\mtx{h} - \mL(\mtx{h}\mtx{f})\mtx{g} - \mL(\mtx{g}\mtx{h})\mtx{f} + \mL(\mtx{f})\mtx{g}\mtx{h}+ \mL(\mtx{g})\mtx{h}\mtx{f} + \mL(\mtx{h})\mtx{f}\mtx{g}\right]\\
&\quad= \trace\left[ \E_\mu[\mL(\mtx{f}\mtx{g}\mtx{h})] -  \E_\mu[\mL(\mtx{f}\mtx{g})\mtx{h} - \mtx{f}\mtx{g}\mL(\mtx{h})] -\E_\mu[\mL(\mtx{h}\mtx{f})\mtx{g} - \mtx{h}\mtx{f}\mL(\mtx{g})] - \E_\mu[\mL(\mtx{g}\mtx{h})\mtx{f} - \mtx{g}\mtx{h}\mL(\mtx{f})] \right] \\
&\quad= 0.
\end{align*}
This concludes the second part of the lemma.
\end{proof}

\subsection{Reversibility}
\label{sec:reversibility-pf}

In this section, we establish Proposition~\ref{prop:reversibility}, which states that
reversibility of the semigroup~\eqref{eqn:semigroup}
 on real-valued functions
is equivalent with the reversibility of the semigroup
on matrix-valued functions.  The pattern of argument
was suggested to us by Ramon van Handel, and it will
be repeated below in the proofs that certain functional
inequalities for real-valued functions are equivalent
with functional inequalities for matrix-valued functions.

\begin{proof}[Proof of Proposition~\ref{prop:reversibility}]
The implication that matrix reversibility~\eqref{eqn:reversibility_1}
for all $d \in \N$ implies scalar reversibility is obvious: just take $d = 1$.
To check the converse,
we require an elementary identity.
For all vectors $\vct{u},\vct{v}\in \mathbb{C}^d$ and all matrices $\mtx{A},\mtx{B}\in \mathbb{H}_d$,
\begin{align}
\vct{u}^*(\mtx{A}\mtx{B})\vct{v}
&= \sum_{j=1}^d (\vct{u}^*\mtx{A}\mathbf{e}_j)(\mathbf{e}_j^*\mtx{B}\vct{v}) =: \sum_{j=1}^da_j\bar{b}_j\\
&= \sum_{j=1}^d\left[\real(a_j)\real(b_j) + \imag(a_j)\imag(b_j) - \iunit\real(a_j)\imag(b_j) + \iunit\imag(a_j)\real(b_j)\right].  \label{eqn:Ramon}
\end{align}
We have defined $a_j:= \vct{u}^*\mtx{A}\mathbf{e}_j$ and $b_j:= \vct{v}^*\mtx{A}\mathbf{e}_j$ for each $j=1,\dots,d$.
As usual, $(\mathbf{e}_j : 1 \leq j \leq d)$ is the standard basis for $\CC^d$.

Now, consider two matrix-valued functions $\mtx{f},\mtx{g}:\Omega\rightarrow\mathbb{H}_d$.  Introduce the scalar functions $f_j := \vct{u}^*\mtx{f}\mathbf{e}_j$ and $g_j := \vct{v}^*\mtx{g}\mathbf{e}_j$ for each $j=1,\dots,d$. 
The definition~\eqref{eqn:semigroup} of the semigroup $(P_t)_{t\geq0}$ as an expectation ensures that 
\[\vct{u}^*(P_t\mtx{f})\mathbf{e}_j = P_tf_j = P_t(\real(f_j)) + \iunit\,P_t(\imag(f_j)) = \real(P_t f_j) + \iunit \imag(P_tf_j).\]
The parallel statement holds for $\vct{v}^*(P_t\mtx{g})\mathbf{e}_j$.  
Therefore, we can use formula \eqref{eqn:Ramon} to compute that 
\begin{align*}
& \vct{u}^*\E_\mu [(P_t\mtx{f}) \, \mtx{g}]\vct{v} \\
&\quad = \sum_{j=1}^d \E_\mu [\vct{u}^*(P_t\mtx{f})\mathbf{e}_j\mathbf{e}_j^*\mtx{g}\vct{v}] = \sum_{j=1}^d \E_\mu [(P_tf_j)\,\bar{g}_j]\\
&\quad = \sum_{j=1}^d\E_\mu \left[(P_t\real(f_j))\real(g_j) + (P_t\imag(f_j))\imag(g_j) - \iunit(P_t\real(f_j))\imag(g_j) + \iunit(P_t\imag(f_j))\real(g_j)\right]\\
&\quad = \sum_{j=1}^d\E_\mu \left[\real(f_j)(P_t\real(g_j)) + \imag(f_j)(P_t\imag(g_j)) - \iunit\real(f_j)(P_t\imag(g_j)) + \iunit\imag(f_j)(P_t\real(g_j))\right]\\
&\quad= \sum_{j=1}^d \E_\mu [f_j\,(P_t\bar{g}_j)] = \sum_{j=1}^d \E_\mu [\vct{u}^*\mtx{f}\mathbf{e}_j\mathbf{e}_j^*(P_t\mtx{g})\vct{v}]
= \vct{u}^*\E_\mu [\mtx{f} \, (P_t\mtx{g})]\vct{v}.
\end{align*}
The matrix identity \eqref{eqn:reversibility_1} follows immediately because $\vct{u},\vct{v}\in \mathbb{C}^d$
are arbitrary.
\end{proof}

\subsection{Dimension reduction}

The following lemma explains how to relate the carr{\'e} du champ operator of a matrix-valued function to the carr{\'e} du champ operators of some scalar functions. It will help us transform the scalar Poincar{\'e} inequality and the scalar Bakry--{\'E}mery criterion to their matrix equivalents.

\begin{lemma}[Dimension reduction of carr{\'e} du champ]\label{lem:dimension_reduction}
Let $(P_t)_{t \geq 0}$ be the semigroup defined in~\eqref{eqn:semigroup}.
The carr{\'e} du champ operator $\Gamma$ and the iterated carr{\'e} du champ operator $\Gamma_2$ satisfy
\begin{align}
\vct{u}^*\Gamma(\mtx{f})\vct{u} &= \sum_{j=1}^d\left(\Gamma\big(\real(\vct{u}^*\mtx{f}\mathbf{e}_j)\big) + \Gamma\big(\imag(\vct{u}^*\mtx{f}\mathbf{e}_j)\big)\right); \label{eqn:scalar_gamma}\\
\vct{u}^*\Gamma_2(\mtx{f})\vct{u} &= \sum_{j=1}^d\left(\Gamma_2\big(\real(\vct{u}^*\mtx{f}\mathbf{e}_j)\big) + \Gamma_2\big(\imag(\vct{u}^*\mtx{f}\mathbf{e}_j)\big)\right). \label{eqn:scalar_gamma2}
\end{align} 
These formulae hold for all $d \in \N$, for all suitable functions $\mtx{f}:\Omega\to \mathbb{H}_d$, and for all vectors $\vct{u}\in\mathbb{C}^d$. \end{lemma}

\begin{proof}
The definition~\eqref{eqn:definition_Gamma} of $\mL$ implies that 
\[\vct{u}^*\mL(\mtx{f})\vct{v} = \mL(\vct{u}^*\mtx{f}\vct{v}) = \mL(\real(\vct{u}^*\mtx{f}\vct{v})) + \iunit\cdot \mL(\imag(\vct{u}^*\mtx{f}\vct{v})).\] 
Introduce the scalar function $f_j := \vct{u}^*\mtx{f}\mathbf{e}_j$ for each $j=1,\dots,d$. 
Then we can use the definition~\eqref{eqn:definition_Gamma} of $\Gamma$ and formula \eqref{eqn:Ramon} to compute that 
\begin{align*}
\vct{u}^*\Gamma(\mtx{f})\vct{u} &= \frac{1}{2}\left(\vct{u}^*\mL(\mtx{f}^2)\vct{u} - \vct{u}^*\mtx{f}\mL(\mtx{f})\vct{u} - \vct{u}^*\mL(\mtx{f})\mtx{f}\vct{u}\right)\\
&= \frac{1}{2}\sum_{j=1}^d\left(\vct{u}^*\mL(\mtx{f}\mathbf{e}_j\mathbf{e}_j^*\mtx{f})\vct{u} - \vct{u}^*\mtx{f}\mathbf{e}_j\mathbf{e}_j^*\mL(\mtx{f})\vct{u} - \vct{u}^*\mL(\mtx{f})\mathbf{e}_j\mathbf{e}_j^*\mtx{f}\vct{u}\right)\\
&= \frac{1}{2}\sum_{j=1}^d\left(\mL(f_j\,\bar{f}_j) - f_j\,\mL(\bar{f}_j) - \mL(f_j)\,\bar{f}_j\right)\\
&= \frac{1}{2}\sum_{j=1}^d\left(\mL(\real(f_j)^2) + \mL(\imag(f_j)^2) - 2\real(f_j)\,\mL(\real(f_j)) - 2\imag(f_j)\,\mL(\imag(f_j)) \right)\\
&= \sum_{j=1}^d\left(\Gamma(\real(f_j)) + \Gamma(\imag(f_j))\right).
\end{align*}
This is the first identity~\eqref{eqn:scalar_gamma}.   The second identity \eqref{eqn:scalar_gamma2} follows from a similar argument based on the definition~\eqref{eqn:definition_Gamma2} of $\Gamma_2$ and the relation \eqref{eqn:scalar_gamma}.
\end{proof}

\subsection{Equivalence of scalar and matrix inequalities}
\label{sec:scalar_matrix}

In this section, we verify Proposition~\ref{prop:poincare_equiv} and Proposition~\ref{prop:BE_equiv}.
These results state that functional inequalities for the action of the semigroup~\eqref{eqn:semigroup}
on real-valued functions induce functional inequalities for its action on matrix-valued functions.

\begin{proof}[Proof of Proposition~\ref{prop:poincare_equiv}]
It is evident that the validity of the matrix Poincar\'e inequality \eqref{Poincare_inequality_matrix} for all $d \in \N$ implies the scalar Poincar\'e inequality \eqref{Poincare_inequality_scalar}, which is simply the $d = 1$ case.  For the reverse implication, we invoke formula \eqref{eqn:Ramon} to learn that
\[
\vct{u}^*\mVar_\mu[\mtx{f}]\vct{u} = \sum_{j=1}^d\left(\Var_\mu\big[\real(\vct{u}^*\mtx{f}\mathbf{e}_j)\big]+\Var_\mu\big[\imag(\vct{u}^*\mtx{f}\mathbf{e}_j)\big]\right).
\]
Moreover, we can take the expectation $\E_\mu$ of formula \eqref{eqn:scalar_gamma} to obtain 
\[
\vct{u}^*\mathcal{E}(\mtx{f})\vct{u} = \sum_{j=1}^d\left(\mathcal{E}\big(\real(\vct{u}^*\mtx{f}\mathbf{e}_j)\big)+\mathcal{E}\big(\imag(\vct{u}^*\mtx{f}\mathbf{e}_j)\big)\right).
\]
Applying the scalar Poincar{\'e} inequality \eqref{Poincare_inequality_scalar} to the real scalar functions $\real(\vct{u}^*\mtx{f}\mathbf{e}_j)$ and $\imag(\vct{u}^*\mtx{f}\mathbf{e}_j)$, we obtain \[\vct{u}^*\mVar_\mu[\mtx{f}]\vct{u} \leq \alpha\cdot \vct{u}^*\mathcal{E}(\mtx{f})\vct{u}\quad \text{for all $\vct{u}\in \mathbb{C}^d$}.\]
This immediately implies the matrix Poincar{\'e} inequality \eqref{Poincare_inequality_matrix}.
\end{proof}

\begin{proof}[Proof of Proposition~\ref{prop:BE_equiv}]
It is evident that the validity of the matrix Bakry--{\'E}mery criterion~\eqref{Bakry-Emery_criterion_matrix} for all $d \in \N$ implies the validity of the scalar criterion~\eqref{Bakry-Emery_criterion_scalar}, as we only need to set $d = 1$.
To develop the reverse implication, we use Lemma~\ref{lem:dimension_reduction} to compute that
\begin{align*}
\vct{u}^*\Gamma(\mtx{f})\vct{u} &= \sum_{j=1}^d\left(\Gamma\big(\real(\vct{u}^*\mtx{f}\mathbf{e}_j)\big) + \Gamma\big(\imag(\vct{u}^*\mtx{f}\mathbf{e}_j)\big)\right)\\
&\leq c \sum_{j=1}^d\left(\Gamma_2\big(\real(\vct{u}^*\mtx{f}\mathbf{e}_j)\big) + \Gamma_2\big(\imag(\vct{u}^*\mtx{f}\mathbf{e}_j)\big)\right)\\
&= c\cdot \vct{u}^*\Gamma_2(\mtx{f})\vct{u}.
\end{align*}
The inequality is applying \eqref{Bakry-Emery_criterion_scalar} to real scalar functions $\real(\vct{u}^*\mtx{f}\mathbf{e}_j)$ and $\imag(\vct{u}^*\mtx{f}\mathbf{e}_j)$ for each $j=1,\dots,d$. 
Since $\vct{u} \in \CC^d$ is arbitrary, we immediately obtain \eqref{Bakry-Emery_criterion_matrix}.
\end{proof}

\subsection{Derivative formulas}

A standard way to establish the equivalence between the Poincar\'e inequality and the exponential ergodicity property is by studying derivatives with respect to the time parameter $t$. The following result, extending~\cite[Lemma 2.3]{ABY20:Matrix-Poincare}, calculates the derivatives of the matrix variance and the Dirichlet form along the semigroup $(P_t)_{t\geq0}$.  The result parallels the scalar case.

\begin{lemma}[Dissipation of variance and energy] \label{lem:derivative_formula}
Let $(P_t)_{t\geq 0}$ be a Markov semigroup with stationary measure $\mu$,
infinitesimal generator $\mL$, and Dirichlet form $\mathcal{E}$.
For all suitable $\mtx{f}:\Omega\rightarrow \mathbb{H}_d$,
\begin{equation}\label{eqn:variance_derivative}
\frac{\diff{} }{\diff t}\mVar_\mu[P_t\mtx{f}] = -2\mathcal{E}(\mtx{f})\quad \text{for all $t>0$}.
\end{equation}
Moreover, if the semigroup is reversible,
\begin{equation}\label{eqn:energy_derivative}
\frac{\diff{} }{\diff t}\mathcal{E}(P_t\mtx{f}) = -2\E_\mu\big[(\mL(P_t\mtx{f}))^2\big]\quad \text{for all $t>0$}.
\end{equation}
\end{lemma}

\begin{proof}
By the definition~\eqref{eqn:matrix_variance} of the matrix variance and the stationarity property $\Expect_{\mu} P_t = \Expect_\mu$, we can calculate that 
\begin{align*}
\frac{\diff{} }{\diff t}\mVar_\mu[P_t\mtx{f}] = \frac{\diff{} }{\diff t}\big[\E_\mu (P_t\mtx{f})^2 - (\E_\mu\mtx{f})^2\big]
=\E_\mu\big[\mL(P_t\mtx{f})(P_t\mtx{f}) + (P_t\mtx{f})\mL(P_t\mtx{f})\big]
= -2\mathcal{E}(P_t\mtx{f}).
\end{align*}
The second equality above uses the derivative relation \eqref{eqn:derivative_relation} for the generator,
and the third equality is the expression~\eqref{eqn:Dirichlet_expression_1} for the Dirichlet form. 
Similarly, we can calculate that 
\begin{align*}
\frac{\diff{} }{\diff t} \mathcal{E}(P_t\mtx{f})
&= - \frac{\diff{} }{\diff t} \E_\mu\big[(P_t\mtx{f})\mL(P_t\mtx{f})\big]\\
&= - \E_\mu\big[\mL(P_t\mtx{f})\mL(P_t\mtx{f}) + (P_t\mtx{f})\mL(\mL(P_t\mtx{f}))\big]
= - 2\E_\mu\big[(\mL(P_t\mtx{f}))^2\big].
\end{align*}
The first equality is \eqref{eqn:Dirichlet_expression_2}. The last equality
holds because $\mL$ is symmetric.
\end{proof}

The matrix Poincar\'e inequality \eqref{eqn:matrix_Poincare} allows us to convert the derivative formulas
in Lemma~\ref{lem:derivative_formula} into differential inequalities for matrix-valued functions.
The next lemma gives the solution to these differential inequalities. 

\begin{lemma}[Differential matrix inequality] \label{lem:matrix_differential_inequality}
Assume that $\mtx{A}:[0,+\infty) \rightarrow \mathbb{H}_d$ is a differentiable
matrix-valued function that satisfies the differential inequality
\[\frac{\diff{} }{\diff t}\mtx{A}(t) \preccurlyeq \nu \cdot \mtx{A}(t) \quad \text{for all $t > 0$,}\]
where $\nu \in \mathbb{R}$ is a constant. Then 
\[\mtx{A}(t) \preccurlyeq \econst^{\nu t}\cdot\mtx{A}(0)\quad \text{for all $t\geq 0$}. \]
\end{lemma}

\begin{proof}
Consider the matrix-valued function $\mtx{B}(t):= \econst^{-\nu t} \mtx{A}(t)$ for $t \geq 0$.
Then $\mtx{B}(0) = \mtx{A}(0)$, and 
\[\frac{\diff{} }{\diff t}\mtx{B}(t) = \econst^{-\nu t}\frac{\diff{} }{\diff t}\mtx{A}(t) - \nu \econst^{-\nu t} \mtx{A}(t)\preccurlyeq \mtx{0}. \]
Since integration preserves the semidefinite order,
\[\econst^{-\nu t}\mtx{A}(t) = \mtx{B}(t) \preccurlyeq  \mtx{B}(0) = \mtx{A}(0).\]
Multiply by $\econst^{\nu t}$ to arrive at the stated result.
\end{proof}

\subsection{Consequences of the Poincar\'e inequality} 
\label{sec:equivalence_Poincare}

This section contains the proof of Proposition~\ref{prop:matrix_poincare}, the equivalence between the matrix Poincar\'e inequality and exponential ergodicity properties. This proof is adapted from its scalar analog~\cite[Theorem 2.18]{van550probability}.

\begin{proof}[Proof of Proposition~\ref{prop:matrix_poincare}]
\textit{Proof that \eqref{Poincare_inequality} $\Rightarrow$ \eqref{variance_convergence}.}
To see that the matrix Poincar{\'e} inequality~\eqref{Poincare_inequality}
implies exponential ergodicity~\eqref{variance_convergence} of the variance,
combine Lemma~\ref{lem:derivative_formula} with the matrix Poincar{\'e}
inequality to obtain a differential inequality:
\[\frac{\diff{} }{\diff t}\mVar_\mu[P_t\mtx{f}] = -2\mathcal{E}(P_t\mtx{f}) \preccurlyeq -\frac{2}{\alpha} \mVar_\mu[P_t\mtx{f}].\]
Lemma~\ref{lem:matrix_differential_inequality} gives the solution: 
\[\mVar_\mu[P_t\mtx{f}] \preccurlyeq \econst^{-2t/\alpha} \mVar_\mu[P_0\mtx{f}] = \econst^{-2t/\alpha} \mVar_\mu[\mtx{f}].\]
This is the ergodicity of the variance.

\textit{Proof that \eqref{variance_convergence} $\Rightarrow$ \eqref{Poincare_inequality}.}
To obtain the matrix Poincar{\'e} inequality~\eqref{Poincare_inequality} from exponential ergodicity~\eqref{variance_convergence}
of the variance, use the derivative~\eqref{eqn:variance_derivative} of the variance
and the fact that $P_0$ is the identity map to see that
\[\mathcal{E}(\mtx{f}) = \lim_{t\downarrow0} \frac{\mVar_\mu[\mtx{f}]- \mVar_\mu[P_t\mtx{f}]}{2t} \succcurlyeq \lim_{t\downarrow0} \frac{1-\econst^{-2t/\alpha}}{2t} \cdot \mVar_\mu[\mtx{f}] = \frac{1}{\alpha} \mVar_\mu[\mtx{f}].\]
The inequality follows from \eqref{variance_convergence}.

\textit{Proof that \eqref{Poincare_inequality} $\Rightarrow$ \eqref{energy_convergence} under reversibility.}
Next, we argue that the matrix Poincar\'e inequality \eqref{Poincare_inequality}
implies exponential ergodicity~\eqref{energy_convergence} of the energy, assuming that the semigroup is reversible.
In this case, the zero-mean property \eqref{eqn:mean_zero} implies that
$\E_\mu[\mtx{g}\mL(\mtx{f})] = \E_\mu[(\mtx{g}-\E_\mu\mtx{g})\mL(\mtx{f})]$ and $\E_\mu[\mL(\mtx{f})\mtx{g}] = \E_\mu[\mL(\mtx{f})(\mtx{g}-\E_\mu\mtx{g})]$ for all suitable $\mtx{f},\mtx{g}$.  Therefore,
\begin{align*}
\mathcal{E}(\mtx{f}) &= - \frac{1}{2}\E_\mu\left[\mtx{f}\mL(\mtx{f})+\mL(\mtx{f})\mtx{f}\right] = - \frac{1}{2}\E_\mu\left[(\mtx{f}-\E_\mu\mtx{f})\mL(\mtx{f}) + \mL(\mtx{f})(\mtx{f}-\E_\mu\mtx{f})\right]\\
&\preccurlyeq \frac{1}{2\alpha}\E_\mu \big[(\mtx{f}-\E_\mu\mtx{f})^2\big]  + \frac{\alpha}{2}\E_\mu \big[\mL(\mtx{f})^2\big] \preccurlyeq \frac{1}{2} \mathcal{E}(\mtx{f}) + \frac{\alpha}{2}\E_\mu \big[\mL(\mtx{f})^2\big].
\end{align*}
The first inequality holds because $\mtx{A}\mtx{B} + \mtx{B}\mtx{A}\preccurlyeq \mtx{A}^2+\mtx{B}^2$ for all $\mtx{A},\mtx{B}\in \mathbb{H}_d$, and the second follows from the matrix Poincar{\'e} inequality~\eqref{Poincare_inequality}.  Rearranging, we obtain the relation $\mathcal{E}(\mtx{f})\preccurlyeq \alpha \E_\mu [\mL(\mtx{f})^2]$ for all suitable $\mtx{f}$.  Combine this fact with the derivative formula \eqref{eqn:energy_derivative} to reach
\[\frac{\diff{} }{\diff t} \mathcal{E}(P_t\mtx{f}) = - 2\E_\mu\big[\mL(P_t\mtx{f})^2\big] \preccurlyeq - \frac{2}{\alpha} \mathcal{E}(P_t\mtx{f}).\]
Lemma~\ref{lem:matrix_differential_inequality} gives the solution to the differential inequality:
\[\mathcal{E}(P_t\mtx{f})\preccurlyeq \econst^{-2t/\alpha} \mathcal{E}(P_0\mtx{f}) = \econst^{-2t/\alpha} \mathcal{E}(\mtx{f}).\]
This is the ergodicity of energy.

\textit{Proof that \eqref{energy_convergence} $\Rightarrow$ \eqref{Poincare_inequality} under ergodicity.}
To see that exponential ergodicity~\eqref{energy_convergence} of the energy implies the matrix Poincar\'e inequality \eqref{Poincare_inequality} when the semigroup is ergodic, we combine \eqref{energy_convergence} with the derivative~\eqref{eqn:variance_derivative}
of the Dirichlet form to obtain
\[\frac{\diff{} }{\diff t}\mVar_\mu[P_t\mtx{f}] = -2\mathcal{E}(P_t\mtx{f}) \succcurlyeq -2\econst^{-2t/\alpha}\mathcal{E}(\mtx{f}).\]
Using the ergodicity assumption~\eqref{eqn:ergodicity} on the semigroup, we have 
\begin{align*} 
\mVar_\mu[\mtx{f}] &= \mVar_\mu[P_0\mtx{f}] - \lim_{t\rightarrow\infty}\mVar_\mu[P_t\mtx{f}] = -\int_0^\infty \frac{\diff{} }{\diff t}\mVar_\mu[P_t\mtx{f}] \idiff t \\
&\preccurlyeq 2\int_0^\infty\econst^{-2t/\alpha} \idiff t \cdot \mathcal{E}(\mtx{f}) = \alpha \mathcal{E}(\mtx{f}).
\end{align*}
The first equality follows from the ergodicity relation 
\[\lim_{t\rightarrow\infty}\mVar_\mu[P_t\mtx{f}] = \lim_{t\rightarrow\infty}\E_\mu(P_t\mtx{f}-\E_\mu\mtx{f})^2 = \mtx{0}.\]
This completes the proof of Proposition~\ref{prop:matrix_poincare}.
\end{proof}

\subsection{Equivalence result for local Poincar\'e inequality}
\label{sec:equivalence_local_Poincare}
 
Proposition~\ref{prop:local_Poincare} states that the matrix Bakry--\'Emery criterion, the local Poincar\'e inequality, and the local ergodicity of the carr\'e du champ operator are equivalent with each other.  This section is dedicated to the proof, which is modeled on the scalar argument~\cite[Theorem 2.36]{van550probability}.  

\begin{proof}[Proof of Proposition~\ref{prop:local_Poincare}]
\textit{Proof that \eqref{Bakry-Emery_criterion} $\Rightarrow$ \eqref{local_ergodicity}.}
Let us show that the matrix Bakry--\'Emery criterion \eqref{Bakry-Emery_criterion} implies local ergodicity~ \eqref{local_ergodicity} of the carr\'e du champ operator.  Given any suitable $\mtx{f}$ and any $t\geq 0$, construct the function $\mtx{A}(s) := P_{t-s}\Gamma(P_s\mtx{f})$ for $s\in[0,t]$. Then we have
\begin{align*}
\frac{\diff{} }{\diff s} \mtx{A}(s) &= - \mL P_{t-s} \Gamma(P_s\mtx{f}) + P_{t-s}\Gamma(\mL P_s\mtx{f},P_s\mtx{f}) + P_{t-s}\Gamma(P_s\mtx{f},\mL P_s\mtx{f}) \\
&= - P_{t-s}\big( \mL \Gamma(P_s\mtx{f}) - \Gamma(\mL P_s\mtx{f},P_s\mtx{f}) -\Gamma(P_s\mtx{f},\mL P_s\mtx{f})\big) \\
&= -2 P_{t-s} \Gamma_2(P_s\mtx{f})\\
&\preccurlyeq -2c^{-1} P_{t-s}\Gamma(P_s\mtx{f})\\
&= -2c^{-1} \mtx{A}(s).
\end{align*}
The inequality follows from \eqref{Bakry-Emery_criterion}.   Apply Lemma~\ref{lem:matrix_differential_inequality} to reach to bound $\mtx{A}(t)\preccurlyeq \econst^{-2t/c} \mtx{A}(0)$.  This yields \eqref{local_ergodicity} because $\mtx{A}(t) = \Gamma(P_t\mtx{f})$ and $\mtx{A}(0) = P_t\Gamma(\mtx{f})$.

\textit{Proof that \eqref{local_ergodicity} $\Rightarrow$ \eqref{local_Poincare}.}
Next, we argue that local ergodicity of the carr\'e du champ operator \eqref{local_ergodicity} implies the local matrix Poincar\'e inequality \eqref{local_Poincare}.  Construct the function $\mtx{B}(s) := P_{t-s}((P_s\mtx{f})^2)$ for $s\in[0,t]$. Taking the derivative with respect to $s$ gives 
\begin{align*}
\frac{\diff{} }{\diff s} \mtx{B}(s) =&\ - \mL P_{t-s} ((P_s\mtx{f})^2) + P_{t-s}(\mL(P_s\mtx{f})P_s\mtx{f}) + P_{t-s}(P_s\mtx{f}\mL(P_s\mtx{f})) \\
=&\ - P_{t-s}\left( \mL ((P_s\mtx{f})^2) - \mL(P_s\mtx{f})P_s\mtx{f} - P_s\mtx{f}\mL(P_s\mtx{f})\right) \\
=&\ -2 P_{t-s} \Gamma(P_s\mtx{f})\\
\succcurlyeq&\  -2\econst^{-2s/c} P_{t-s}P_s\Gamma(\mtx{f})\\
=&\ -2\econst^{-2s/c} P_t\Gamma(\mtx{f}).
\end{align*}
Therefore,
\[P_t(\mtx{f}^2) - (P_t\mtx{f})^2 = \mtx{B}(0) - \mtx{B}(t) \preccurlyeq 2\int_0^t\econst^{-2s/c}\idiff s \cdot P_t\Gamma(\mtx{f}) = c \, (1-\econst^{-2t/c})\,P_t\Gamma(\mtx{f}).  \]
This is the local ergodicity property.

\textit{Proof that \eqref{local_Poincare} $\Rightarrow$ \eqref{Bakry-Emery_criterion}.}
Last, we show that the local matrix Poincar\'e inequality \eqref{local_Poincare} implies the matrix Bakry--\'Emery criterion \eqref{Bakry-Emery_criterion}.  Construct the function $\mtx{C}(t) := P_t(\mtx{f}^2) - (P_t\mtx{f})^2 - c\,(1-\econst^{-2t/c})\,P_t\Gamma(\mtx{f})$. Evidently, $\mtx{C}(0) = \mtx{0}$, and the local Poincar{\'e} inequality \eqref{local_Poincare} implies that $\mtx{C}(t)\preccurlyeq \mtx{0}$ for all $t\geq 0$.  Now, the first derivative satisfies 
\begin{align*}
\frac{\diff{} }{\diff t}\bigg|_{t=0} \mtx{C}(t)
	= \mL(\mtx{f}^2) -\mL(\mtx{f})\mtx{f} - \mtx{f}\mL(\mtx{f}) - 2\Gamma(\mtx{f}) = \mtx{0}.
\end{align*}
The second derivative takes the form
\begin{align*}
\frac{\diff{^2}}{\diff t^2}\bigg|_{t=0} \mtx{C}(t) &= \mL^2(\mtx{f}^2)-\mL^2(\mtx{f})\mtx{f} - \mtx{f}\mL^2(\mtx{f}) - 2(\mL\mtx{f})^2 + 4c^{-1}\Gamma(\mtx{f}) - 4\mL\Gamma(\mtx{f}) \\
&= 4c^{-1}\left(\Gamma(\mtx{f}) -  c\Gamma_2(\mtx{f})\right).
\end{align*}
Therefore,
\[\Gamma(\mtx{f}) -  c\Gamma_2(\mtx{f}) = \frac{c}{4}\frac{\diff{^2}}{\diff t^2} \bigg|_{t=0} \mtx{C}(t) = \frac{c}{2}\lim_{t\rightarrow 0}\frac{\mtx{C}(t)}{t^2} \preccurlyeq \mtx{0}.\]
This verifies the validity of the matrix Bakry--\'Emery criterion with constant $c$.
\end{proof}

\section{From curvature conditions to matrix moment inequalities}
\label{sec:trace_to_moment}

The main results of this paper, Theorems~\ref{thm:polynomial_moment} and~\ref{thm:exponential_moment},
demonstrate that the Bakry--{\'E}mery criterion~\eqref{Bakry-Emery} leads to
trace moment inequalities for random matrices.
This section is dedicated to the proofs of these theorems.
These arguments appear to be new, even in the scalar setting,
but see~\cite{Led92:Heat-Semigroup,Sch99:Curvature-Nonlocal}
for some precedents.

\subsection{Overview}

Let $(P_t)_{t \geq 0}$ be a reversible, ergodic semigroup
acting on matrix-valued functions. Assume that the semigroup satisfies a Bakry--{\'E}mery criterion~\eqref{Bakry-Emery},
so Proposition~\ref{prop:local_Poincare} implies that it is locally ergodic.
Without loss of generality, we may assume that the matrix-valued function $\mtx{f}$
is zero-mean: $\E_\mu\mtx{f}=\mtx{0}$.

For a standard matrix function $\phi$, the basic idea is to estimate a trace moment
of the form $\E_\mu\trace[\mtx{f}\,\varphi(\mtx{f})]$ via a classic semigroup argument:
\[\E_\mu\trace[\mtx{f}\,\varphi(\mtx{f})] =\E_\mu\trace[P_0(\mtx{f})\,\varphi(\mtx{f})] = \lim_{t\rightarrow\infty}\E_\mu\trace[P_t(\mtx{f})\,\varphi(\mtx{f})] - \int_0^\infty\frac{\diff{}}{\diff t}\E_\mu\trace[P_t(\mtx{f})\,\varphi(\mtx{f})]\idiff t.\]
By ergodicity \eqref{eqn:ergodicity}, $\lim_{t\rightarrow\infty}\E_\mu\trace[P_t(\mtx{f})\,\varphi(\mtx{f})] = \E_\mu\trace[(\E_\mu\mtx{f})\,\varphi(\mtx{f})] = 0$.
In the second term on the right-hand side, the time derivative places the infinitesimal generator $\mL$ in the integrand,
which then becomes 
\begin{equation} \label{eqn:overview_gamma}
-\E_\mu\trace[\mL(P_t\mtx{f})\,\varphi(\mtx{f})] = \E_\mu\trace \Gamma(P_t\mtx{f},\varphi(\mtx{f})). 
\end{equation}
This familiar formula is the starting point for our method.

To control the trace of the carr{\'e} du champ, we employ the following fundamental lemma,
which is related to the Stroock--Varopoulos
inequality~\cite{Str84:Introduction-Theory,Var85:Hardy-Littlewood-Theory}.

\begin{lemma}[Chain rule inequality]\label{lem:key_Gamma}
Let $\varphi:\mathbb{R}\rightarrow\mathbb{R}$ be a function such that $\psi := |\varphi'|$ is convex.
For all suitable $\mtx{f},\mtx{g}:\Omega\rightarrow \mathbb{H}_d$, 
\begin{equation*}\label{eqn:general_Gamma}
\E_\mu \trace\Gamma(\mtx{g},\varphi(\mtx{f}))\leq \Big(\E_\mu\trace \left[\Gamma(\mtx{f})\,\psi(\mtx{f})\right]\cdot\E_\mu\trace \left[\Gamma(\mtx{g})\,\psi(\mtx{f})\right]\Big)^{1/2}.
\end{equation*}
\end{lemma}

\noindent
The proof of this lemma appears below in Section~\ref{sec:key_lemma}.

Lemma~\ref{lem:key_Gamma} isolates the contributions
from the matrix $P_t \mtx{f}$ and the matrix $\phi(\mtx{f})$
in the formula~\eqref{eqn:overview_gamma}.
To estimate $\Gamma(P_t \mtx{f})$,
we invoke the local ergodicity property,
Proposition~\ref{prop:local_Poincare}\eqref{local_ergodicity}.
Last, we apply the matrix decoupling techniques, based on H{\"o}lder and Young trace inequalities, 
to bound $\E\trace \left[\Gamma(\mtx{f})\,\psi(\mtx{f})\right]$ and $\E\trace \left[\Gamma(P_t\mtx{f})\,\psi(\mtx{f})\right]$
in terms of the original quantity of interest $\E_{\mu}\trace[\mtx{f} \, \phi(\mtx{f})]$.
The following sections supply full details.

Our approach incorporates some techniques and ideas from~\cite[Theorems 4.2 and 4.3]{paulin2016efron},
but the argument is distinct.  Appendix~\ref{apdx:Stein_method} gives more details about the connection.

\subsection{Proof of chain rule inequality} 
\label{sec:key_lemma}

To prove Lemma~\ref{lem:key_Gamma}, we require a novel trace inequality.

\begin{lemma}[Mean value trace inequality]\label{lem:mean_value_inequality} Let $\varphi:\mathbb{R}\rightarrow\mathbb{R}$ be a function such that $\psi:= |\varphi'|$ is convex. For all $\mtx{A},\mtx{B},\mtx{C}\in\mathbb{H}_d$, 
\[\trace\left[\mtx{C} \, \big(\varphi(\mtx{A})-\varphi(\mtx{B})\big)\right]\leq \inf_{s>0} \frac{1}{4}\trace\left[\left(s\,(\mtx{A}-\mtx{B})^2+s^{-1}\,\mtx{C}^2\right)\big(\psi(\mtx{A})+\psi(\mtx{B})\big)\right].\]
\end{lemma}

Lemma~\ref{lem:mean_value_inequality} is a common generalization
of \cite[Lemmas 9.2 and 12.2]{paulin2016efron}. Roughly speaking, it exploits convexity to bound
the difference $\varphi(\mtx{A})-\varphi(\mtx{B})$
in the spirit of the mean value theorem.
We defer the proof of Lemma~\ref{lem:mean_value_inequality}
to Appendix~\ref{apdx:mean_value}. 

\begin{proof}[Proof of Lemma~\ref{lem:key_Gamma} from Lemma~\ref{lem:mean_value_inequality}]
For simplicity, we abbreviate 
\[\mtx{f}_t = \mtx{f}(Z_t),\quad \mtx{g}_t = \mtx{g}(Z_t) \quad\text{and}\quad \mtx{f}_0 = \mtx{f}(Z_0), \quad \mtx{g}_0 = \mtx{g}(Z_0).\]
By Proposition~\ref{prop:Gamma_property}\eqref{limit_formula},
\begin{equation}\label{step:key_lemma_1}
\begin{split}
\E_\mu \trace\Gamma(\mtx{g},\varphi(\mtx{f})) =&\ \E_{Z\sim\mu}\trace\lim_{t\downarrow 0}\frac{1}{2t} \E\left[\left(\mtx{g}_t-\mtx{g}_0\right)\left(\varphi(\mtx{f}_t)-\varphi(\mtx{f}_0)\right)\,\big|\,Z_0=Z\right]\\
=&\ \E_{Z\sim\mu}\lim_{t\downarrow 0}\frac{1}{2t} \E \left[\trace\left[\left(\mtx{g}_t-\mtx{g}_0\right)\left(\varphi(\mtx{f}_t)-\varphi(\mtx{f}_0)\right)\right]\,\big|\,Z_0=Z\right].
\end{split}
\end{equation}
Fix a parameter $s > 0$.  For each $t > 0$,  the mean value trace inequality,
Lemma~\ref{lem:mean_value_inequality}, yields
\begin{equation}\label{step:key_lemma_2}
\begin{split}
\trace \left[\left(\mtx{g}_t-\mtx{g}_0\right)\big(\varphi(\mtx{f}_t)-\varphi(\mtx{f}_0)\big)\right]
&\leq \frac{1}{4}\trace\left[\left(s\,(\mtx{f}_t-\mtx{f}_0)^2+s^{-1}\,(\mtx{g}_t-\mtx{g}_0)^2\right)\big(\psi(\mtx{f}_t)+\psi(\mtx{f}_0)\big)\right]\\
&= \frac{1}{2} \trace\left[\left(s\,(\mtx{f}_t-\mtx{f}_0)^2+s^{-1}\,(\mtx{g}_t-\mtx{g}_0)^2\right)\psi(\mtx{f}_0)\right]\\
&\qquad  + \frac{1}{4} \trace\left[\left(s(\mtx{f}_t-\mtx{f}_0)^2+s^{-1}\,(\mtx{g}_t-\mtx{g}_0)^2\right)\big(\psi(\mtx{f}_t)-\psi(\mtx{f}_0)\big)\right].
\end{split}
\end{equation}
It follows from the triple product result, Lemma~\ref{lem:three_limit}, that 
the second term satisfies
\begin{equation}\label{step:key_lemma_3}
\E_{Z\sim\mu}\lim_{t\downarrow 0}\frac{1}{t} \trace\E \left[ \left(s\,(\mtx{f}_t-\mtx{f}_0)^2+s^{-1}\,(\mtx{g}_t-\mtx{g}_0)^2\right)\big(\psi(\mtx{f}_t)-\psi(\mtx{f}_0)\big) \,\big|\,Z_0=Z\right] =0.
\end{equation}
Sequence the displays \eqref{step:key_lemma_1},\eqref{step:key_lemma_2} and \eqref{step:key_lemma_3} to reach 
\begin{align*}
\E_\mu \trace\Gamma(\mtx{g},\varphi(\mtx{f}))&\leq \frac{1}{2}\E_{Z\sim\mu}\lim_{t\downarrow 0} \frac{1}{2t} \trace \E\left[\left(s\,(\mtx{f}_t-\mtx{f}_0)^2+s^{-1}\,(\mtx{g}_t-\mtx{g}_0)^2\right)\psi(\mtx{f}_0) \,\big|\,Z_0=Z\right] \\
&= \frac{1}{2}\E_{Z\sim\mu}\trace\Big[\Big(s\,\lim_{t\downarrow0}\frac{1}{2t} \E[(\mtx{f}_t-\mtx{f}_0)^2\,|\,Z_0=Z] \\
&\qquad\qquad\qquad\ + s^{-1}\,\lim_{t\downarrow0}\frac{1}{2t} \E[(\mtx{g}_t-\mtx{g}_0)^2\,|\,Z_0=Z] \Big)\psi(\mtx{f}(Z))\Big] \\
&= \frac{1}{2} \E_\mu\trace \left[\left(s\,\Gamma(\mtx{f}) +s^{-1}\,\Gamma(\mtx{g})\right)\,\psi(\mtx{f})\right].
\end{align*}
The last relation is Proposition~\ref{prop:Gamma_property}\eqref{limit_formula}.
Minimize the right-hand side over $s\in(0,\infty)$ to arrive at 
\[\E_\mu \trace\Gamma(\mtx{g},\varphi(\mtx{f}))\leq \big(\E_\mu\trace \left[\Gamma(\mtx{f})\,\psi(\mtx{f})\right]\big)^{1/2}\cdot\big(\E_\mu\trace \left[\Gamma(\mtx{g})\,\psi(\mtx{f})\right]\big)^{1/2}.\]
This completes the proof of Lemma~\ref{lem:key_Gamma}.
\end{proof}

\subsection{Polynomial moments}
\label{sec:polynomial_moments_proof}

This section is dedicated to the proof of Theorem~\ref{thm:polynomial_moment}, which states that the Bakry--{\'E}mery criterion implies matrix polynomial moment bounds.

\subsubsection{Setup}

Consider a reversible, ergodic Markov semigroup $(P_t)_{t \geq 0}$
with stationary measure $\mu$.
Assume that the semigroup satisfies the Bakry--{\'E}mery criterion~\eqref{Bakry-Emery}
with constant $c > 0$.
By Proposition~\ref{prop:local_Poincare}, this is equivalent to local ergodicity.

Fix a suitable function $\mtx{f}:\Omega \rightarrow \mathbb{H}_d$.
Proposition~\ref{prop:Gamma_property}\eqref{limit_formula} implies that the carr{\'e}
du champ is shift invariant.  In particular, $\Gamma(\mtx{f}) = \Gamma(\mtx{f}-\E_\mu\mtx{f})$.
Therefore, we may assume that $\E_\mu\mtx{f}=\mtx{0}$.

The quantity of interest is
\[
\E_\mu\trace |\mtx{f}|^{2q}
	= \E_\mu\trace \left[\mtx{f}\cdot \sgn(\mtx{f})\cdot |\mtx{f}|^{2q-1}\right]
	=: \E_{\mu} \trace \left[ \mtx{f} \, \varphi(\mtx{f}) \right].
\]
We have introduced the signed moment function $\varphi: x\mapsto \sgn(x)\cdot \abs{x}^{2q-1}$
for $x \in \R$.  Note that the absolute derivative $\psi(x) := \abs{\varphi'(x)} = (2q-1)\abs{x}^{2q-2}$
is convex when $q= 1$ or when $q\geq 1.5$.

\begin{remark}[Missing powers]
A similar argument holds when $q \in (1, 1.5)$.
It requires a variant of Lemma~\ref{lem:key_Gamma}
that holds for monotone $\psi$, but has an
extra factor of $2$ on the right-hand side.
\end{remark}

\subsubsection{A Markov semigroup argument}

By the ergodicity assumption \eqref{eqn:ergodicity}, it holds that 
\[\lim_{t\rightarrow\infty}\E_\mu\trace[P_t(\mtx{f})\,\varphi(\mtx{f})] = \E_\mu\trace[(\E_\mu\mtx{f})\,\varphi(\mtx{f})] = 0.\]
Therefore,
\begin{equation}\label{step:polynomial_1}
\begin{split}
\E_\mu\trace \abs{\mtx{f}}^{2q} &= \E_\mu\trace \left[P_0\mtx{f}\, \varphi(\mtx{f})\right] - \lim_{t\rightarrow\infty}\E_\mu\trace[P_t(\mtx{f})\,\varphi(\mtx{f})]\\
&= -\int_0^\infty\frac{\diff{} }{\diff t}\E_\mu\trace\left[(P_t\mtx{f}) \, \varphi(\mtx{f})\right]\idiff t = -\int_0^\infty\E_\mu\trace\left[\mL(P_t\mtx{f}) \, \varphi(\mtx{f})\right]\idiff t.
\end{split}
\end{equation}
By convexity of $\psi$, we can invoke the chain rule inequality, Lemma~\ref{lem:key_Gamma}, to obtain 
\begin{equation}\label{step:polynomial_2}
\begin{split}
- \E_\mu\trace\left[\mL(P_t\mtx{f})\, \varphi(\mtx{f})\right] =&\ \E_\mu\trace\Gamma(P_t\mtx{f},\varphi(\mtx{f}))\\
\leq&\ \left(\E_\mu\trace\left[\Gamma(\mtx{f})\,\psi(\mtx{f}) \right]\cdot \E_\mu\trace\left[\Gamma(P_t\mtx{f})\,\psi(\mtx{f})\right]\right)^{1/2} \\
=&\ (2q-1)\left(\E_\mu\trace\left[\Gamma(\mtx{f})\abs{\mtx{f}}^{2q-2} \right]\cdot \E_\mu\trace\left[\Gamma(P_t\mtx{f})\abs{\mtx{f}}^{2q-2}\right]\right)^{1/2}\\
\leq&\ (2q-1)\,\econst^{-t/c}\left(\E_\mu\trace\left[\Gamma(\mtx{f})\abs{\mtx{f}}^{2q-2} \right]\cdot \E_\mu\trace\left[(P_t\Gamma(\mtx{f}))\abs{\mtx{f}}^{2q-2}\right]\right)^{1/2}.
\end{split}
\end{equation}
The last inequality is the local ergodicity condition,
Proposition~\ref{prop:local_Poincare}\eqref{local_ergodicity}.

\subsubsection{Decoupling}

Apply H\"older's inequality for the trace followed by H\"older's inequality for the expectation
to obtain 
\begin{equation} \label{step:polynomial_2.5}
\begin{aligned}
\E_\mu\trace\left[\Gamma(\mtx{f}) \abs{\mtx{f}}^{2q-2} \right] &\leq \left(\E_\mu\trace \Gamma(\mtx{f})^q \right)^{1/q}\cdot \left(\E_\mu\trace|\mtx{f}|^{2q}\right)^{(q-1)/q} \quad\text{and} \\
\E_\mu\trace\left[(P_t\Gamma(\mtx{f})) \abs{\mtx{f}}^{2q-2}\right] &\leq  \left(\E_\mu\trace{} (P_t\Gamma(\mtx{f}))^q \right)^{1/q}\cdot \big(\E_\mu\trace \abs{\mtx{f}}^{2q}\big)^{(q-1)/q}.
\end{aligned}
\end{equation}
Introduce the bounds~\eqref{step:polynomial_2.5} into \eqref{step:polynomial_2} to find that
\begin{equation}\label{step:polynomial_3}
\begin{split}
& - \E_\mu\trace\left[\mL(P_t\mtx{f}) \,\varphi(\mtx{f})\right] \\
&\qquad\qquad \leq (2q-1)\,\econst^{-t/c}\left(\E_\mu\trace \Gamma(\mtx{f})^q \cdot \E_\mu\trace{}(P_t\Gamma(\mtx{f}))^q \right)^{1/(2q)} \big(\E_\mu\trace\abs{\mtx{f}}^{2q}\big)^{(q-1)/q}.
\end{split}
\end{equation}
Substitute \eqref{step:polynomial_3} into \eqref{step:polynomial_1} and rearrange the expression to reach
\begin{equation}\label{step:polynomial_4}
\big(\E_\mu\trace \abs{\mtx{f}}^{2q}\big)^{1/q}\leq (2q-1)\left(\E_\mu\trace \Gamma(\mtx{f})^q \right)^{1/(2q)} \int_0^{\infty}  \econst^{-t/c}\left(\E_\mu\trace{} (P_t\Gamma(\mtx{f}))^q\right)^{1/(2q)}\idiff t.
\end{equation}
It remains to remove the semigroup from the integral.

\subsubsection{Endgame}

The trace power $\trace[ (\cdot)^q ]$ is convex on $\mathbb{H}_d$ for $q\geq1$;  see~\cite[Theorem 2.10]{carlen2010trace}.
Therefore, the Jensen inequality \eqref{eqn:semigroup_Jensen_2} for the semigroup implies that
\begin{equation}\label{step:polynomial_5}
\E_\mu\trace{} (P_t\Gamma(\mtx{f}))^q \leq \E_\mu \trace \Gamma(\mtx{f})^q.
\end{equation} 
Substituting \eqref{step:polynomial_5} into \eqref{step:polynomial_4} yields
\[\big(\E_\mu\trace \abs{\mtx{f}}^{2q}\big)^{1/q}\leq (2q-1) \left(\E_\mu\trace \Gamma(\mtx{f})^q \right)^{1/q} \int_0^\infty \econst^{-t/c} \idiff t = c \, (2q-1)\left(\E_\mu\trace\Gamma(\mtx{f})^q\right)^{1/q}.\]
This establishes \eqref{eqn:polynomial_moment_1}.

Define the uniform bound $v_{\mtx{f}} := \norm{ \norm{ \Gamma(\mtx{f}) } }_{L_{\infty}(\mu)}$. 
We have the further estimate
\[\left(\E_\mu\trace\left[\Gamma(\mtx{f})^q\right]\right)^{1/(2q)}\leq d^{1/(2q)} \sqrt{v_{\mtx{f}}}.\]
The statement \eqref{eqn:polynomial_moment_2} now follows from \eqref{eqn:polynomial_moment_1}.
This step completes the proof of Theorem~\ref{thm:polynomial_moment}.

\subsection{Exponential moments} 
\label{sec:exponential_moments_proof}

In this section, we establish Theorem~\ref{thm:exponential_concentration},
the exponential matrix concentration inequality.  The main technical ingredient
is a bound on exponential moments:

\begin{theorem}[Exponential moments]\label{thm:exponential_moment}
Instate the hypotheses of Theorem~\ref{thm:exponential_concentration}.
For all $\theta\in(-\sqrt{\beta/c},\sqrt{\beta/c})$, 
\begin{equation}\label{eqn:exponential_moment_1}
\log\E_\mu \ntr \econst^{\theta(\mtx{f}-\E_\mu\mtx{f})} \leq \frac{c\theta^2 r_{\mtx{f}}(\beta)}{2(1-c\theta^2/\beta)}.
\end{equation}
Moreover, if $v_{\mtx{f}} <+\infty$, then 
\begin{equation}\label{eqn:exponential_moment_2}
\log\E_\mu \ntr \econst^{\theta(\mtx{f}-\E_\mu\mtx{f})} \leq \frac{cv_{\mtx{f}}\theta^2}{2}
\quad\text{for all $\theta \in \R$.}
\end{equation}
\end{theorem}

\noindent
The proof of Theorem~\ref{thm:exponential_moment} occupies the rest of this subsection.
Afterward, in Section~\ref{sec:exponential_concentration_proof}, we derive
Theorem~\ref{thm:exponential_concentration}.

\subsubsection{Setup}

As usual, we consider a reversible, ergodic Markov semigroup $(P_t)_{t \geq 0}$
with stationary measure $\mu$.
Assume that the semigroup satisfies the Bakry--{\'E}mery criterion~\eqref{Bakry-Emery}
for a constant $c > 0$, so it is locally ergodic.

Choose a suitable function $\mtx{f}:\Omega \rightarrow \mathbb{H}_d$.
We may assume that $\E_\mu\mtx{f}=\mtx{0}$.  Furthermore, we only
need to consider the case $\theta \geq 0$.  The results for $\theta < 0$
follow formally under the change of variables $\theta \mapsto - \theta$
and $\mtx{f} \mapsto - \mtx{f}$.

The quantity of interest is the normalized trace mgf:
\[m(\theta) := \E_\mu \ntr \econst^{\theta\mtx{f}}\quad \text{for}\ \theta\geq0.\]
We will bound the derivative of this function:
\[
m'(\theta) = \E_\mu \ntr \left[\mtx{f} \, \econst^{\theta\mtx{f}}\right]
	=: \E_{\mu} \ntr [ \mtx{f} \, \varphi(\mtx{f}) ].
\]
We have introduced the function $\varphi : x \mapsto \econst^{\theta x}$ for $x \in \R$.
Note that its absolute derivative $\psi(x) := \abs{ \varphi'(x) } = \theta \econst^{\theta x}$
is a convex function, since $\theta \geq 0$.
Here and elsewhere, we use the properties of the trace mgf that are
collected in Lemma~\ref{prop:trace_mgf}. 

\subsubsection{A Markov semigroup argument}

By the ergodicity assumption \eqref{eqn:ergodicity}, we have
\begin{equation}\label{step:exponential_1}
\begin{split}
m'(\theta) &= \E_\mu \ntr \left[P_0\mtx{f}\,\econst^{\theta\mtx{f}}\right] - \lim_{t\rightarrow\infty}\E_\mu\trace\left[P_t(\mtx{f})\,\econst^{\theta\mtx{f}}\right] \\
&= -\int_0^\infty \frac{\diff{}}{\diff{t}}\E_\mu \ntr \left[P_t(\mtx{f})\,\econst^{\theta\mtx{f}}\right]\diff t = -\int_0^\infty \E_\mu \ntr \left[\mL(P_t\mtx{f})\,\econst^{\theta\mtx{f}}\right]\diff t.
\end{split}
\end{equation}
Invoke the chain rule inequality, Lemma~\ref{lem:key_Gamma}, to obtain 
\begin{equation}\label{step:exponential_2}
\begin{split}
-\E_\mu \ntr \left[\mL(P_t\mtx{f})\,\econst^{\theta\mtx{f}}\right] &= \E_\mu \ntr \Gamma(P_t\mtx{f},\econst^{\theta\mtx{f}})\\
&\leq \theta \left(\E_\mu\ntr \left[\Gamma(\mtx{f})\,\econst^{\theta\mtx{f}}\right]\cdot \E_\mu\ntr \left[\Gamma(P_t\mtx{f})\,\econst^{\theta\mtx{f}}\right]\right)^{1/2} \\
&\leq \theta \econst^{-t/c}\left(\E_\mu\ntr \left[\Gamma(\mtx{f})\,\econst^{\theta\mtx{f}}\right]\cdot \E_\mu\ntr \left[(P_t\Gamma(\mtx{f}))\, \econst^{\theta\mtx{f}}\right]\right)^{1/2}.
\end{split}
\end{equation}
The second inequality is the local ergodicity condition, Proposition~\ref{prop:local_Poincare}\eqref{local_ergodicity}.

\subsubsection{Decoupling}

The next step is to use an entropy inequality to separate the carr\'e du champ operator in \eqref{step:exponential_2} from the matrix exponential. The following trace inequality appears as \cite[Proposition A.3]{mackey2014}; see also \cite[Theorem 2.13]{carlen2010trace}.

\begin{fact}(Young's inequality for matrix entropy)\label{lem:Young_inequality} Let $\mtx{X}$ be a random matrix in $\mathbb{H}_d$, and let $\mtx{Y}$ be a random matrix in $\mathbb{H}_d^+$ such that $\E\ntr \mtx{Y} = 1$. Then 
\begin{equation*}\label{eqn:Young_inequality}
\E \ntr\left[\mtx{X}\mtx{Y}\right] \leq \log\E\ntr\econst^{\mtx{X}} + \E\ntr\left[\mtx{Y}\log \mtx{Y}\right].
\end{equation*} 
\end{fact}

Apply Fact~\ref{lem:Young_inequality} to see that, for any $\beta>0$,
\begin{equation}\label{step:exponential_3}
\begin{split}
\E_\mu\ntr \left[\Gamma(\mtx{f}) \, \econst^{\theta\mtx{f}}\right] &= \frac{m(\theta)}{\beta} \E_\mu\ntr \left[\beta \Gamma(\mtx{f})\frac{\econst^{\theta\mtx{f}}}{m(\theta)}\right]\\
&\leq \frac{m(\theta)}{\beta} \left(\log\E_\mu\ntr \exp\left(\beta\Gamma(\mtx{f})\right) + \E_\mu\ntr\left[\frac{\econst^{\theta\mtx{f}}}{m(\theta)}\log\frac{\econst^{\theta\mtx{f}}}{m(\theta)}\right]\right) \\
&= m(\theta)\, r(\beta) + \frac{1}{\beta}\E_\mu\ntr\left[\econst^{\theta\mtx{f}}\log\frac{\econst^{\theta\mtx{f}}}{m(\theta)}\right].
\end{split}
\end{equation}
We have identified the exponential mean $r(\beta) := \beta^{-1}\log\E_\mu\ntr \exp\left(\beta\Gamma(\mtx{f}) \right)$.

Likewise,
\begin{equation*}
\E_\mu\ntr \left[(P_t\Gamma(\mtx{f}))\,\econst^{\theta\mtx{f}}\right]\leq \frac{m(\theta)}{\beta} \log\E_\mu\ntr \exp\left(\beta P_t\Gamma(\mtx{f})\right) + \frac{1}{\beta}\E_\mu\ntr\left[\econst^{\theta\mtx{f}}\log\frac{\econst^{\theta\mtx{f}}}{m(\theta)}\right].
\end{equation*}
The trace exponential $\ntr \exp(\cdot)$ is operator convex; see \cite[Theorem 2.10]{carlen2010trace}.
The Jensen inequality \eqref{eqn:semigroup_Jensen_2} for the semigroup implies that 
\begin{equation*}\label{step:exponential_5}
\E_\mu\ntr \exp\left(\beta P_t\Gamma(\mtx{f})\right) \leq \E_\mu\ntr \exp\left(\beta\Gamma(\mtx{f})\right)
	= \exp \left(\beta r(\beta)\right).
\end{equation*}
Combine the last two displays to obtain
\begin{equation}\label{step:exponential_4}
\E_\mu\ntr \left[(P_t\Gamma(\mtx{f}))\,\econst^{\theta\mtx{f}}\right]
	\leq m(\theta)\, r(\beta) + \frac{1}{\beta}\E_\mu\ntr\left[\econst^{\theta\mtx{f}}\log\frac{\econst^{\theta\mtx{f}}}{m(\theta)}\right].
\end{equation}
Thus, the two terms on the right-hand side of~\eqref{step:exponential_2}
have matching bounds.

Sequence the displays \eqref{step:exponential_2}, \eqref{step:exponential_3}, and \eqref{step:exponential_4}
to reach
\begin{equation}\label{step:exponential_6}
-\E_\mu \ntr \left[\mL(P_t\mtx{f})\,\econst^{\theta\mtx{f}}\right] \leq \econst^{-t/c}\theta \left( m(\theta)\, r(\beta) + \frac{1}{\beta}\E_\mu\ntr\left[\econst^{\theta\mtx{f}}\log\frac{\econst^{\theta\mtx{f}}}{m(\theta)}\right] \right).
\end{equation}
This is the integrand in \eqref{step:exponential_1}.
Next, we simplify this expression to arrive at a differential inequality.

\subsubsection{A differential inequality}

In view of Proposition~\ref{prop:trace_mgf}\eqref{eqn:m.g.f_Property_1}, we have $\log m(\theta)\geq0$ and hence 
\[\log\frac{\econst^{\theta\mtx{f}}}{m(\theta)} = \theta \mtx{f} - \log m(\theta) \cdot \Id \preccurlyeq \theta \mtx{f}.\]
It follows that 
\begin{equation}\label{step:exponential_7}
\E_\mu\ntr\left[\econst^{\theta\mtx{f}}\log\frac{\econst^{\theta\mtx{f}}}{m(\theta)}\right]\leq \theta \E_\mu \ntr\left[\mtx{f}\,\econst^{\theta\mtx{f}}\right] = \theta\, m'(\theta).
\end{equation}
Combine \eqref{step:exponential_6} and \eqref{step:exponential_7} to reach
\[- \E_\mu \ntr \left[\mL(P_t\mtx{f})\,\econst^{\theta\mtx{f}}\right] \leq \econst^{-t/c}\theta \left(m(\theta)\,r(\beta) + \frac{\theta}{\beta} m'(\theta) \right).\]
Substitute this bound into \eqref{step:exponential_1} and compute the integral
to arrive at the differential inequality 
\begin{equation}\label{eqn:differential_inequality}
m'(\theta)\leq  c\theta \,m(\theta)\,r(\beta) + \frac{c\theta^2}{\beta} m'(\theta)\quad\text{for $\theta\geq 0$.}
\end{equation}
Finally, we need to solve for the trace mgf.

\subsubsection{Solving the differential inequality}
Fix parameters $\theta$ and $\beta$ where $0\leq \theta <\sqrt{\beta/c}$.
By rearranging the expression \eqref{eqn:differential_inequality},
we find that
\[\frac{\diff{} }{\diff \zeta}\log m(\zeta) \leq \frac{c\zeta \, r(\beta)}{1-c\zeta^2/\beta}\leq \frac{c\zeta\,r(\beta)}{1-c\theta^2/\beta}
\quad\text{for $\zeta \in (0, \theta]$.}
\]
Since $\log m(0) = 0$, we can integrate this bound over $[0,\theta]$ to obtain
\[\log m(\theta) \leq \frac{c\theta^2 r(\beta)}{2(1-c\theta^2/\beta)}.\]
This is the first claim \eqref{eqn:exponential_moment_1}. 

Moreover, it is easy to check that $r(\beta)\leq v_{\mtx{f}}$.
Since this bound is independent of $\beta$, we can take $\beta\rightarrow +\infty$ in \eqref{eqn:exponential_moment_1} to achieve \eqref{eqn:exponential_moment_2}.  This completes the proof of Theorem~\ref{thm:exponential_moment}.

\subsection{Exponential matrix concentration}
\label{sec:exponential_concentration_proof}

We are now ready to prove Theorem~\ref{thm:exponential_concentration},
the exponential matrix concentration inequality,
as a consequence of the moment bounds of Theorem~\ref{thm:exponential_moment}.
To do so, we use the standard matrix Laplace transform method,
summarized in Appendix~\ref{apdx:matrix_moments}.

\begin{proof}[Proof of Theorem~\ref{thm:exponential_concentration} from Theorem~\ref{thm:exponential_moment}] 
To obtain inequalities for the maximum eigenvalue $\lambda_{\max}$,
we apply Proposition~\ref{prop:matrix_exponential_concentration} to the random matrix $\mtx{X} = \mtx{f}(Z) -\E_\mu\mtx{f}$
where $Z \sim \mu$.
To do so, we first need to weaken the moment bound \eqref{eqn:exponential_moment_1}:
\[\log\E_\mu \ntr \econst^{\theta(\mtx{f}-\E_\mu\mtx{f})} \leq \frac{c\theta^2r(\beta)}{2(1-c\theta^2/\beta)}\leq \frac{c\theta^2r(\beta)}{2(1-\theta\sqrt{c/\beta})}\quad \text{for $0\leq \theta < \sqrt{\beta/c}$}.\]
Then substitute $c_1=c r(\beta)$ and $c_2 = \sqrt{c/\beta}$ into Proposition~\ref{prop:matrix_exponential_concentration}
to achieve the results stated in Theorem~\ref{thm:exponential_concentration}.

To obtain bounds for the minimum eigenvalue $\lambda_{\min}$, we apply Proposition~\ref{prop:matrix_exponential_concentration} instead
to the random matrix $\mtx{X} = -(\mtx{f}(Z) -\E_\mu\mtx{f})$
where $Z \sim \mu$.
\end{proof}

\section{Bakry--{\'E}mery criterion for product measures}
\label{sec:product_measure_all}

In this section, we introduce the classic Markov process for a product measure.  We check the Bakry--{\'E}mery criterion for this Markov process, which leads to matrix concentration results for product measures.
 
\subsection{Product measures and Markov processes}

Consider a product space $\Omega = \Omega_1\otimes \Omega_2\otimes \cdots\otimes \Omega_n $ equipped with a product measure $\mu = \mu_1\otimes \mu_2\otimes \cdots\otimes\mu_n$.  We can construct a Markov process $(Z_t)_{t\geq0} = (Z^1_t,Z^2_t,\dots,Z^n_t)_{t\geq 0}$ on $\Omega$ whose stationary measure is $\mu$.  Let $\{N_t^i\}_{i=1}^n$ be a sequence of independent Poisson processes.  Whenever $N_t^i$ increases for some $i$, we replace the value of $Z_t^i$ in $Z_t$ by an independent sample from $\mu_i$ while keeping the remaining coordinates fixed. 

To describe the Markov semigroup associated with this Markov process, we need some notation.
For each subset $I\subseteq \{1,\dots,n\}$ and all $z,w\in\Omega$, define the interlacing operation 
\[(z;w)_I := (\eta^1,\eta^2,\dots,\eta^n)\quad \text{where}\quad
\begin{cases}
\eta^i = w^i, & i\in I; \\
\eta^i = z^i, & i\notin I.
\end{cases}
\]
In particular, $(z;w)_{\emptyset} = z$, and we abbreviate $(z;w)_i = (z^1,\dots,z^{i-1},w^i,z^{i+1},\dots,z^n)$.  In this section, the superscript stands for the index of the coordinate.

Let $Z = (Z^1,Z^2,\dots,Z^n)\in\Omega$ be a random vector drawn from the measure $\mu$; that is, each coordinate $Z^{i}\in\Omega_i$ is drawn independently from the measure $\mu_i$.  Through this section, we write $\E_Z := \E_{Z\sim\mu}$.  The Markov semigroup $(P_t)_{t\geq 0}$ induced by the Markov process is given by 
\begin{equation}\label{eqn:tensor_Markov}
P_t\mtx{f}(z) = \sum_{I\subseteq \{1,\dots,n\}}(1-\econst^{-t})^{|I|}\econst^{-t(n-|I|)} \cdot \E_Z \mtx{f}\big((z;Z)_I\big) \quad\text{for all $z\in\Omega$.}
\end{equation}
This formula is valid for every $\mu$-integrable function $\mtx{f}:\Omega \rightarrow \mathbb{H}_d$.  The ergodicity \eqref{eqn:ergodicity} of the semigroup follows immediately from~\eqref{eqn:tensor_Markov}
because $\lim_{t\rightarrow\infty}(1-\econst^{-t})^{|I|}\econst^{-t(n-|I|)}=0$ whenever $|I|<n$.

The infinitesimal generator $\mL$ of the semigroup admits the explicit form 
\begin{equation}\label{eqn:tensor_L}
\mL\mtx{f} = \lim_{t\downarrow 0}\frac{P_t\mtx{f}-\mtx{f}}{t} = -\sum_{i=1}^n\delta_i\mtx{f}.
\end{equation}
The difference operator $\delta_i$ is given by
\[\delta_i\mtx{f}(z) := \mtx{f}(z) - \E_Z \mtx{f}\big((z;Z)_i\big)\quad \text{for all $z\in\Omega$}.\] 
This infinitesimal generator $\mL$ is well defined for all integrable functions, so the class of suitable functions contains $L_1(\mu)$. It follows from the definition of $\delta_i$ that 
\[\E_\mu[\mtx{f} \,\delta_i(\mtx{g})] = \E_\mu[\delta_i(\mtx{f}) \,\delta_i(\mtx{g})] = \E_\mu[\delta_i(\mtx{f}) \, \mtx{g}]\quad \text{for each $1\leq i\leq n$}.\]
Thus, the infinitesimal generator $\mL$ is symmetric on $L_2(\mu)$.  As a consequence, the semigroup is reversible, and the Dirichlet form is given by 
\[\mathcal{E}(\mtx{f},\mtx{g}) = \E_\mu\left[\sum_{i=1}^n\delta_i(\mtx{f})\delta_i(\mtx{g})\right]=\sum_{i=1}^n\E_Z\left[\left(\mtx{f}(Z)-\E_{\tilde{Z}}\mtx{f}((Z;\tilde{Z})_i)\right)\left(\mtx{g}(Z)-\E_{\tilde{Z}}\mtx{g}((Z;\tilde{Z})_i)\right)\right]\]
for any $\mtx{f},\mtx{g}:\Omega \rightarrow \mathbb{H}_d$, where $\tilde{Z}$ is an independent copy of $Z$. All the results above and their proofs can be found in \cite{van550probability,ABY20:Matrix-Poincare}. 

\subsection{Carr\'e du champ operators} The following lemma gives the formulas for the matrix carr\'e du champ operator and the iterated matrix carr\'e du champ operator. 

\begin{lemma}[Product measure: Carr{\'e} du champs] \label{lem:tensor_Gamma} The matrix carr\'e du champ operator $\Gamma$ and the iterated matrix carr\'e du champ operator $\Gamma_2$ of the semigroup~\eqref{eqn:tensor_Markov} are given by the formulas
\begin{align}\label{eqn:tensor_Gamma}
\Gamma(\mtx{f},\mtx{g})(z) &= \frac{1}{2}\sum_{i=1}^n \E_Z\left[\big(\mtx{f}(z)-\mtx{f}((z;Z)_i)\big)\cdot\big(\mtx{g}(z)-\mtx{g}((z;Z)_i)\big)\right]
\intertext{and}
\Gamma_2(\mtx{f},\mtx{g})(z) &= \frac{1}{4}\sum_{i=1}^n \E_{\tilde{Z}}\E_Z \Big[\big(\mtx{f}(z)-\mtx{f}((z;Z)_i)\big)\cdot\big(\mtx{g}(z)-\mtx{g}((z;Z)_i)\big)\\
& \qquad\qquad\qquad\qquad\qquad + \big(\mtx{f}((z;\tilde{Z})_i)-\mtx{f}((z;Z)_i)\big)\cdot\big(\mtx{g}((z;\tilde{Z})_i)-\mtx{g}((z;Z)_i)\big)\Big] \\
& + \frac{1}{4}\sum_{i\neq j}\E_{\tilde{Z}}\E_Z \Big[\big(\mtx{f}(z)-\mtx{f}((z;\tilde{Z})_i) - \mtx{f}((z;Z)_j) + \mtx{f}(((z;\tilde{Z})_i;Z)_j) \big)\\ 
&\qquad\qquad\qquad\qquad\qquad \times\big(\mtx{g}(z)-\mtx{g}((z;\tilde{Z})_i) - \mtx{g}((z;Z)_j) + \mtx{g}(((z;\tilde{Z})_i;Z)_j) \big) \Big].
\label{eqn:tensor_Gamma2}
\end{align}
These expressions are valid for all suitable $\mtx{f},\mtx{g}:\Omega \rightarrow \mathbb{H}_d$ and all $z\in \Omega$.
The random variables $Z$ and $\tilde{Z}$ are independent draws from the measure $\mu$.
\end{lemma}

\begin{proof}[Proof of Lemma~\ref{lem:tensor_Gamma}] The expression \eqref{eqn:tensor_Gamma} is a consequence of the form \eqref{eqn:tensor_L} of the infinitesimal generator and the definition \eqref{eqn:definition_Gamma} of the carr\'e du champ operator $\Gamma$. Further, the following displays are consequences of \eqref{eqn:tensor_L} and \eqref{eqn:tensor_Gamma}. 
\begin{align*}
& \mL\Gamma(\mtx{f},\mtx{g})(z) \\
&\qquad = -\sum_{i=1}^n\delta_i\Gamma(\mtx{f},\mtx{g})(z)\\
&\qquad = -\frac{1}{2}\sum_{i,j=1}^n\E_{\tilde{Z}}\E_{Z} \Big[\big(\mtx{f}(z)-\mtx{f}((z;Z)_j)\big)\cdot\big(\mtx{g}(z)-\mtx{g}((z;Z)_j)\big) \\
&\qquad\qquad\qquad\qquad\qquad\qquad\quad - \big(\mtx{f}((z;\tilde{Z})_i)-\mtx{f}(((z;\tilde{Z})_i;Z)_j)\big)\cdot\big(\mtx{g}((z;\tilde{Z})_i)-\mtx{g}(((z;\tilde{Z})_i;Z)_j)\big)\Big].\\
& \Gamma(\mtx{f},\mL\mtx{g})(z)\\
&\qquad = -\sum_{i=1}^n\Gamma(\mtx{f},\delta_i\mtx{g})(z)\\
&\qquad = -\frac{1}{2}\sum_{i,j=1}^n\E_{Z}\Big[\big(\mtx{f}(z)-\mtx{f}((z;Z)_j)\big)\\
& \qquad\qquad\qquad\qquad\qquad\qquad\quad \times \big(\mtx{g}(z)-\E_{\tilde{Z}}\big[\mtx{g}((z;\tilde{Z})_i)\big] - \mtx{g}((z;Z)_j) + \E_{\tilde{Z}}\big[\mtx{g}(((z;Z)_j;\tilde{Z})_i)\big]\big)\Big]\\
&\qquad = -\frac{1}{2}\sum_{i,j=1}^n\E_{\tilde{Z}}\E_{Z}\Big[\big(\mtx{f}(z)-\mtx{f}((z;Z)_j)\big)\\
&\qquad \qquad\qquad\qquad\qquad\qquad\quad \times \big(\mtx{g}(z)-\mtx{g}((z;\tilde{Z})_i) - \mtx{g}((z;Z)_j) + \mtx{g}(((z;Z)_j;\tilde{Z})_i)\big)\Big].\\
&\Gamma(\mL\mtx{f},\mtx{g})(z)\\
&\qquad = -\sum_{i=1}^n\Gamma(\delta_i\mtx{f},\mtx{g})(z)\\
&\qquad = -\frac{1}{2}\sum_{i,j=1}^n\E_{Z}\Big[\big(\mtx{f}(z)-\E_{\tilde{Z}}\big[\mtx{f}((z;\tilde{Z})_i)\big] - \mtx{f}((z;Z)_j) + \E_{\tilde{Z}}\big[\mtx{f}(((z;Z)_j;\tilde{Z})_i)\big]\big)\\
&\qquad\qquad\qquad\qquad\qquad\qquad\quad \times \big(\mtx{g}(z)-\mtx{g}((z;Z)_j)\big)\Big]\\
&\qquad = -\frac{1}{2}\sum_{i,j=1}^n\E_{\tilde{Z}}\E_{Z}\Big[\big(\mtx{f}(z)-\mtx{f}((z;\tilde{Z})_i) - \mtx{f}((z;Z)_j) + \mtx{f}(((z;Z)_j;\tilde{Z})_i)\big)\\
& \qquad\qquad\qquad\qquad\qquad\qquad\quad \times \big(\mtx{g}(z)-\mtx{g}((z;Z)_j)\big)\Big].
\end{align*}
If $j=i$, then $((z;\tilde{Z})_i;Z)_j = (z;Z)_i$ and $((z;Z)_j;\tilde{Z})_i = (z;\tilde{Z})_i$.  But if $j\neq i$, then $((z;Z)_j;\tilde{Z})_i = ((z;\tilde{Z})_i;Z)_j$. Therefore, by the definition \eqref{eqn:definition_Gamma2} of iterated carr\'e du champ operator $\Gamma_2$, we can compute that
\begin{align*}
&\Gamma_2(\mtx{f},\mtx{g})(z)\\
&\qquad =  \frac{1}{4}\sum_{i,j=1}^n\E_{\tilde{Z}}\E_{Z}\Big[\big(\mtx{f}(z)-\mtx{f}((z;Z)_j)\big)\cdot\big(\mtx{g}(z)-\mtx{g}((z;Z)_j)\big) \\
& \qquad\qquad\qquad\qquad\qquad\qquad\quad + \big(\mtx{f}((z;\tilde{Z})_i)-\mtx{f}(((z;\tilde{Z})_i;Z)_j)\big)\cdot\big(\mtx{g}((z;\tilde{Z})_i)-\mtx{g}(((z;\tilde{Z})_i;Z)_j)\big)\\
& \qquad\qquad\qquad\qquad\qquad\qquad\quad - \big(\mtx{f}(z)-\mtx{f}((z;Z)_j)\big)\cdot\big(\mtx{g}((z;\tilde{Z})_i) - \mtx{g}(((z;Z)_j;\tilde{Z})_i)\big)\\
& \qquad\qquad\qquad\qquad\qquad\qquad\quad - \big(\mtx{f}((z;\tilde{Z})_i) - \mtx{f}(((z;Z)_j;\tilde{Z})_i)\big)\cdot\big(\mtx{g}(z)-\mtx{g}((z;Z)_j)\big)\Big]\\
&\qquad = \frac{1}{4}\sum_{i=1}^n\E_{\tilde{Z}}\E_Z \Big[\big(\mtx{f}(z)-\mtx{f}((z;Z)_i)\big)\cdot\big(\mtx{g}(z)-\mtx{g}((z;Z)_i)\big) \\
& \qquad\qquad\qquad\qquad\qquad\qquad\quad + \big(\mtx{f}((z;\tilde{Z})_i)-\mtx{f}((z;Z)_i)\big)\cdot\big(\mtx{g}((z;\tilde{Z})_i)-\mtx{g}((z;Z)_i)\big)\Big] \\
&\qquad\qquad + \frac{1}{4}\sum_{i\neq j}\E_{\tilde{Z}}\E_Z \Big[\big(\mtx{f}(z)-\mtx{f}((z;\tilde{Z})_i) - \mtx{f}((z;Z)_j) + \mtx{f}(((z;\tilde{Z})_i;Z)_j) \big)\\ 
&\qquad\qquad\qquad\qquad\qquad\qquad\qquad\quad \times \big(\mtx{g}(z)-\mtx{g}((z;\tilde{Z})_i) - \mtx{g}((z;Z)_j) + \mtx{g}(((z;\tilde{Z})_i;Z)_j) \big) \Big].
\end{align*}
This gives the expression \eqref{eqn:tensor_Gamma2}.
\end{proof}

\subsection{Bakry--\'Emery criterion}

It is clear from Lemma~\ref{lem:tensor_Gamma} that the formula~\eqref{eqn:tensor_Gamma} for $\Gamma$ appears within the formula~\eqref{eqn:tensor_Gamma2} for $\Gamma_2$.  We immediately conclude that the Bakry--{\'E}mery criterion holds. 

\begin{theorem}[Product measure: Bakry--{\'E}mery] \label{thm:product_measure_localPoincare} For the semigroup~\eqref{eqn:tensor_Markov}, the Bakry--\'Emery criterion \eqref{Bakry-Emery} holds with $c = 2$.
That is, for any suitable function $f:\Omega \rightarrow \mathbb{R}$,
\begin{equation*}\label{eqn:product_measure_localPoincare}
\Gamma(f)\leq 2\Gamma_2(f).
\end{equation*}
\end{theorem}

\begin{proof} Comparing the two expressions in Lemma~\ref{lem:tensor_Gamma} with $f=g$ gives 
\begin{align*}
\Gamma_2(f)(z) &= \frac{1}{4}\sum_{i=1}^n \E_{\tilde{Z}}\E_Z \Big[\big(f(z)-f((z;Z)_i)\big)^2 + \left(f((z;\tilde{Z})_i)-f((z;Z)_i)\right)^2\Big] \\
&\qquad + \frac{1}{4}\sum_{i\neq j} \E_{\tilde{Z}}\E_Z \Big[\Big(f(z)-f((z;\tilde{Z})_i) - f((z;Z)_j) + f(((z;\tilde{Z})_i;Z)_j) \Big)^2\Big]\\
&\geq \frac{1}{4}\sum_{i=1}^n \E_Z \left[\big(f(z)-f((z;Z)_i)\big)^2\right]\\
&= \frac{1}{2}\Gamma(f)(z),
\end{align*}
which is the stated inequality.
\end{proof}

After completing this paper, we learned that Theorem~\ref{thm:product_measure_localPoincare}
appears in \cite[Example 6.6]{junge2015noncommutative} with a different style of proof.

\begin{remark}[Matrix Poincar\'e inequality: Constants]
Following the discussion in Section~\ref{sec:local_matrix_Poincare_inequality}, Theorem~\ref{thm:product_measure_localPoincare} implies the matrix Poincar\'e inequality~\eqref{eqn:matrix_Poincare} with $\alpha = 2$. However, Aoun et al.~\cite{ABY20:Matrix-Poincare} proved that the Markov process~\eqref{eqn:tensor_Markov} actually satisfies the matrix Poincar\'e inequality with $\alpha = 1$; see also \cite[Theorem 5.1]{cheng2016characterizations}.  This gap is not surprising because the averaging operation that is missing in the local Poincar\'e inequality contributes to the global convergence of the Markov semigroup.
\end{remark}

\subsection{Matrix concentration results}
\label{sec:concentration_results_product}
In this subsection, we complete the proofs of the matrix concentration results for product measures stated in Section~\ref{sec:main_results}.  

For a product measure $\mu = \mu_1\otimes\mu_2\otimes\cdots\otimes\mu_n$, Theorem~\ref{thm:product_measure_localPoincare} shows that there is a reversible ergodic Markov semigroup whose stationary measure is $\mu$ and which satisfies the Bakry--\'Emery criterion \eqref{Bakry-Emery} with constant $c=2$. We then apply Theorem~\ref{thm:polynomial_moment} with $c=2$ to obtain the polynomial moment bounds in Corollary~\ref{cor:product_measure_Efron--Stein}. Similarly, we apply Theorem~\ref{thm:exponential_concentration} with $c=2$ to obtain the subgaussian concentration inequalities in Corollary~\ref{cor:product_measure_tailbound}.

\section{Bakry--{\'E}mery criterion for log-concave measures}
\label{sec:log-concave}

In this section, we study a class of log-concave measures; the most important example in this class is the standard Gaussian measure.  First, we introduce the standard diffusion process associated with a log-concave measure.  We verify that the associated semigroup is reversible and ergodic via standard arguments. Then we introduce the Bakry--{\'E}mery criterion which follows from the uniform strong convexity of the potential.

\subsection{Log-concave measures and Markov processes}

Consider the Markov processes $(Z_t)_{t\geq 0}$ on $\mathbb{R}^n$ generated by the stochastic differential equation:
\begin{equation}\label{eqn:SDE}
\diff Z_t = -\nabla W(Z_t)\idiff{t} + \sqrt{2}\idiff{B_t},
\end{equation}
where $B_t$ is the standard $n$-dimensional Brownian motion and $W:\mathbb{R}^n\rightarrow \mathbb{R}$ is a smooth convex function. The stationary measure $\mu$ of this process has the density $\diff \mu = \rho^\infty(z)\idiff{z} = M^{-1}\econst^{-W(z)}\idiff{z}$, where $M := \int_{\mathbb{R}^n}\econst^{-W(z)}\idiff z$ is a normalization constant.  The infinitesimal generator $\mL$ is given by
\begin{equation}\label{eqn:log-concave_L}
\mL\mtx{f}(z) = -\sum_{i=1}^n\partial_iW(z)\cdot\partial_i\mtx{f}(z) + \sum_{i=1}^n\partial_i^2\mtx{f}(z)
\quad\text{for all $z=(z_1,\dots,z_n)\in \mathbb{R}^n$.}
\end{equation}
The class of suitable functions is the Sobolev space $\mathrm{H}_{2,\mu}(\mathbb{R}^n;\mathbb{H}_d)$, defined in~\eqref{def:H2_function}. Here and elsewhere, $\partial_i$ means $\partial/\partial z_i$ and $\partial_{ij}$ means $\partial^2/(\partial z_i\partial z_j)$ for all $i,j=1,\dots,n$. 

\subsubsection{Reversibility} The reversibility of this Markov $(Z_t)_{t\geq0}$ can be verified with a standard calculation. We restrict our attention to functions in $\mathrm{H}_{2,\mu}(\mathbb{R}^n;\mathbb{H}_d)$.  Integration by parts yields 
\begin{align*}
\E_\mu[\mL(\mtx{f})\mtx{g}] =&\ \int_{\mathbb{R}^n}\left(-\sum_{i=1}^n\partial_iW(z)\cdot\partial_i\mtx{f}(z) + \sum_{i=1}^n\partial_i^2\mtx{f}(z)\right)\mtx{g}(z)\rho^\infty(z)\idiff z\\
=&\ -\sum_{i=1}^n\int_{\mathbb{R}^n}\partial_i\mtx{f}(z)\cdot \partial_i\mtx{g}(z)\cdot\rho^\infty(z)\idiff z\\
=&\ \int_{\mathbb{R}^n}\mtx{f}(z)\left(-\sum_{i=1}^n\partial_iW(z)\cdot\partial_i\mtx{g}(z) + \sum_{i=1}^n\partial_i^2\mtx{g}(z)\right)\rho^\infty(z)\idiff z\\
=&\ \E_\mu[\mtx{f}\mL(\mtx{g})].
\end{align*}
This shows that $\mL$ is symmetric on $L_2(\mu)$ and thus $(Z_t)_{t\geq 0}$ is reversible. From the calculation above, we also obtain a simple formula for the associated Dirichlet form:
\[\mathcal{E}(\mtx{f},\mtx{g}) = \sum_{i=1}^n\E_\mu\left[\partial_i\mtx{f}\cdot \partial_i\mtx{g}\right]\quad \text{for all $\mtx{f},\mtx{g}\in \mathrm{H}_{2,\mu}(\mathbb{R}^n;\mathbb{H}_d)$}.\]
These results parallel the scalar case, but the partial derivatives are matrix-valued.

\subsubsection{Ergodicity} We now turn to the ergodicity of the Markov process given by \eqref{eqn:SDE}, which generally reduces to studying the convergence of the corresponding Fokker--Planck equation:
\begin{equation}\label{eqn:Fokker-Planck}
\begin{cases}
\frac{\partial}{\partial t}\rho_{x}(z,t) = \mL^*\rho_{x}(z,t) := \sum_{i=1}\partial_i(\partial_iW(z)\rho_{x}(z,t)) + \sum_{i=1}^n\partial_i^2\rho_{x}(z,t); \\[3pt]
\rho_{x}(z,0) = \delta(z-x).
\end{cases}
\end{equation}
We define $\rho_{x}(z,t)$ to be the density of $Z_t$, conditional on $Z_0 = x\in\mathbb{R}^n$.  As usual, $\delta(z-x)$ is the Dirac distribution centered at $x$. The associated Markov semigroup $(P_t)_{t\geq 0}$ can be recognized as
\begin{equation}\label{eqn:log-concave_semigroup}
P_t \mtx{f}(x) = \E_\mu\left[\mtx{f}(Z_t) \,|\, Z_0 = x \right] = \int_{\mathbb{R}^n} \mtx{f}(z) \rho_{x}(z,t) \idiff z \quad \text{for all $t\geq0$ and all $x\in \mathbb{R}^n$ }.
\end{equation}
The semigroup $(P_t)_{t\geq 0}$ is ergodic in the sense of \eqref{eqn:ergodicity} if and only if $\rho_{x}(\cdot,t)$ converges weakly to $\rho^\infty$ for all $x\in \mathbb{R}^n$.

A fundamental way to prove the convergence of \eqref{eqn:Fokker-Planck} to the stationary density $\rho^\infty$ is through the method of Lyapunov functions~\cite{hairer2010convergence,ji2019convergence}.  However, ergodicity in the weak sense follows more easily from the assumption that the function $W$ is uniformly strongly convex.  That is,
\[
(\operatorname{Hess} W)(z) := \big[\partial_{ij} W(z)\big]_{i,j=1}^n
	\succcurlyeq \eta \cdot \Id_n
	\quad\text{for all $z \in \mathbb{R}^n$.}
\]
To see this, recall the Brascamp--Lieb inequality~\cite[Theorem 4.1]{BRASCAMP1976366},
which states that the (ordinary) variance of a scalar function $h:\mathbb{R}^n\rightarrow \mathbb{R}$
is bounded as
\[\Var_\mu[h]\leq \int_{\mathbb{R}^n} (\nabla h(z))^\trsp\big((\operatorname{Hess} W)(z)\big)^{-1}\nabla h(z) \idiff \mu(z).\]
Combine the last two displays to arrive at the Poincar\'e inequality $\Var_\mu[h]\leq \eta^{-1}\mathcal{E}(h)$.

Next, consider the scalar function $\phi_{x}(z,t) := (\rho_{x}(z,t) - \rho^\infty(z))/\rho^\infty(z)$.  Let us check that its variance $\Var_\mu[\phi_{x}(\cdot,t)]$ converges to $0$ exponentially fast.  Indeed, it is not hard to verify that $\phi_{x}(z,t)$ satisfies the partial differential equation
\[\frac{\partial}{\partial t} \phi_{x}(z,t) = \mL\phi_{x}(z,t) \quad \text{for $t\geq0$ and $z\in \mathbb{R}^n$}.\]
Along with the Poincar\'e inequality and the fact that $\E_\mu\phi_{x}(\cdot,t) = 0$, this implies 
\[\frac{\diff{} }{\diff t} \Var_\mu[\phi_{x}(\cdot,t)]  = - 2\mathcal{E}(\phi_{x}(\cdot,t))\leq  - 2\eta \Var_\mu[\phi_{x}(\cdot,t)]. \]
Therefore, the quantity $\Var_\mu[\phi_{x}(\cdot,t)]$ converges to $0$ exponentially fast because 
\[\Var_\mu[\phi_{x}(\cdot,t)] \leq \econst^{-2\eta (t-t_0)} \Var_\mu[\phi_{x}(\cdot,t_0)]\quad \text{for}\ t\geq t_0>0.\]
As a consequence, for any $f\in \mathrm{H}_{2,\mu}(\mathbb{R}^n;\mathbb{R})$ and any $x\in \mathbb{R}^n$, 
\begin{align*}
\left|P_tf(x) - \E_\mu f\right| &= \left|\int_{\mathbb{R}^n} f(z)(\rho_{x}(z,t)-\rho^\infty(z))\idiff z \right| = \left|\int_{\mathbb{R}^n} f(z)\rho^\infty(z)\phi_{x}(z,t)\idiff z\right|  \\
&\leq \int_{\mathbb{R}^n} |f(z)|\cdot\rho^\infty(z)\cdot|\phi_{x}(z,t)|\idiff z \leq \left(\E_\mu |f|^2\right)^{1/2}\Var_\mu[\phi_{x}(\cdot,t)]^{1/2} \rightarrow 0.
\end{align*}
This justifies the pointwise convergence of $P_t\mtx{f}$, which is stronger than the $L_2(\mu)$ ergodicity \eqref{eqn:ergodicity} of the semigroup $(P_t)_{t\geq0}$. 

\subsection{Carr\'e du champ operators}
After checking reversibility and ergodicity, we now turn to the derivation of the matrix carr\'e du champ operator and the iterated matrix carr\'e du champ operator. Their explicit forms are given in the next lemma.

\begin{lemma}[Log-concave measure: Carr{\'e} du champs] \label{lem:log-concave_Gamma} The matrix carr\'e du champ operator $\Gamma$ and the iterated matrix carr\'e du champ operator $\Gamma_2$ of the Markov process defined by~\eqref{eqn:SDE} are given by the formulas 
\begin{equation}\label{eqn:log-concave_Gamma}
\Gamma(\mtx{f},\mtx{g})= \sum_{i=1}^n\partial_i\mtx{f}\cdot \partial_i\mtx{g}
\end{equation}
and 
\begin{equation}\label{eqn:log-concave_Gamma2}
\Gamma_2(\mtx{f},\mtx{g}) = \sum_{i,j=1}^n\partial_{ij}W\cdot \partial_i\mtx{f} \cdot \partial_j\mtx{g} + \sum_{i,j=1}^n\partial_{ij}\mtx{f}\cdot \partial_{ij}\mtx{g}
\end{equation}
for all suitable $\mtx{f},\mtx{g}:\mathbb{R}^n\rightarrow\mathbb{H}_d$. 
\end{lemma}

\begin{proof}[Proof of Lemma~\ref{lem:log-concave_Gamma}]
Knowing the explicit form \eqref{eqn:log-concave_L} of the Markov generator $\mL$, we can compute the carr\'e du champ operator $\Gamma$ as
\begin{align*}
\Gamma(\mtx{f},\mtx{g})=&\ \frac{1}{2}\sum_{i=1}^n\left(-\partial_i W\cdot \partial_i(\mtx{f}\mtx{g}) + \partial_i^2(\mtx{f}\mtx{g}) - \big(-\partial_i W\cdot\partial_i \mtx{f}  + \partial_i^2\mtx{f}\big)\mtx{g} - \mtx{f}\big(-\partial_i W\cdot \partial_i\mtx{g} + \partial_i^2\mtx{g}\big)\right)\\
=&\ \sum_{i=1}^n\partial_i\mtx{f}\cdot \partial_i\mtx{g}.
\end{align*} 
Moreover, combining the expressions \eqref{eqn:log-concave_L} and \eqref{eqn:log-concave_Gamma} yields the following: 
\begin{align*}
\mL\Gamma(\mtx{f},\mtx{g}) =&\ -\sum_{i=1}^n\partial_iW\cdot \partial_i\left(\sum_{j=1}^n\partial_j\mtx{f}\cdot \partial_j\mtx{g}\right) + \sum_{i=1}^n\partial_i^2\left(\sum_{j=1}^n\partial_j\mtx{f}\cdot \partial_j\mtx{g}\right)\\
=&\ \sum_{i,j}^n\left(-\partial_iW\cdot \partial_{ij}\mtx{f}\cdot \partial_j\mtx{g} - \partial_iW\cdot \partial_j\mtx{f}\cdot \partial_{ij}\mtx{g} + \partial_i^2(\partial_j\mtx{f})\cdot \partial_j\mtx{g}+2\partial_{ij}\mtx{f}\cdot\partial_{ij}\mtx{g}+ \partial_j\mtx{f}\cdot \partial_i^2(\partial_j\mtx{g})\right).\\
\Gamma(\mL\mtx{f},\mtx{g}) =&\ \sum_{j=1}^n\partial_j\left(\sum_{i=1}^n \big(- \partial_iW\cdot \partial_i\mtx{f} + \partial_i^2\mtx{f})\right)\cdot \partial_j\mtx{g}\\
=&\ \sum_{i,j=1}^n\left(-\partial_{ij}W\cdot \partial_i\mtx{f}\cdot \partial_j\mtx{g} - \partial_iW\cdot \partial_{ij}\mtx{f}\cdot \partial_j\mtx{g} + \partial_i^2(\partial_j\mtx{f})\cdot \partial_j\mtx{g}\right).\\
\Gamma(\mtx{f},\mL\mtx{g}) =&\ \sum_{j=1}^n\partial_j\mtx{f}\cdot \partial_j\left(\sum_{i=1}^n \big(- \partial_iW \cdot\partial_i\mtx{g} + \partial_i^2\mtx{g})\right)\\
=&\ \sum_{i,j=1}^n\left(-\partial_{ij}W\cdot \partial_j\mtx{f}\cdot \partial_i\mtx{g} - \partial_iW\cdot \partial_j\mtx{f}\cdot \partial_{ij}\mtx{g} + \partial_j\mtx{f}\cdot \partial_i^2(\partial_j\mtx{g})\right).
\end{align*}
Then we can compute that
\begin{align*}
\Gamma_2(\mtx{f},\mtx{g}) =&\ \frac{1}{2}\left(\mL\Gamma(\mtx{f},\mtx{g}) -\Gamma(\mL\mtx{f},\mtx{g}) -\Gamma(\mtx{f},\mL\mtx{g})\right)\\
=&\ \frac{1}{2}\sum_{i,j=1}^n\left(\partial_{ij}W\cdot \partial_i\mtx{f}\cdot \partial_j\mtx{g}+ \partial_{ij}W\cdot \partial_j\mtx{f}\cdot \partial_i\mtx{g}\right) + \sum_{i,j=1}^n\partial_{ij}\mtx{f}\cdot \partial_{ij}\mtx{g}\\
=&\ \sum_{i,j=1}^n\partial_{ij}W\cdot \partial_i\mtx{f}\cdot \partial_j\mtx{g} + \sum_{i,j=1}^n\partial_{ij}\mtx{f}\cdot \partial_{ij}\mtx{g}.
\end{align*}
This gives the expression \eqref{eqn:log-concave_Gamma2}.
\end{proof}

\subsection{Bakry--\'Emery criterion} It is a well-known result that a Bakry--\'Emery criterion follows from the uniform strong convexity of $W$. For example, see the discussion in \cite[Sec. 4.8]{bakry2013analysis}. Nevertheless, we provide a short proof here for the sake of completeness.

\begin{fact}[Log-concave measure: Matrix Bakry--{\'E}mery] \label{fact:log-concave_localPoincare}
Consider the Markov process defined by \eqref{eqn:SDE}.  If the potential $W:\mathbb{R}\rightarrow\mathbb{R}$ satisfies $(\operatorname{Hess} W)(z)\succcurlyeq \eta \cdot \Id_n $ for all $z\in \mathbb{R}^n$ for some constant $\eta>0$, then the Bakry--\'Emery criterion \eqref{Bakry-Emery} holds with $c = \eta^{-1}$. That is, for any suitable function $f:\mathbb{R}^n\rightarrow\mathbb{R}$,
\begin{equation*}\label{eqn:log-concave_localPoincare}
\Gamma(f)\preccurlyeq \eta^{-1}\Gamma_2(f).
\end{equation*} 
\end{fact}

\begin{proof} Comparing the two expressions in Lemma~\ref{lem:log-concave_Gamma} with $f=g$ gives that 
\begin{align*}
\Gamma_2(f) &= \sum_{i,j=1}^n\partial_{ij}W\cdot \partial_if\cdot \partial_jf + \sum_{i,j=1}^n(\partial_{ij}f)^2 \\
&\geq (\nabla f)^\trsp(\operatorname{Hess} W)\nabla f
\geq \eta \sum_{i=1}^n(\partial_i\mtx{f})^2
= \eta\cdot\Gamma(\mtx{f}).
\end{align*}
The second inequality follows from the uniform strong convexity of $W$.
Proposition~\ref{prop:BE_equiv} extends the scalar Bakry--{\'E}mery criterion
to matrices.
\end{proof} 

\subsection{Standard normal distribution}\label{sec:Gaussian}

The most important example of a strongly log-concave measure
occurs for the potential 
\[W(z) = \frac{1}{2}z^\trsp z\quad\text{for all $z\in \mathbb{R}^n$.}\] 
In this case, the corresponding log-concave measure $\mu$ coincides with
the density of the $n$-dimensional standard Gaussian distribution $N(\vct{0},\Id_n)$:
\[\diff \mu = \frac{1}{\sqrt{(2\pi)^n}}\exp\left(-\frac{1}{2}z^\trsp z\right) \idiff{z}\quad \text{for all $z\in \mathbb{R}^n$.}\]
The associated Markov process is known as the Ornstein--Uhlenbeck process.  The semigroup $(P_t)_{t\geq 0}$
has a simple form, given by the Mehler formula:
\[P_t\mtx{f}(z) = \E  \mtx{f}\left(\econst^{-t}z + \sqrt{1-\econst^{-2t}}\xi\right)\quad\text{where $\xi\sim N(\vct{0},\Id_n)$.} \]
The ergodicity of this Markov semigroup is obvious from the above formula because $\econst^{-t}\rightarrow 0$ as $t\rightarrow +\infty$.  Lemma~\ref{lem:log-concave_Gamma} gives the matrix carr\'e du champ operator $\Gamma$ and the iterated matrix carr\'e du champ operator $\Gamma_2$ for the Ornstein--Uhlenbeck process: 
\[\Gamma(\mtx{f},\mtx{g}) = \sum_{i=1}^n\partial_i\mtx{f}\cdot \partial_i\mtx{g}\quad \text{and}\quad \Gamma_2(\mtx{f},\mtx{g}) = \sum_{i=1}^n\partial_i\mtx{f}\cdot \partial_i\mtx{g} + \sum_{i,j=1}^n\partial_{ij}\mtx{f}\cdot \partial_{ij}\mtx{g}.\]
Clearly, $\Gamma(\mtx{f})\preccurlyeq \Gamma_2(\mtx{f})$.  Therefore, the Bakry--{\'E}mery criterion~\eqref{Bakry-Emery} holds with $c = 1$.

\subsection{Matrix concentration results}
\label{sec:concentration_results_log-concave} 

Finally, we prove the matrix concentration results for log-concave measures stated in Section~\ref{sec:main_results}.

Consider a log-concave probability measure $\diff \mu \propto \econst^{-W(z)}\idiff z$ on $\mathbb{R}^n$,
where the potential satisfies the strong convexity condition $\operatorname{Hess} W\succcurlyeq \eta\Id_n$
for $\eta > 0$.
Fact~\ref{fact:log-concave_localPoincare} states that the associated semigroup \eqref{eqn:log-concave_semigroup} satisfies the Bakry--\'Emery criterion with constant $c=\eta^{-1}$.  We then apply Theorem~\ref{thm:polynomial_moment} with $c=\eta^{-1}$ to obtain the polynomial moment bounds in Corollary~\ref{cor:log-concave_polynomial_inequality}. Similarly, we apply Theorem~\ref{thm:exponential_concentration} with $c=\eta^{-1}$ to obtain the subgaussian concentration inequalities in Corollary~\ref{cor:log-concave_concentration}.

\section{Extension to Riemannian manifolds}\label{sec:extension_Riemannian_manifold}

In this section, we give a high-level discussion about diffusion processes on Riemannian manifolds. The book \cite{bakry2013analysis} contains a comprehensive treatment of the subject. For an introduction to calculus
on Riemannian manifolds, references include \cite{petersen2016riemannian,lee2018introduction}.

\subsection{Measures on Riemannian manifolds}

Let $(M, \mathfrak{g})$ be an $n$-dimensional Riemannian manifold whose co-metric tensor $\mathfrak{g}(x) = (g^{ij}(x) : 1 \leq i,j \leq n)$ is symmetric and positive definite for every $x \in M$.  We write $\mtx{G}(x) = (g_{ij} : 1 \leq i, j \leq n)$ for the metric tensor, which satisfies the relation $\mtx{G}(x) = \mathfrak{g}(x)^{-1}$.

The Riemannian measure $\mu_\mathfrak{g}$ on the manifold $(M,\mathfrak{g})$ has density $\diff \mu_\mathfrak{g} \propto w_\mathfrak{g}(x(z)) \idiff{z}$ with respect to the Lebesgue measure in local coordinates. The weight $w_\mathfrak{g} := \det(\mathfrak{g})^{-1/2}$.  Whenever this measure is finite, we normalize it to obtain a probability.
In particular, a compact Riemannian manifold always admits a Riemannian probability measure.

The matrix Laplace--Beltrami operator $\Delta_{\mathfrak{g}}$ on the manifold is defined as
\[\Delta_\mathfrak{g}\mtx{f}(x) := \frac{1}{w_\mathfrak{g}} \sum_{i,j=1}^n\partial_i\left(w_\mathfrak{g}g^{ij}\partial_j\mtx{f}(x)\right)\quad \text{for suitable $\mtx{f} : M \to \mathbb{H}_d$ and $x\in M$.}\] 
Here, $\partial_i$ and the like represent the components of the differential with respect to local coordinates.
The diffusion process on $M$ whose infinitesimal generator is $\Delta_{\mathfrak{g}}$
is called the Riemannian Brownian motion.  The measure $\mu_\mathfrak{g}$ is the
stationary measure for the Brownian motion.

To generalize, one may consider a weighted measure $\diff \mu \propto \econst^{-W} \diff \mu_\mathfrak{g}$ where the potential $W:M\rightarrow \mathbb{R}$ is sufficiently smooth. The associated infinitesimal generator is then the Laplace--Beltrami operator plus a drift term: 
\begin{equation} \label{eqn:mL-drift}
\mL\mtx{f}(x) := -\sum_{i,j=1}^n g^{ij}\,\partial_iW\, \partial_j\mtx{f} + \frac{1}{w_\mathfrak{g}} \sum_{i,j=1}^n\partial_i\left(w_\mathfrak{g}g^{ij}\,\partial_j\mtx{f}(x)\right)\quad \text{for suitable $\mtx{f} : M \to \mathbb{H}_d$.}
\end{equation}
It is not hard to check that $\mL$ is symmetric with respect to $\mu$, and hence the induced diffusion process with drift is reversible. 

\subsection{Carr{\'e} du champ operators}

Next, we present expressions for the matrix carr{\'e} du champ operators associated with the infinitesimal generator $\mL$ defined in~\eqref{eqn:mL-drift}.  The derivation follows from a standard symbol calculation, as in the scalar setting. 

\subsubsection{Carr\'{e} du champ operator} The carr{\'e} du champ operator coincides with the squared ``magnitude'' of the differential:
\begin{equation}\label{eqn:gamma_Riemannian}
\Gamma(\mtx{f}) = \sum_{i,j=1}^ng^{ij}\,\partial_i\mtx{f}\, \partial_j\mtx{f}\quad \text{for suitable $\mtx{f}:M\rightarrow\mathbb{H}_d$}.
\end{equation}
Note that this expression contains a matrix product. Calculation of the carr{\'e} du champ involves a choice of local coordinates. Nevertheless, expressions of the carr{\'e} du champ in different choices of local coordinates are equivalent under change of variables. 

Another way to calculate the carr{\'e} du champ $\Gamma(\mtx{f})$ is by relating it to the tangential gradient of $\mtx{f}$ on the manifold. For a point $x\in M$, let $T_xM$ denote the tangent space at $x$. The tangential gradient $\nabla_M\mtx{f}(x)$ of a matrix-valued function $\mtx{f}:M\rightarrow\mathbb{H}_d$ can be written as
\[\nabla_M\mtx{f}(x) = \sum_{i=1}^N \vct{v}_i\otimes \mtx{A}_i\]
for some vectors $\{\vct{v}_i\}_{i=1}^N\subset T_xM$ and some matrices $\{\mtx{A}_i\}_{i=1}^N\subset \mathbb{H}_d$ that depend on the representation of the manifold $M$. The integer $N$ is not necessarily the dimension of $M$.  When $d=1$, the tangential gradient $\nabla_M\mtx{f}(x)$ is also a vector in $T_xM$. Now, the carr{\'e} du champ at the point $x$ is given by an equivalent expression:
\begin{equation}\label{eqn:gamma_Riemannian_alternative}
\Gamma(\mtx{f})(x) = \langle \nabla_M\mtx{f}(x),\nabla_M\mtx{f}(x)\rangle_{\mtx{G}} := \sum_{i,j=1}^N\langle \vct{v}_i,\vct{v}_j\rangle_{\mtx{G}}\cdot \mtx{A}_i\mtx{A}_j
\end{equation}
where $\langle \cdot ,\cdot\rangle_{\mtx{G}}$ is the inner product on $T_xM$ associated with the metric tensor $\mtx{G}$. 

The expression \eqref{eqn:gamma_Riemannian_alternative} coincides with \eqref{eqn:gamma_Riemannian} if we choose
$(\vct{v}_i : 1 \leq i \leq n)$ to be the moving frame of $N = n$ local coordinates.  In this case, 
$\langle \vct{v}_i(x),\vct{v}_i(x)\rangle_{\mtx{G}} = g_{ij}(x)$ for $i,j=1,\dots,n$.  Moreover, the tangential gradient can be written as
\[\nabla_M\mtx{f}(x) = \sum_{i=1}^n \vct{v}_i(x)\otimes \nabla_M^i\mtx{f}(x),\]
where $\nabla_M^i\mtx{f}(x):= \sum_{j=1}^ng^{ij}\partial_j\mtx{f}$ for $i=1,\dots,n$. Then one can rewrite the expression \eqref{eqn:gamma_Riemannian_alternative} in the form \eqref{eqn:gamma_Riemannian} by recalling that $\mtx{G}=\mathfrak{g}^{-1}$.

The expression \eqref{eqn:gamma_Riemannian_alternative} is especially useful when the Riemannian manifold $M$ is embedded into a higher-dimensional Euclidean space $\mathbb{R}^N$ with the metric tensor $\mtx{G}$ induced by the Euclidean metric. That is, $M$ is a Riemannian submanifold of $\mathbb{R}^N$. In this case, for a function $\mtx{f} : \R^{N} \to \mathbb{H}_d$, the tangential gradient  $\nabla_M\mtx{f}(x)$ is simply the projection of $\nabla_{\mathbb{R}^N}\mtx{f}(x)$ onto the tangent space $T_xM$, where $\nabla_{\mathbb{R}^N}\mtx{f}$ is the ordinary gradient of $\mtx{f}$ in the embedding space $\R^N$.  Let us elaborate.  Suppose that $x = (x_1,\dots,x_N)$ is the representation of a point $x\in M$ with respect to the standard basis $\{\mathbf{e}_i\}_{i=1}^N$ of $\mathbb{R}^N$. Define the orthogonal projection $\mathrm{Proj}_x$ onto the tangent space $T_xM$. Then the tangential gradient satisfies 
\[\nabla_M\mtx{f}(x) = (\mathrm{Proj}_x \otimes \Id)\left(\sum_{i=1}^N \mathbf{e}_i\otimes \frac{\partial \mtx{f}(x)}{\partial x_i}\right) = \sum_{i=1}^N (\mathrm{Proj}_x\mathbf{e}_i)\otimes \frac{\partial \mtx{f}(x)}{\partial x_i}.\]
This expression of the tangential gradient helps simplify the calculation of the carr{\'e} du champ operator in many interesting examples.

\subsubsection{Iterated carr{\'e} du champ operator} To introduce the iterated matrix carr{\'e} du champ operator, we first define the Hessian
$\nabla^2 \mtx{f} := (\nabla^2_{ij} \mtx{f} : 1 \leq i,j \leq n)$ of a matrix-valued function $\mtx{f}:M\rightarrow \mathbb{H}_d$,
where
\[\nabla^2_{ij} \mtx{f} := \partial_{ij}\mtx{f} - \sum_{k=1}^n\gamma_{ij}^k\partial_k\mtx{f} \quad \text{for $i,j=1,\dots,n$}.\]
The Christoffel symbols $\gamma_{ij}^k$ are the quantities
\[
\gamma_{ij}^k := \frac{1}{2}\sum_{l=1}^ng^{kl}(\partial_{j}g_{il} + \partial_{i}g_{jl} - \partial_{l}g_{ij})
	\quad \text{for $i,j,k=1,2,\dots,n$}.\]
When the matrix dimension $d > 1$, the Hessian $\nabla^2 \mtx{f}$ is a 4-tensor.

Now, the iterated matrix carr{\'e} du champ operator $\Gamma_2$ admits the formula
\begin{equation} \label{eqn:gamma2-riemann}
\Gamma_2(\mtx{f}) = \sum_{i,j,k,l=1}^ng^{ij}g^{kl} \, \nabla^2_{ik}\mtx{f}\, \nabla^2_{jl}\mtx{f} + \sum_{i,j,k,l=1}^ng^{ik}g^{jl}\left(\operatorname{Ric}_{kl} + \nabla^2_{kl}W\right) \partial_i\mtx{f}\, \partial_j\mtx{f}.
\end{equation}
Again, this expression involves matrix products.
The Ricci tensor $\operatorname{\mtx{Ric}} = (\operatorname{Ric}_{ij} : 1 \leq i,j \leq n)$ is given by 
\[\operatorname{Ric}_{ij} := \sum_{k=1}^n \left(\partial_k\gamma_{ij}^k - \partial_i\gamma_{kj}^k\right) + \sum_{k,l=1}^n\left(\gamma_{kl}^k\gamma_{ij}^l - \gamma_{il}^k\gamma_{jk}^l\right).\]
The Ricci tensor expresses the curvature of the manifold.

\subsection{Bakry--\'Emery criterion}\label{sec:BE_Riemannian}
Since the first sum in the expression~\eqref{eqn:gamma2-riemann} for $\Gamma_2(\mtx{f})$ is a positive-semidefinite matrix, we have the inequality 
\begin{equation}\label{eqn:gamma_2_Riemannian}
\Gamma_2(\mtx{f}) \succcurlyeq \sum_{i,j,k,l=1}^ng^{ik}g^{jl}\left(\operatorname{Ric}_{kl} + \nabla^2_{kl}W\right) \partial_i\mtx{f}\, \partial_j\mtx{f}.
\end{equation}
In a Euclidean space, the Ricci tensor is everywhere zero, so the Bakry--\'Emery criterion~\eqref{Bakry-Emery} relies on the strong convexity of the potential $W$, as we have seen in Section~\ref{sec:log-concave}.  In contrast, on a Riemannian manifold, the Ricci tensor plays an important role.

Let us now assume that the Riemannian manifold is unweighted; that is, the potential $W = 0$ identically.  By comparing the displays \eqref{eqn:gamma_Riemannian} and \eqref{eqn:gamma_2_Riemannian} for a scalar function $f:M\to\mathbb{R}$, we can see that the scalar Bakry--\'Emery criterion holds with constant $c=\rho^{-1}$, provided that
\[
\mathfrak{g}(x)\operatorname{\mtx{Ric}}(x)\mathfrak{g}(x)\succcurlyeq \rho \mathfrak{g}(x)\quad \text{or equivalently}\quad \operatorname{\mtx{Ric}}(x) \succcurlyeq \rho \mtx{G}(x) \quad \text{for all $x\in M$}.
\]
That is, the eigenvalues of $\operatorname{\mtx{Ric}}$ relative to the metric $\mtx{G}$ are bounded from below by $\rho$. This is often referred as the curvature condition $CD(\rho,\infty)$.
Proposition~\ref{prop:BE_equiv} allows us to lift the scalar Bakry--{\'E}mery
criterion to matrix-valued functions; we can also achieve this goal by direct
argument.

We remark that the uniform positiveness of the Ricci curvature tensor also leads to a Poincar\'e inequality for the diffusion process on the manifold; see \cite[Section 4.8]{bakry2013analysis}. Therefore, proposition~\ref{prop:matrix_poincare} implies that the associated Markov semigroup is ergodic in the sense of \eqref{eqn:ergodicity}. 

As a typical example, consider the $n$-dimensional unit sphere $\mathbb{S}^n \subset \R^{n+1}$,
equipped with the induced Riemmanian structure.  The associated Riemannian measure is the
uniform distribution.  For the sphere, the Ricci curvature tensor is constant:
$\operatorname{\mtx{Ric}} = (n-1)\mtx{G}$; see \cite[Section 2.2]{bakry2013analysis}.
Therefore, the Brownian motion on $\mathbb{S}^n$ satisfies a Bakry--\'Emery criterion \eqref{Bakry-Emery} with $c = (n-1)^{-1}$ for $n\geq 2$.

Next, consider the special orthogonal group $\mathrm{SO}(n) \subset \R^{n \times n}$
with the induced Riemannian structure.  The canonical measure is the Haar probability measure.
For this manifold, it is known that the eigenvalues of the Ricci tensor are bounded below by $\rho = (n-1)/4$;
see~\cite[p.~27]{ledoux2001concentration}.
Therefore, the special orthogonal group $\mathrm{SO}(n)$ satisfies the Bakry--{\'E}mery
criterion~\eqref{Bakry-Emery} with $c = 4/(n-1)$.

There are many other Riemannian manifolds where a lower bound on the Ricci curvature is available.
We refer the reader to~\cite[Sec.~2.2.1]{ledoux2001concentration} for more examples and references.

\subsection{Calculations of carr{\'e} du champ operators}\label{sec:Riemannian_gamma}
In this section, we provide calculations of carr{\'e} du champ operators for the concrete examples in Section~\ref{sec:riemann-exp}. 

\subsubsection{Example~\ref{example:sphere_I}: Sphere I}

In this example, we consider the unit sphere $\mathbb{S}^n \subset \R^{n+1}$ as a Riemannian submanifold
of $\R^{n+1}$ for $n \geq 2$.  The canonical Riemannian measure is the uniform probability measure $\sigma_n$
on the sphere.

Let $(\mtx{A}_1, \dots, \mtx{A}_{n+1}) \subset \mathbb{H}_d$ be a fixed collection of Hermitian matrices.
Draw a random vector $\vct{x} = (x_1, \dots, x_{n+1}) \in \mathbb{S}^n$ from the uniform measure; we use
boldface to emphasize that $\vct{x}$ is a vector in the embedding space.  Consider the matrix-valued function
$$
\mtx{f}(\vct{x}) = \sum_{i=1}^{n+1} x_i \mtx{A}_i.
$$
We can use the expression~\eqref{eqn:gamma_Riemannian_alternative} to compute the carr{\'e} du champ
of $\mtx{f}$.

Indeed, the ordinary gradient of $\mtx{f}$ as a function on $\mathbb{R}^{n+1}$ is given by 
\[\nabla_{\mathbb{R}^{n+1}} \mtx{f}(\vct{x}) = \sum_{i=1}^{n+1}\mathbf{e}_i\otimes \frac{\partial \mtx{f}(\vct{x})}{\partial x_i} = \sum_{i=1}^{n+1}\mathbf{e}_i\otimes \mtx{A}_i\quad \text{for all $\vct{x}\in \mathbb{R}^{n+1}$}.\]
As usual, $\{\mathbf{e}_i\}_{i=1}^{n+1}$ is the standard basis of $\mathbb{R}^{n+1}$.
Define the orthogonal projection $\mathrm{Proj}_{\vct{x}} = \Id - \vct{x}\vct{x}^\trsp$ onto the tangent space $T_{\vct{x}}\mathbb{S}^n = \{\vct{y}\in \mathbb{R}^{n+1}: \vct{y}^\trsp\vct{x}=0 \}$. 
Thus, the tangential gradient is the projection of the ordinary gradient onto the tangent space:
\[\nabla_{\mathbb{S}^n}\mtx{f}(\vct{x}) = (\mathrm{Proj}_{\vct{x}} \otimes \Id)\nabla_{\mathbb{R}^{n+1}} \mtx{f}(\vct{x}) = \sum_{i=1}^{n+1}(\mathbf{e}_i - x_i\vct{x})\otimes \mtx{A}_i.\]
By the expression \eqref{eqn:gamma_Riemannian_alternative}, we can compute the carr{\'e} du champ at each point $\vct{x}\in \mathbb{S}^{n}$ as
\begin{align*}
\Gamma(\mtx{f})(\vct{x}) &= \sum_{i,j=1}^{n+1}(\mathbf{e}_i - x_i\vct{x})^\trsp(\mathbf{e}_j - x_j\vct{x})\cdot \mtx{A}_i\mtx{A}_j = \sum_{i,j=1}^{n+1}(\delta_{ij} - x_ix_j)\,\mtx{A}_i\mtx{A}_j \\
&= \sum_{i=1}^{n+1}\mtx{A}_i^2 - \sum_{i,j=1}^{n+1}x_ix_j\,\mtx{A}_i\mtx{A}_j = \sum_{i=1}^{n+1}\mtx{A}_i^2 - \left(\sum_{i=1}^{n+1}x_i\mtx{A}_i\right)^2.
\end{align*}
This calculation verifies the formula \eqref{eqn:gamma_sphere_I}.
It is now evident that
$$
\mtx{0} \psdle \Gamma(\mtx{f})(\vct{x}) \psdle \sum_{i=1}^{n+1}\mtx{A}_i^2 \quad\text{for all $\vct{x} \in \mathbb{S}^n$.}
$$
Therefore, the variance proxy $v_{\mtx{f}} \leq \norm{ \sum_{i=1}^{n+1} \mtx{A}_i^2 }$.

\subsubsection{Example~\ref{example:sphere_II}: Sphere II}

We maintain the setup and notation from the last subsection, and we consider the matrix-valued function
$$
\mtx{f}(\vct{x}) = \sum_{i=1}^{n+1}x_i^2\mtx{A}_i
\quad\text{where $\vct{x} \sim \sigma_n$ on $\mathbb{S}^n$.}
$$
Treating $\mtx{f}$ as a function on the embedding space $\mathbb{R}^{n+1}$,
the ordinary gradient is given by 
\[\nabla_{\mathbb{R}^{n+1}} \mtx{f}(\vct{x}) = \sum_{i=1}^{n+1}\mathbf{e}_i\otimes \frac{\partial \mtx{f}(\vct{x})}{\partial x_i} = 2\sum_{i=1}^{n+1}x_i\mathbf{e}_i\otimes \mtx{A}_i\quad \text{for all $\vct{x}\in \mathbb{R}^{n+1}$}.\]
Thus, the tangential gradient of $\mtx{f}$ at a point $\vct{x}\in \mathbb{S}^{n}$ can be computed as  
\[\nabla_{\mathbb{S}^n}\mtx{f}(\vct{x}) = (\mathrm{Proj}_{\vct{x}} \otimes \Id)\nabla_{\mathbb{R}^{n+1}} \mtx{f}(\vct{x}) = 2\sum_{i=1}^{n+1}(x_i\mathbf{e}_i - x_i^2\vct{x})\otimes \mtx{A}_i.\]
By the expression \eqref{eqn:gamma_Riemannian_alternative} of the carr{\'e} du champ operator, we can compute that
\begin{align*}
\Gamma(\mtx{f})(\vct{x}) &= 4\sum_{i,j=1}^{n+1}(x_i\mathbf{e}_i - x_i^2\vct{x})^\trsp(x_j\mathbf{e}_j - x_j^2\vct{x})\cdot \mtx{A}_i\mtx{A}_j = 4\sum_{i=1}^{n+1}x_i^2\mtx{A}_i^2 - 4\sum_{i,j=1}^{n+1}x_i^2x_j^2\,\mtx{A}_i\mtx{A}_j\\
&= 4\sum_{i,j=1}^{n+1}x_i^2x_j^2\mtx{A}_i^2 - 4\sum_{i,j=1}^{n+1}x_i^2x_j^2\,\mtx{A}_i\mtx{A}_j = 2\sum_{i,j=1}^{n+1}x_i^2x_j^2(\mtx{A}_i - \mtx{A}_j)^2.
\end{align*}
This establishes the formula \eqref{eqn:gamma_sphere_II}.

Using this result, we can obtain some bounds for the variance proxy.
First, introduce the maximum norm difference $a:= \max_{i,j}\norm{\smash{\mtx{A}_i-\mtx{A}_j}}$.
Then the carr{\'e} du champ satisfies
\[\Gamma(\mtx{f})(\vct{x}) \preccurlyeq 2\sum_{i,j=1}^{n+1}x_i^2x_j^2\norm{\smash{\mtx{A}_i-\mtx{A}_j}}^2\cdot \Id_d \preccurlyeq 2a^2\sum_{i,j=1}^{n+1}x_i^2x_j^2\cdot \Id_d  = 2a^2\Id_d.\]
Thus, $v_{\mtx{f}} \leq 2 a^2$.
Here is an alternative approach.  For an arbitrary matrix $\mtx{B} \in \mathbb{H}_d$,
we can write
\begin{align*}
\Gamma(\mtx{f})(\vct{x}) &= 2\sum_{i,j=1}^{n+1}x_i^2x_j^2(\mtx{A}_i - \mtx{B} + \mtx{B} - \mtx{A}_j)^2\\
&= 4\sum_{i=1}^{n+1}x_i^2(\mtx{A}_i - \mtx{B})^2 - 4\left(\sum_{i=1}^{n+1}x_i^2\mtx{A}_i - \mtx{B}\right)^2\\
&\preccurlyeq 4\sum_{i=1}^{n+1}x_i^2(\mtx{A}_i - \mtx{B})^2
\end{align*}
Defining $b:=\min_{\mtx{B}\in \mathbb{H}_d} \max_i \norm{\mtx{A}_i-\mtx{B}}$,
we see that the variance proxy $v_{\mtx{f}} \leq 4 b^2$.
Modulo an extra factor of two, the second bound represents
a qualitative improvement over the first.

\subsubsection{Example~\ref{example:SO_d}: Special orthogonal group}

Let $(\mtx{A}_1, \dots, \mtx{A}_n) \subset \mathbb{H}_d(\R)$ be fixed, real, symmetric matrices.
Draw $(\mtx{O}_1, \dots, \mtx{O}_n) \subset \mathrm{SO}(d)$ independent and uniformly from
the Haar measure on the special orthogonal group $\mathrm{SO}(d)$.  Consider the random matrix
$$
\mtx{f}(\mtx{O}_1, \dots, \mtx{O}_n) = \sum_{i=1}^n \mtx{O}_i \mtx{A}_i \mtx{O}_i^\trsp.
$$
To study this random matrix model, we will use local geodesic/normal coordinates
on the product manifold $\mathrm{SO}(d)^{\otimes n}$ to compute the carr{\'e} du champ;
for example, see~\cite[Sec. 5]{lee2018introduction} \& \cite[Sec. 3]{hall2015lie}.
Since $\mathrm{SO}(d)^{\otimes n}$ is a Lie group, we only need to consider the
geodesic frame of the tangent space at the identity element $(\Id_d, \dots, \Id_d)$.

For each $1\leq k<l\leq d$, let $\mtx{S}_{kl}\in \mathbb{M}_d$ be the unit skew-symmetric matrices: 
\[(\mtx{S}_{kl})_{kl} = 1/\sqrt{2}\quad\text{and}\quad (\mtx{S}_{kl})_{lk} = -1/\sqrt{2}\quad \text{and other entries of $\mtx{S}_{kl}$ are zero}.\]
Define the tangent vectors 
\[\mtx{V}_{kl}^i = \underbrace{(\mtx{0},\dots,\mtx{S}_{kl},\dots,\mtx{0})}_{\text{The $i$th coordinate is $\mtx{S}_{kl}$}} \quad \text{for $i=1,\dots,n$ and $1\leq k<l\leq d$}.\]
Then $( \mtx{V}_{kl}^i : 1\leq i\leq n \text{ and } 1\leq k<l\leq d )$ forms an orthonormal basis for the tangent space at the identity element of the Lie group $\mathrm{SO}(d)^{\otimes n}$, with respect to the Hilbert--Schmidt inner product:
\[\langle (\mtx{P}_1,\dots,\mtx{P}_n), (\mtx{Q}_1,\dots,\mtx{Q}_n)\rangle_\mathrm{HS} = \sum_{i=1}^n \trace[\mtx{P}_i^*\mtx{Q}_i]\quad \text{for $\mtx{P}_1,\dots,\mtx{P}_n,\mtx{Q}_1,\dots,\mtx{Q}_n\in\mathbb{M}_d$}.\]
This basis $\{\mtx{V}_{kl}^i\}_{1\leq i\leq n,1\leq k<l\leq d}$ can be translated to an orthonormal basis of the tangent space at another point $(\mtx{O}_1, \dots, \mtx{O}_n)$ by the group operation: $(\mtx{0},\dots,\mtx{S}_{kl},\dots,\mtx{0})\mapsto (\mtx{0},\dots,\mtx{S}_{kl}\mtx{O}_i,\dots,\mtx{0})$.

Now, for each $(\mtx{O}_1, \dots, \mtx{O}_n)\in\mathrm{SO}(d)^{\otimes n}$, consider the local geodesic map corresponding to the direction $\mtx{V}_{kl}^i$:
\[(\mtx{O}_1, \dots, \mtx{O}_i, \dots, \mtx{O}_n) \mapsto (\mtx{O}_1, \dots, \econst^{\varepsilon\mtx{S}_{kl}}\mtx{O}_i, \dots, \mtx{O}_n)\quad \text{for some small $\varepsilon\geq0$}.\] 
Then the directional derivative of $\mtx{f}$ in local geodesic coordinates, evaluated at the point $(\mtx{O}_1,\dots,\mtx{O}_n)$ where $\eps = 0$, is given by 
\[\frac{\partial \mtx{f}}{ \partial \mtx{V}_{kl}^i}(\mtx{O}_1, \dots, \mtx{O}_n) = \mtx{S}_{kl}\mtx{O}_i\mtx{A}_i\mtx{O}_i^\trsp - \mtx{O}_i\mtx{A}_i\mtx{O}_i^\trsp \mtx{S}_{kl} =: \mtx{S}_{kl}\mtx{B}_i - \mtx{B}_i\mtx{S}_{kl},\]
where $\mtx{B}_i := \mtx{O}_i\mtx{A}_i\mtx{O}_i^\trsp$.  In local geodesic coordinates, the co-metric tensor $\mathfrak{g}$ at the origin equals the identity.  Using the formula \eqref{eqn:gamma_Riemannian}, we can compute the carr{\'e} du champ as   
\begin{align*}
\Gamma(\mtx{f})(\mtx{O}_1, \dots, \mtx{O}_n) &=\sum_{i=1}^n\sum_{1\leq k<l\leq b} \left(\frac{\partial \mtx{f}}{ \partial \mtx{V}_{kl}^i}\right)^2 = \sum_{i=1}^n\sum_{1\leq k<l\leq b}\left(\mtx{S}_{kl}\mtx{B}_i - \mtx{B}_i\mtx{S}_{kl}\right)^2\\
&= \sum_{i=1}^n\sum_{1\leq k<l\leq b}\left(-\mtx{S}_{kl}\mtx{B}_i^2\mtx{S}_{kl} - \mtx{B}_i\mtx{S}_{kl}^2\mtx{B}_i + \mtx{S}_{kl}\mtx{B}_i\mtx{S}_{kl}\mtx{B}_i + \mtx{B}_i\mtx{S}_{kl}\mtx{B}_i\mtx{S}_{kl}\right).
\end{align*}
It is not hard to check that, for any real matrix $\mtx{M}\in \mathbb{M}_d(\mathbb{R})$, 
\[\sum_{1\leq k<l\leq b} \mtx{S}_{kl}\mtx{M}\mtx{S}_{kl} =  -\frac{1}{2}(\trace[\mtx{M}] \cdot \Id_d - \mtx{M}^\trsp).\]
Therefore, we can obtain that 
\begin{align*}
\Gamma(\mtx{f})(\mtx{O}_1, \dots, \mtx{O}_n) &= \frac{1}{2}\sum_{i=1}^n\left( \trace[\mtx{B}_i^2] \cdot \Id_d - \mtx{B}_i^2 + (d-1)\mtx{B}_i^2 - (\trace[\mtx{B}_i]\cdot \Id_d - \mtx{B}_i)\mtx{B}_i - \mtx{B}_i(\trace[\mtx{B}_i]\cdot \Id_d - \mtx{B}_i) \right) \\
&= \frac{1}{2}\sum_{i=1}^n\left( \trace[\mtx{B}_i^2]\cdot \Id_d + d\cdot \mtx{B}_i^2 - 2\trace[\mtx{B}_i]\cdot \mtx{B}_i\right)\\
&= \frac{1}{2}\sum_{i=1}^n\mtx{O}_i\left( \trace[\mtx{A}_i^2]\cdot \Id_d + d\cdot \mtx{A}_i^2 - 2\trace[\mtx{A}_i]\cdot \mtx{A}_i\right)\mtx{O}_i^\trsp.\\
&= \frac{1}{2}\sum_{i=1}^n\mtx{O}_i\left\{ \left(\trace[\mtx{A}_i^2]-\frac{\trace[\mtx{A}_i]^2}{d}\right)\cdot \Id_d + d\cdot \left(\mtx{A}_i - \frac{\trace[\mtx{A}_i]}{d}\cdot \Id_d \right)^2\right\}\mtx{O}_i^\trsp.
\end{align*}
This justifies the formula \eqref{eqn:gamma_SO_d}. Since each $\mtx{O}_i$ is an orthogonal matrix,
the variance proxy satisfies
\begin{align*}
v_{\mtx{f}} &= 
\max\nolimits_{\mtx{O}_i} \norm{\Gamma(\mtx{f})(\mtx{O}_1, \dots, \mtx{O}_n)} \\
&\leq \frac{1}{2}\sum_{i=1}^n\norm{ \left(\trace[\mtx{A}_i^2]-d^{-1}\trace[\mtx{A}_i]^2\right)\cdot \Id_d + d\cdot \left(\mtx{A}_i - d^{-1}\trace[\mtx{A}_i]\cdot \Id_d \right)^2 }\\
&= \frac{1}{2}\sum_{i=1}^n \left( \trace[\mtx{A}_i^2]-d^{-1}\trace[\mtx{A}_i]^2 + d\cdot \norm{\mtx{A}_i - d^{-1}\trace[\mtx{A}_i]\cdot \Id_d  }^2 \right).
\end{align*}
Note that this bound is sharp because we can always choose some particular point $(\mtx{O}_1, \dots, \mtx{O}_n)$ to achieve equality. 
\subsection{Matrix concentration results}\label{sec:concentration_results_Riemannian}
At last, we provide a proof of Theorem~\ref{thm:riemann-simple} from Theorem~\ref{thm:polynomial_moment} and Theorem~\ref{thm:exponential_concentration}.

Consider a compact $n$-dimensional Riemannian submanifold $M$ of a Euclidean space. The uniform measure $\mu$ on $M$ is the stationary measure of the associated Brownian motion on $M$. As discussed in Section~\ref{sec:BE_Riemannian}, the Brownian motion satisfies a Bakry--\'Emery criterion with constant $c=\rho^{-1}$ if the eigenvalues of the Ricci curvature tensor are bounded below by $\rho$. We then apply Theorem~\ref{thm:polynomial_moment} and Theorem~\ref{thm:exponential_concentration} with $c=\rho^{-1}$ to obtain the matrix concentration inequalities in Theorem~\ref{thm:riemann-simple}. 

For any point $x\in M$, we can compute the carr\'e du champ $\Gamma(\mtx{f})(x)$ in local normal coordinates centered at $x$. In this case, the co-metric tensor $\mathfrak{g}$ is the identity matrix $\Id_n$ when evaluated at $x$. The expression of the variance proxy $v_{\mtx{f}}$ in Theorem~\ref{thm:riemann-simple} then follows from formula \eqref{eqn:gamma_Riemannian} of the carr\'e du champ operator.

\appendix

\section{Matrix moments and concentration}
\label{apdx:matrix_moments}

For reference, this appendix summarizes a few standard results on matrix moments and concentration. Proposition~\ref{prop:matrix_Chebyshev} explains how to transfer the polynomial moments bounds in Theorem~\ref{thm:polynomial_moment} into matrix concentration inequalities. Proposition~\ref{prop:trace_mgf} states some properties of the trace mgf that are used in the proof of Theorem~\ref{thm:exponential_moment}. Proposition~\ref{prop:matrix_exponential_concentration} allows us to derive the exponential concentration inequalities in Theorem~\ref{thm:exponential_concentration} from the exponential moment bounds in Theorem~\ref{thm:exponential_moment}.

\subsection{The matrix Chebyshev inequality}
We can obtain concentration inequalities
for a random matrix given bounds for the polynomial trace moments.
This result extends Chebyshev's probability inequality.  For instance, see \cite[Proposition 6.2]{mackey2014}.

\begin{proposition}[Matrix Chebyshev inequality]\label{prop:matrix_Chebyshev}
Let $\mtx{X}\in\mathbb{H}_d$ be a random matrix. For all $t\geq 0$, 
\[\Prob{\|\mtx{X}\|\geq t}\leq \inf_{q\geq 1}t^{-q}\cdot \E\trace|\mtx{X}|^q.\]
Furthermore, 
\[\E\|\mtx{X}\|\leq  \inf_{q\geq 1}\left(\E\trace|\mtx{X}|^q\right)^{1/q}.\]
\end{proposition}

As mentioned in Section~\ref{sec:main_results}, Proposition~\ref{prop:matrix_Chebyshev} can be applied to the polynomial moment bounds in Theorem~\ref{thm:polynomial_moment} to yield subgaussian concentration inequalities. 

\subsection{The matrix Laplace transform method} We can also obtain exponential concentration inequalities via the matrix Laplace transform. Let $\mtx{X}\in\mathbb{H}_d$ be a random matrix. The normalized trace moment generating function (mgf) of $\mtx{X}$ is defined as 
\[m(\theta):= \E\ntr \econst^{\theta\mtx{X}},\quad \text{for}\ \theta\in \mathbb{R}.\]
This definition is due to Ahlswede and Winter \cite{ahlswede2002strong}. In the proof of Theorem~\ref{thm:exponential_moment}, we have used some properties of the trace mgf given in the following proposition, which restates \cite[Lemma 12.3]{paulin2016efron}. 

\begin{proposition}[Properties of the trace mgf]\label{prop:trace_mgf} Assume that $\mtx{X}\in\mathbb{H}_d$ is a zero-mean random matrix that is bounded in norm. Define the normalized trace mgf
$m(\theta) := \E \ntr \econst^{\theta\mtx{X}}$ for $\theta\in\mathbb{R}$. Then  
\begin{equation}\label{eqn:m.g.f_Property_1}
\log m(\theta) \geq 0 \quad \text{and} \quad \log m(0) = 0.
\end{equation}
The derivative of the trace mgf satisfies
\begin{equation*}\label{eqn:m.g.f_Property_2}
m'(\theta) = \E \ntr\left[\mtx{X}\,\econst^{\theta\mtx{X}}\right] \quad \text{and} \quad m'(0) = 0.
\end{equation*}
The trace mgf is a convex function; in particular 
\begin{equation*}\label{eqn:m.g.f_Property_3}
m'(\theta) \leq 0\quad \text{for}\quad \theta\leq 0 \quad \text{and} \quad m'(\theta) \geq 0\quad \text{for}\quad \theta\geq 0.
\end{equation*}
\end{proposition}

Using the matrix Laplace transform method, one can convert estimates on the trace mgf into bounds on the extreme eigenvalues of a random matrix.   For example, see~\cite[Proposition 3.3]{mackey2014}. In particular, having an explicit bound on the trace mgf, we can obtain concrete estimates on the maximum eigenvalue. See \cite[Section 4.2.4]{mackey2014} for a proof. 

\begin{proposition}\label{prop:matrix_exponential_concentration} Let $\mtx{X}\in\mathbb{H}_d$ be a random matrix with normalized trace mgf $m(\theta):= \E\ntr \econst^{\theta\mtx{X}}$. Assume that there are constants $c_1,c_2 \geq 0$ for which 
\[\log m(\theta) \leq \frac{c_1\theta^2}{2(1-c_2\theta)}\quad \text{when}\ 0\leq \theta<\frac{1}{c_2}.\]
Then for all $t\geq0$, 
\[\Prob{\lambda_{\max}(\mtx{X})\geq t} \leq d\cdot \exp\left(\frac{-t^2}{2c_1+2c_2t}\right).\]
Furthermore, 
\[\E\lambda_{\max}(\mtx{X})\leq \sqrt{2c_1\log d} + c_2\log d.\] 
\end{proposition}

We have applied Proposition~\ref{prop:matrix_exponential_concentration} to the trace mgf bounds in Theorem~\ref{thm:exponential_moment} to derive exponential concentration inequalities as those in Theorem~\ref{thm:exponential_concentration}.

\section{Mean value trace inequality}
\label{apdx:mean_value}

In this section, we establish the mean value trace inequality, Lemma~\ref{lem:mean_value_inequality}.
This result is a generalization of~\cite[Lemmas 9.2 and 12.2]{paulin2016efron}.
The proof is similar in spirit, but it uses some additional ingredients from matrix analysis.

The key idea is to use tensorization to lift a pair of noncommuting matrices
to a pair of commuting tensors.  This step gives us access to tools that are
not available for general matrices.
For any two Hermitian matrices $\mtx{X},\mtx{Y}\in \mathbb{H}_d$,
define a linear operator $\mtx{X}\otimes \mtx{Y} : \mathbb{M}_d \to \mathbb{M}_d$ whose action is given by 
\[(\mtx{X}\otimes \mtx{Y})(\mtx{Z}) = \mtx{X}\mtx{Z}\mtx{Y}\quad \text{for all}\ \mtx{Z}\in \mathbb{M}_d.\]
The linear operator $\mtx{X}\otimes \mtx{Y}$ is self-adjoint with respect
to the standard inner product on $\mathbb{M}_d$: 
\[\langle(\mtx{X}\otimes \mtx{Y})(\mtx{Z}_1),\mtx{Z}_2\rangle_{\mathbb{M}_d} = \trace\left[\mtx{Y}\mtx{Z}_1^*\mtx{X}\mtx{Z}_2\right] = \trace\left[\mtx{Z}_1^*\mtx{X}\mtx{Z}_2\mtx{Y}\right] = \langle\mtx{Z}_1,(\mtx{X}\otimes \mtx{Y})(\mtx{Z}_2)\rangle_{\mathbb{M}_d}\quad \text{for all}\ \mtx{Z}_1,\mtx{Z}_2\in \mathbb{M}_d.\]
Therefore, for any function $\varphi:\mathbb{R}\rightarrow \mathbb{R}$, we can define
the tensor function $\varphi(\mtx{X}\otimes \mtx{Y})$ using the spectral resolution
of $\mtx{X} \otimes \mtx{Y}$.  It is not hard to check that 
\[\varphi(\mtx{X}\otimes \Id) = \varphi(\mtx{X})\otimes \Id\quad \text{and}\quad \varphi(\Id\otimes \mtx{Y}) = \Id\otimes \varphi(\mtx{Y}).\]
Note that the tensors $\mtx{X}\otimes\Id$ and $\Id\otimes\mtx{Y}$ commute with each other,
regardless of whether $\mtx{X}$ and $\mtx{Y}$ commute.

\begin{proof}[Proof of Lemma~\ref{lem:mean_value_inequality}]
We can write 
\begin{align*}
\varphi(\mtx{A})-\varphi(\mtx{B}) &= \big(\varphi(\mtx{A})\otimes \Id -\Id\otimes \varphi(\mtx{B})\big) (\Id)\\
&=  \big(\varphi(\mtx{A}\otimes \Id) -\varphi(\Id\otimes \mtx{B})\big) (\Id) 
= \int_0^1\frac{\diff{} }{\diff \tau}\varphi\big(\tau \mtx{A}\otimes \Id + (1-\tau) \Id\otimes \mtx{B}\big) (\Id) \idiff \tau.
\end{align*}
Since $\mtx{A}\otimes \Id$ commutes with $\Id \otimes \mtx{B}$, we have 
\[\frac{\diff{} }{\diff \tau}\varphi\big(\tau \mtx{A}\otimes \Id + (1-\tau) \Id\otimes \mtx{B}\big)  = \varphi'\big(\tau \mtx{A}\otimes \Id + (1-\tau) \Id\otimes \mtx{B}\big) (\mtx{A}\otimes \Id- \Id\otimes \mtx{B}).\]
As a consequence, 
\begin{align*}
\varphi(\mtx{A})-\varphi(\mtx{B}) &= \int_0^1\varphi'\big(\tau \mtx{A}\otimes \Id + (1-\tau) \Id\otimes \mtx{B}\big) (\mtx{A}\otimes \Id- \Id\otimes \mtx{B})(\Id) \idiff \tau \\
&= \int_0^1\varphi'\big(\tau \mtx{A}\otimes \Id + (1-\tau) \Id\otimes \mtx{B}\big)(\mtx{A}- \mtx{B}) \idiff \tau
=: \int_0^1\mathcal{M}_{\tau}(\mtx{A},\mtx{B})(\mtx{A}- \mtx{B})\idiff \tau.
\end{align*}
Since $\mathcal{M}_{\tau}(\mtx{A},\mtx{B})$ is a self-adjoint linear operator on the Hilbert space $\mathbb{M}_d$, we can apply the operator Cauchy--Schwarz inequality~\cite[Lemma A.2]{paulin2016efron}.  For any $s>0$,
\begin{multline}\label{step:mean_value_1}
\trace\left[\mtx{C} \,\big(\varphi(\mtx{A})-\varphi(\mtx{B})\big)\right] = \ip{ \mtx{C} }{ \varphi(\mtx{A})-\varphi(\mtx{B})}_{\mathbb{M}_d}
= \int_0^1 \ip{ \mtx{C} }{ \mathcal{M}_{\tau}(\mtx{A},\mtx{B})(\mtx{A}- \mtx{B}) }_{\mathbb{M}_d} \idiff \tau \\
\leq \int_0^1\left[ \frac{s}{2} \ip{ \mtx{A}-\mtx{B} }{ \abs{\mathcal{M}_{\tau}(\mtx{A},\mtx{B})}(\mtx{A}-\mtx{B}) }_{\mathbb{M}_d} + \frac{s^{-1}}{2} \ip{ \mtx{C} }{ \abs{\mathcal{M}_{\tau}(\mtx{A},\mtx{B})}(\mtx{C})}_{\mathbb{M}_d}\right] \idiff \tau.
\end{multline}
By assumption, $\psi := \abs{\varphi'}$ is convex.  Thus, for all $\tau\in[0,1]$,
\begin{align*}
\abs{\mathcal{M}_{\tau}(\mtx{A},\mtx{B})} &= \abs{ \varphi'\big(\tau \mtx{A}\otimes \Id + (1-\tau) \Id\otimes \mtx{B}\big)}
= \psi\big(\tau \mtx{A}\otimes \Id + (1-\tau) \Id\otimes \mtx{B}\big)\\
&\preccurlyeq \tau \cdot \psi\left(\mtx{A}\otimes \Id\right) + (1-\tau)\cdot \psi\left(\Id\otimes \mtx{B}\right)
= \tau\cdot \psi(\mtx{A})\otimes \Id + (1-\tau)\cdot \Id\otimes \psi(\mtx{B}).
\end{align*}
The argument above depends on the commutativity of $\mtx{A}\otimes \Id$ and $\Id\otimes \mtx{B}$, which means that we do not need $\psi$ to be operator convex.  Hence, for any $\mtx{Z}\in\mathbb{M}_d$, 
\begin{multline}\label{step:mean_value_2}
\int_0^1\ip{ \mtx{Z} }{ \abs{\mathcal{M}_{\tau}(\mtx{A},\mtx{B})}(\mtx{Z}) }_{\mathbb{M}_d} \idiff \tau \leq \int_0^1  \ip{ \mtx{Z} }{ \big(\tau\cdot \psi(\mtx{A})\otimes \Id + (1-\tau)\cdot \Id\otimes \psi(\mtx{B})\big)(\mtx{Z})}_{\mathbb{M}_d} \idiff \tau\\
= \frac{1}{2}\big( \ip{ \mtx{Z} }{ \psi(\mtx{A}) \,\mtx{Z} }_{\mathbb{M}_d} + \ip{ \mtx{Z} }{ \mtx{Z}\,\psi(\mtx{B}) }_{\mathbb{M}_d} \big)
= \frac{1}{2}\big(\trace\left[\mtx{Z}\mtx{Z}^*\, \psi(\mtx{A}) \right] + \trace\left[\mtx{Z}^*\mtx{Z} \,\psi(\mtx{B}) \right]\big).
\end{multline}
Applying \eqref{step:mean_value_2} to \eqref{step:mean_value_1}, substituting $\mtx{A}-\mtx{B}$ and $\mtx{C}$ for $\mtx{Z}$,
we arrive at 
\[\trace\left[\mtx{C} \, (\varphi(\mtx{A})-\varphi(\mtx{B}))\right]\leq \frac{1}{4} \trace\left[\left(s\,(\mtx{A}-\mtx{B})^2+ s^{-1}\,\mtx{C}^2\right)\big(\psi(\mtx{A})  + \psi(\mtx{B}) \big) \right].\]
Optimize over $s>0$ to achieve the stated result. \end{proof}

\section{Connection with Stein's method}
\label{apdx:Stein_method}

There is an established approach to proving
matrix concentration inequalities using
the method of exchangeable pairs;
see~\cite{chatterjee2005concentration} for the scalar setting
and \cite{mackey2014,paulin2016efron}
for matrix extensions.
As mentioned in Section~\ref{sec:concentration_history}, the approach in \cite[Sections 10--11]{paulin2016efron}
implicitly relies on a discrete version of the local ergodicity condition.
A limiting version of this argument can also be used to derive the results in
our paper.  This appendix details the connection.

Given a reversible, exponentially ergodic Markov process $(Z_t)_{t\geq 0}$ with a stationary measure $\mu$, one can construct an exchangeable pair as follows.  Fix a time $t > 0$.  Let $Z$ be drawn from the measure $\mu$, and let $\tilde{Z} = Z_t$ where $Z_0=Z$.  By reversibility, it is easy to check that $(Z,\tilde{Z})$ is an exchangeable pair; that is, $(Z,\tilde{Z})$ has the same distribution as $(\tilde{Z},Z)$.

For a zero-mean function $\mtx{f}:\Omega\rightarrow \mathbb{H}_d$,
define the function $\mtx{g}_t: \Omega\rightarrow \mathbb{H}_d$ by 
\[\mtx{g}_t = \left(\frac{P_0-P_t}{t}\right)^{-1}\mtx{f} = t\sum_{k=0}^\infty P_{kt}\mtx{f}.\] 
Then $(\mtx{f}(Z),\mtx{f}(\tilde{Z}))$ is a \emph{kernel Stein pair} associated with the kernel 
\[\mtx{K}_t(z,\tilde{z}) = \frac{\mtx{g}_t(z) - \mtx{g}_t(\tilde{z})}{t}\quad\text{for all $z,\tilde{z}\in \Omega$.} \]
By construction, for all $z, \tilde{z} \in \Omega$,
\begin{gather}\label{eqn:kernel_property_1}
\mtx{K}_t(z,\tilde{z}) = -\mtx{K}_t(\tilde{z},z); \\ \label{eqn:kernel_property_2}
\E\left[\mtx{K}_t(Z,\tilde{Z})\,|\,Z = z\right] = \mtx{f}(z).
\end{gather}
This construction is inspired by Stein's work~\cite{stein1986approximate};
see Chatterjee's PhD thesis~\cite[Section 4.1]{chatterjee2005concentration}.
One consequence of the properties \eqref{eqn:kernel_property_1} and \eqref{eqn:kernel_property_2} is the identity
\begin{equation}\label{eqn:kernel_identity}
\E[\mtx{f}(Z) \, \varphi(\mtx{f}(Z))] = \frac{1}{2}\E\left[\mtx{K}_t(Z,\tilde{Z})\left(\varphi(\mtx{f}(Z))-\varphi(\mtx{f}(\tilde{Z}))\right)\right],
\end{equation}
which holds for any measurable function $\varphi:\mathbb{H}_d\rightarrow\mathbb{H}_d$ that satisfies the regularity condition $\|\mtx{K}_t(Z,\tilde{Z})\,\varphi(\mtx{f}(Z))\|<+\infty$ almost surely.
Paulin et al.~\cite{paulin2016efron} use~\eqref{eqn:kernel_identity}
to establish matrix Efron--Stein inequalities, much in the same
way that we derive Theorem~\ref{thm:polynomial_moment}
and Theorem~\ref{thm:exponential_moment}.

The approach we undertake in this paper is not exactly parallel with the approach in Paulin et al.~\cite{paulin2016efron}. Let us elaborate.  Take the limit of $\mtx{g}_t$ as $t \downarrow 0$, using $\mL = \lim_{t\downarrow0}(P_t-P_0)/t$. We get 
\begin{equation}\label{eqn:inverse_L}
\mtx{g}_0 = (-\mL)^{-1}\mtx{f} = \int_{0}^\infty P_t \mtx{f} \idiff t.
\end{equation}
Indeed, by ergodicity, one can check that 
\[\mtx{f} = \mtx{f}-\E_{\mu}\mtx{f} =  P_0\mtx{f}-P_\infty\mtx{f} = -\int_{0}^\infty \frac{\diff{} }{\diff t}P_t\mtx{f} \idiff t =-\mL \int_{0}^\infty P_t\mtx{f} \idiff t = -\mL \mtx{g}_0\quad \text{in $L_2(\mu)$}.\]
Consequently, we have
\begin{equation}\label{eqn:limit_kernel_identity}
\E_{\mu} [\mtx{f}\,\varphi(\mtx{f})] = -\E_{\mu} \left[\mL(\mtx{g}_0)\, \varphi(\mtx{f})\right] = \E_{\mu} \Gamma(\mtx{g}_0,\varphi(\mtx{f})).
\end{equation}
The identity \eqref{eqn:kernel_identity} is just a discrete version of the formula~\eqref{eqn:limit_kernel_identity}.
In contrast, the argument in this paper is based on the identity
\[
\E_{\mu} [\mtx{f}\,\varphi(\mtx{f})]
	= \int_0^{\infty} \Expect_{\mu} \Gamma( P_t \mtx{f}, \phi(\mtx{f}) ) \idiff{t}.
\]
The integral is not in the same place!
Our approach is technically a bit simpler because it does not require
us to justify the convergence of the integral~\eqref{eqn:inverse_L}.
Nevertheless, our work is strongly inspired by the tools and techniques
developed by Paulin et al.~\cite{paulin2016efron} in the discrete setting.

\section*{Acknowledgments}

We thank Ramon van Handel for his feedback on an early version of this manuscript.
He is responsible for the observation and proof that matrix Poincar{\'e} inequalities
are equivalent with scalar Poincar{\'e} inequalities, and we are grateful to him
for allowing us to incorporate these ideas.

DH was funded by NSF grants DMS-1907977 and DMS-1912654.
JAT gratefully acknowledges funding from ONR awards N00014-17-12146 and N00014-18-12363,
and he would like to thank his family for their support in these difficult times.

\def\bibfont{\footnotesize}

\bibliographystyle{myalpha}
\newcommand{\etalchar}[1]{$^{#1}$}

\end{document}